\documentclass[11 pt]{amsart}
\usepackage{amsmath, todonotes, amsthm, amssymb,enumerate,verbatim,mathrsfs}
\usepackage{tikz}
\usepackage{transparent}
\usetikzlibrary{calc,shapes}

\usepackage[multiple]{footmisc}

\makeatletter
\renewcommand*\@fnsymbol[1]{\the#1}  %% change footnotes to numeric superscripts
\makeatother
\usepackage[a4paper, total={7in, 8.5in}]{geometry}

\title{On the structure of graphs excluding $K_{4}$, $W_{4}$, $K_{2,4}$ and one other graph as a rooted minor}
\author[Moore]{Benjamin Moore}\thanks{The author thanks NSERC for financial support.}
\address[Benjamin Moore]{Department of Combinatorics and Optimization, University of Waterloo, Waterloo, ON, Canada}  
\email{brmoore@uwaterloo.ca}

\newtheoremstyle{case}{}{}{\normalfont}{}{\itshape}{:}{ }{}

\newtheorem{theorem}{Theorem}
\newtheorem{corollary}[theorem]{Corollary}
\newtheorem{proposition}[theorem]{Proposition}
\newtheorem{definition}[theorem]{Definition}

\newtheorem{observation}[theorem]{Observation}

\newtheorem{lemma}[theorem]{Lemma}

\newtheoremstyle{case}{}{}{\normalfont}{}{\itshape}{\normalfont:}{ }{}

\tikzset{
  bigblue/.style={circle, draw=blue!80,fill=blue!40,thick, inner sep=1.5pt, minimum size=5mm},
  bigred/.style={circle, draw=red!80,fill=red!40,thick, inner sep=1.5pt, minimum size=5mm},
  bigblack/.style={circle, draw=black!100,fill=black!40,thick, inner sep=1.5pt, minimum size=5mm},
  bluevertex/.style={circle, draw=blue!100,fill=blue!100,thick, inner sep=0pt, minimum size=2mm},
  redvertex/.style={circle, draw=red!100,fill=red!100,thick, inner sep=0pt, minimum size=2mm},
  blackvertex/.style={circle, draw=black!100,fill=black!100,thick, inner sep=0pt, minimum size=2mm},  
  whitevertex/.style={circle, draw=black!100,fill=white!100,thick, inner sep=0pt, minimum size=2mm},  
  smallblack/.style={circle, draw=black!100,fill=black!100,thick, inner sep=0pt, minimum size=1mm},
  smallwhite/.style={circle, draw=black!100,fill=white!100,thick, inner sep=0pt, minimum size=1mm},    
}

\begin{document}
\begin{abstract}
In this paper we give structural characterizations of graphs not containing rooted $K_{4}$, $W_{4}$, $K_{2,4}$, and a graph we call $L$. 
\end{abstract}

\maketitle 

\section{Introduction}

Graph minors have seen extensive study over the years \cite{Smallminors,k24excludedminors,graphstructuretheorem}. Most graph minors results can be split into two distinct areas; results on the ``rough" structure for large $H$-minor-free graphs \cite{graphstructuretheorem,WQOtheorem}, or results on the exact structure of $H$-minor-free graphs \cite{Crumppaper,NoV8minors,Smallminors}. While there are lots of beautiful results on the approximate structure of large $H$-minor-free graphs, exact structural results have only been found for graphs with few edges \cite{Smallminors}. In this paper, we look at rooted graph minors, which are a well known generalization of graph minors \cite{rootedk23theorem,root, rootedk3}  and we give structural results for a set of small graphs. More precisely, we characterize graphs without rooted $K_{4}, W_{4}, K_{2,4}$ and $L$-minors  (see Figure \ref{L(X)labelling1} for a picture of the graph $L$) (and various subsets of those minors).

Given finite, undirected graphs $G$ and $H$, a graph $G$ has an $H$-minor if and only if a graph isomorphic to $H$ can be obtained from $G$ through the contraction and deletion of some edges, and possibly the removal of some isolated vertices. We will always assume all graphs are connected and in that case, one never needs to delete isolated vertices. To generalize minors to rooted minors, we use a well known equivalent notion of minors.  A graph $G$ has an $H$-minor if and only if there exists a set $\{G_{x} \ | \ x \in V(H)\}$ of pairwise disjoint connected subgraphs of $G$ indexed by the vertices of $H$, such that if $xy \in E(H)$, then there exists a vertex $v \in V(G_{x})$ and a vertex $u \in V(G_{y})$ where $uv \in E(G)$. If the set $\{G_{x}\ | \ x \in V(H) \}$ exists, we will say $\{G_{x} \ | \  x \in V(H) \}$ is an \textit{$H$-model} of $G$, and the sets $G_{x}$ are the \textit{branch sets} of the $H$-model.

 From this definition of $H$-minors, it is easy to generalize $H$-minors to \textit{rooted $H$-minors}. Let $X \subseteq V(G)$ and let $f : X \to V(H)$ be injective. We say $G$ has a  \textit{rooted $H$-minor with respect to $f$} if there is an $H$-model of $G$ such that for all $v \in X$, we have $v \in G_{\pi(v)}$. We call the vertices in $X$ the \textit{roots} of the rooted $H$-minor. 

In general we will want to allow more than one injective map, so if we have injective maps $f_{1}, f_{2},\ldots,f_{n} : X \to V(H)$, that a graph $G$ has a \textit{rooted $H$-minor} if $G$ has a rooted $H$-minor with respect to $f_{i}$, for any $i \in \{1,\ldots,n\}$.  For ease of notation, once the family of maps is defined for some graph $H$ and set of vertices $X$, we will just say that $G$ has an $H(X)$-minor. 

This paper will look at rooted minors with exactly four roots. Specifically, we look at we will look at the following four rooted minors: $K_{4}(X)$-minors where each root is mapped to a distinct vertex, $W_{4}(X)$-minors (where $W_{4}$ is the graph obtained by taking a $4$-cycle and adding a universal vertex),  where  the roots are mapped to the rim vertices, $K_{2,4}(X)$-minors where the roots are mapped to the large side of the bipartition, and $L(X)$-minors which we will define later. We give characterizations of graphs not containing $K_{4}(X)$-minors and $W_{4}(X)$-minors (Theorem \ref{w4cuts}), and graphs not containing $K_{4}(X)$, $W_{4}(X)$ and $K_{2,4}(X)$-minors (Theorem \ref{k4w4k24characterization}). We also show that for $2$-connected graphs, the class of $K_{4}(X)$, $W_{4}(X)$, $K_{2,4}(X)$ and $L(X)$-minor free graphs, that the class is equivalent to the class of $K_{4}(X)$, $W_{4}(X)$, $K_{2,4}(X)$ and $L'(X)$-free graphs for a smaller graph $L'$ (see Figure \ref{L(X)labelling1}).

\begin{figure}
\begin{center}
\includegraphics[scale=0.5]{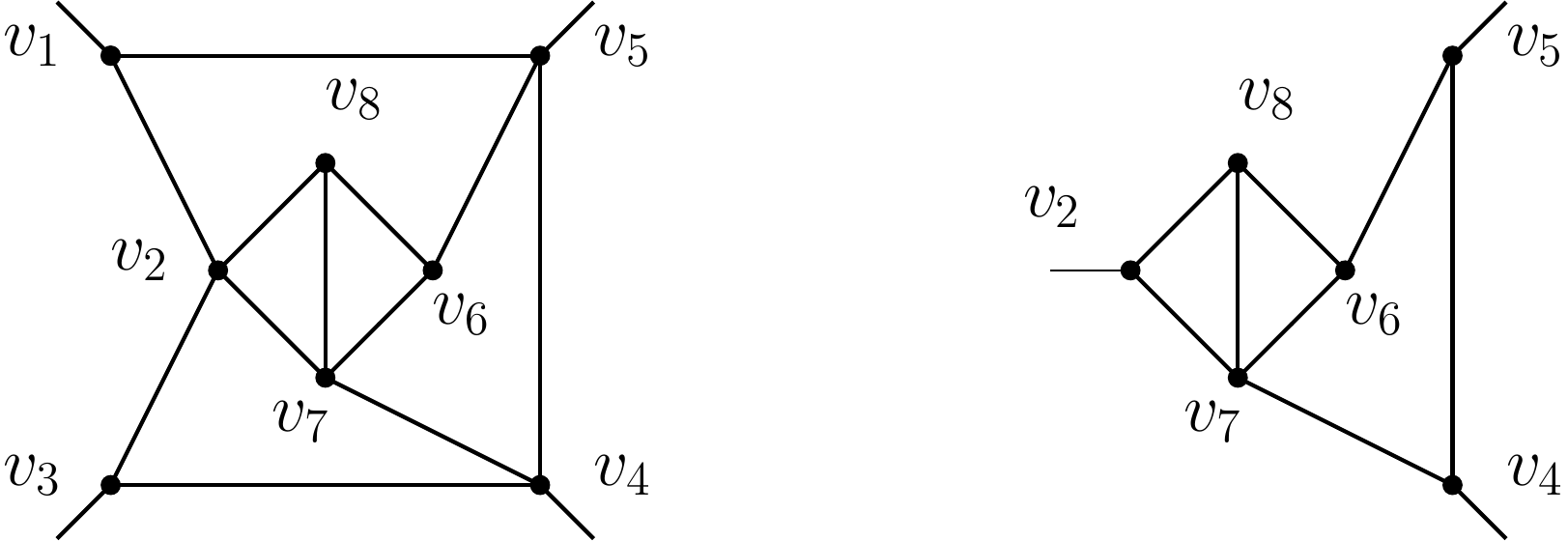}
\caption{The graph $L$ and the graph $L'$. The lines with only one endpoint indicate where the roots will be mapped for the rooted minor.}
\label{L(X)labelling1}
\end{center}
\end{figure}

To motivate the study of these rooted minors in particular, we turn to quantum field theory. In particular, the study of computing Feynman integrals and the related Feynman period (which is an integral which retains key information from the Feynman integral).  In parametric space, Feynman periods are defined from graphs, so one can ask if certain graph properties give any structure to the Feynman period. A celebrated result of Brown is that if a graph has vertex width less than $3$, then the cooresponding Feynman period evaluates to some multiple zeta value \cite{Francisbig}. In general, Brown showed that if a graph is ``reducible", then the corresponding Feynman period evaluates to some multiple zeta value (when the integral converges), and furthermore that being reducible is graph minor closed. In \cite{Ben, reducibilitypaper}, it was shown that reducibility is graph minor closed even when the graphs have ``external momenta" and ``particle masses" (as in, vertex and edge labels). In the case where there is exactly four vertex labels satisfying a physical restriction, and no edge labels, $K_{4}(X), W_{4}(X), K_{2,4}(X)$ and $L(X)$ are forbidden minors for reducibility. We also note there has been other work relating graph minors and computing Feynman periods, in particular the work of Black, Crump, DeVos, and Yeats characterizing Feynman $5$-splitting graphs and $3$-connected graphs which have vertex width less than $3$ \cite{Crumppaper,CrumpMastersthesis}. 

\subsection{Previous Results}

We start by outlining the structural characterization of graphs without $K_{4}(X)$-minors given by Monory and Wood in \cite{root}. This result acts as the starting point of all the other characterizations.

 Let $G$ be a graph and $X = \{a,b,c,d\} \subseteq V(G)$. We say that $G$ has a \textit{$K_{4}(X)$-minor} if and only if $G$ has a $K_{4}(X)$-minor with respect to $\pi$, where $\pi$ is any surjective map from $X$ to $V(K_{4})$ (see Figure \ref{K4example} for an example). Thus $G$ has a $K_{4}(X)$-minor if and only if $G$ has a $K_{4}$-minor where each vertex of $X$ ends up in a distinct branch set. Before we can state Monroy and Wood's characterization, we need some definitions.

Let $H$ be a graph. The graph $H^{+}$ is defined in the following way: for each triangle $T = \{x,y,z\}$ in $H$, we attach a clique of arbitrary size, $F_{T}$, to the triangle. As in, each vertex of $F_{T}$ is adjacent to each vertex of $T$, and not adjacent to any other vertex of $H$. We will let  $F$ denote the set of cliques attached to the triangles. Then as $H^{+}$ is uniquely defined from $H$ and $F$, we will use the notation $H^{+} = (H,F)$. 

Now, consider a planar graph $H$ where the outerface is a $4$-cycle, $C_{4}$, every internal face of $H$ is a triangle, and every triangle is a face. Let $V(C_{4}) = \{a,b,c,d\}$. In this case, we call the graph $H^{+}$ an \textit{$\{a,b,c,d\}$-web}. Now we can state the excluded $K_{4}(X)$-minor theorem.

 \begin{figure}
\begin{center}
\includegraphics[scale =0.7]{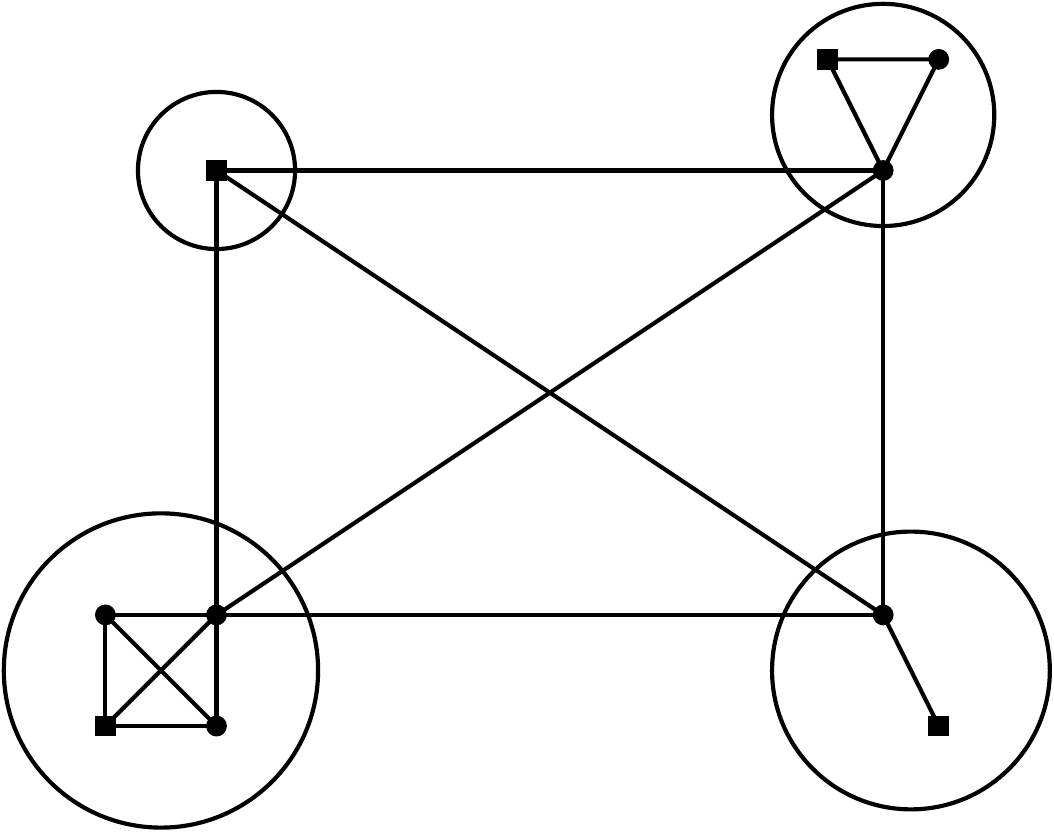}
\caption{A graph with a $K_{4}(X)$-minor. Circles represent the branch sets. The square vertices represent the roots.}
\label{K4example}
\end{center}
\end{figure}

\begin{theorem}[\cite{root}]
\label{k4free}
Let $G$ be a graph and $X = \{a,b,c,d\} \subseteq V(G)$. Then either $G$ has a $K_{4}(X)$-minor or $G$ is a spanning subgraph of a graph belonging to one of the following six classes of graphs:

\begin{itemize}

\item{Class $\mathcal{A}$: Let $H$ be the graph with vertex set $V(H) = \{a,b,c,d,e\}$ and with edge set $E(H) = \{ae,ad,be,bd,ce,cd,de\}$.  Class $\mathcal{A}$ is the set of all graphs $H^{+}$.}

\item{Class $\mathcal{B}$: Let $H$ be the graph with vertex set $V(H) = \{a,b,c,d,e,f\}$ and with edge set $E(H) = \{ae,af,be,bf,ce,cf,de,df,ef\}$. Class $\mathcal{B}$ is the set of all graphs $H^{+}$.}

\item{Class $\mathcal{C}$: Let $H$ be the graph such that $V(H) = \{a,b,c,d,e,f,g\}$ and with edge set $E(H) = \{ae,ag,be,bg,cf,cg,df,dg,ef,eg,fg\}$. Class $\mathcal{C}$ is the set of all graphs $H^{+}$.}

\item{Class $\mathcal{D}$: The set of all $\{a,b,c,d\}$-webs.}

\item{Class $\mathcal{E}$: Let $H'$ be a $\{c,d,e,f\}$-web, where $c,d,e$ and $f$ appear in that order on the outer $4$-cycle. Let $H$ be the graph with vertex set $V(H') \cup \{a,b\}$ and edge set $E(H) = E(H') \cup \{ae,af,be,bf\}$. Class $\mathcal{E}$ is the set of all graphs $H^{+}$. }

\item{Class $\mathcal{F}$: Let $H'$ be a $\{e,f,g,h\}$-web and suppose that $e,f,g,h$ appear in that order on the outer $4$-cycle. Let $H$ be the graph with vertex set $V(H) = V(H') \cup \{a,b,c,d\}$ and edge set $E(H) = E(H') \cup \{ae,af,be,bf,cg,ch,dg,dh\}$. Class $\mathcal{E}$ is the set of all graphs $H^{+}$. }

\end{itemize} 
\end{theorem}

Figure \ref{k4freepic} gives a pictorial representation of the graphs $H$ in the above classes. For the upcoming results, we will start with $3$-connected graphs and then consider lower connectivity afterwards. It is easily seen that Theorem \ref{k4free} simplifies significantly when we restrict to $3$-connected graphs.

\begin{corollary}[\cite{root}]
\label{k4(x)3conn}
Let $G$ be a $3$-connected graph and $X = \{a,b,c,d\} \subseteq V(G)$. Then either $G$ has a $K_{4}(X)$-minor or $G$ is a spanning subgraph of a Class $\mathcal{D}$ graph. In other words, $G$ is a spanning subgraph of an $\{a,b,c,d\}$-web.
\end{corollary}

\begin{figure}
\centering
\includegraphics[scale = 0.5]{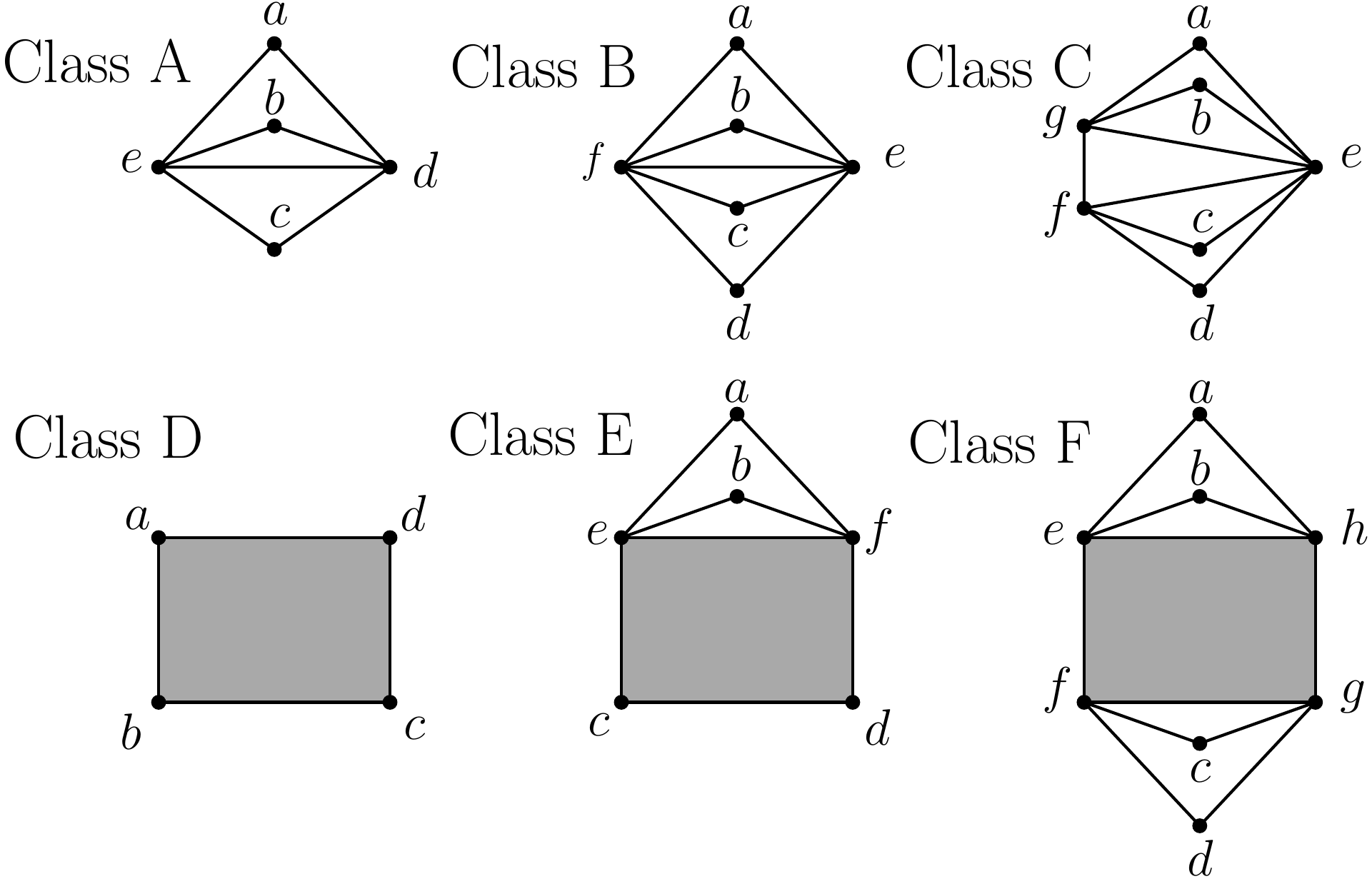}
\caption{The graphs $H$ which appear in Theorem \ref{k4free}. The cliques in the triangles are omitted. The shaded sections are $\{a,b,c,d\}$-webs.}
\label{k4freepic}
\end{figure}

We will also do reductions to planar graphs, so it is useful to notice that when $G$ is a $3$-connected planar graph, we can simplify Theorem \ref{k4free} even further.

\begin{corollary}[\cite{root}]
\label{k4(x)planar}
Let $G$ be a $3$-connected planar graph and $X = \{a,b,c,d\} \subseteq V(G)$.  Then $G$ does not have a $K_{4}(X)$-minor if and only if all the vertices of $X$ lie on the same face.
\end{corollary}

\section{All $3$-connected graphs have a $K_{4}(X)$ or $W_{4}(X)$-minor}

First we define what we mean by $W_{4}(X)$-minors. Let $G$ be a graph and $X = \{a,b,c,d\} \subseteq V(G)$. Let $\mathcal{F}$ be the family of maps from $X$ to $V(W_{4})$ such that each vertex of $X$ is mapped to a distinct vertex of the outer $4$-cycle of $W_{4}$.  For the purposes of this paper, a $W_{4}(X)$-minor refers to the $X$ and $\mathcal{F}$ given above. 

The goal of this section is to proof that for every $3$-connected graph, either we have a $K_{4}(X)$-minor or a $W_{4}(X)$-minor. By Corollary \ref{k4(x)3conn}, it suffices to show that all $3$-connected $\{a,b,c,d\}$-webs have a $W_{4}(X)$-minor. To do so, we will first show all $3$-connected planar have either a $K_{4}(X)$ or $W_{4}(X)$-minor, and then give an easy reduction to the general case. First we prove a lemma about paths in $3$-connected planar graphs. For notation, we will say an $(a,b)$-path is a path whose endpoints are $a$ and $b$.

\begin{lemma}
\label{3connectedplanarpaths}
Let $G$ be a $3$-connected planar graph and let $C$ be a facial cycle (as in, $C$ bounds a face) of $G$. Suppose $v,w \in V(C)$ and $vw \not \in E(G)$. Then there is a $(v,w)$-path $P$ such that $V(P) \cap V(C) = \{v,w\}$.  
\end{lemma} 

\begin{proof}
 By Mengar's Theorem, there are $3$-internally disjoint $(v,w)$-paths, say  $P_{1},P_{2},P_{3}$. Let $F_{1},F_{2}$ be the two facial walks of $C$ from $v$ to $w$ (as in the two disjoint $(v,w)$-paths on $C$). If any two of $P_{1},P_{2},P_{3}$ are the facial walks then we are done.

Therefore we assume at most one of $P_{1},P_{2},P_{3}$ is a facial walk. First we claim that for $i \in \{1,2,3\}$, either $V(P_{i}) \cap V(F_{1}) = \{v,w\}$ or $V(P_{i}) \cap V(F_{2}) = \{v,w\}$. If not, there exists an $i \in \{1,2,3\}$ such that we have two vertices $x_{1} \in V(F_{1}) \setminus \{v,w\}$ and $x_{2} \in V(F_{2}) \setminus \{w,v\}$ where $x_{1},x_{2} \in V(P_{i})$. Then consider the $(x_{1},x_{2})$-subpath on $P_{i}$ which we denote $P_{x_{1},x_{2}}$. We may assume that this subpath has no additional vertices from $F_{1}$ or $F_{2}$. Note that this subpath partitions $C$ into two cycles which separate $v$ and $w$. But  $C$ is a facial cycle, and $G$ is $3$-connected so any $(v,w)$-path intersects the subpath $P_{x_{1},x_{2}}$, contradicting that $P_{1},P_{2},P_{3}$ are internally disjoint $(v,w)$-paths.

Now suppose $\{v,w\} \subsetneq V(P_{1}) \cap V(F_{1})$. We claim that both $V(P_{2}) \cap V(F_{1}) = \{v,w\}$ and $V(P_{3}) \cap V(F_{1}) = \{v,w\}$. By the above argument, $V(P_{1}) \cap V(F_{2}) = \{v,w\}$.  Therefore every vertex of $F_{1}$ either belongs to $P_{1}$ or lies in a cycle created from some subpath of $P_{1}$ and a subpath of $F_{1}$. But since $C$ is a facial cycle, $G$ is planar, and $P_{2}$ and $P_{3}$ are internally disjoint from $P_{1}$, we get that $P_{2}$ and $P_{3}$ both are internally disjoint from $F_{1}$. By similar arguments, at most one of $P_{2}$, $P_{3}$ contains a vertex in $F_{2}$ which is not $v,w$. Therefore at least one of $P_{1}$, $P_{2}$, and $P_{3}$ has no vertex from $C$ except for $v$ and $w$, completing the claim.  
\end{proof}

\begin{lemma}
\label{Planar $3$-connected}
Let $G$ be a $3$-connected planar graph and let $X = \{a,b,c,d\} \subseteq V(G)$. Then $G$ has either a $K_{4}(X)$-minor or a $W_{4}(X)$-minor.
\end{lemma}

\begin{proof}
We may assume $G$ does not contain a $K_{4}(X)$-minor. Then by Corollary \ref{k4(x)planar},  $a,b,c$ and $d$ lie on a facial cycle, $F$, in that order.

First, suppose that $ac \in E(G)$. Without loss of generality we may assume $ac$ lies in the interior of $F$. Notice either the edge $bd \in E(G)$ or by Menger's Theorem  there is a $(b,d)$-path $P$ where $a,c \not \in V(P)$. In either case, this would contradict $F$ being a facial cycle, or $G$ not having a $K_{4}(X)$-minor. 

Therefore we can assume that $ac \not \in E(G)$. Then by Lemma \ref{3connectedplanarpaths} there is an $(a,c)$-path, $P_{a,c}$, which is internally disjoint from $F$. Without loss of generality we may assume that $P_{a,c}$ lies in the interior of $F$. Let $F_{a,c}$ be a facial walk from $a$ to $c$. Notice $C = P_{a,c} \cup F_{a,c}$ partitions the interior of $F$ into two regions, and that $b$ and $d$ lie in distinct regions. By Lemma \ref{3connectedplanarpaths}, we also have a $(b,d)$-path, $P_{b,d}$, which is disjoint from $F$. If $P_{b,d}$ lies on the exterior of $F$, then this would contradict that $F$ is a face. Thus $P_{b,d}$ lies in the interior of $F$. As $b$ and $d$ lie on differing sides of the partition of $C$, by the Jordan Curve Theorem, $P_{a,c}$ intersects $P_{b,d}$ at some vertex $v$.
   Then, contracting $F$ down to a four cycle on $a,b,c,d$ and contracting the subpaths from $v$ to $a,b,c,d$ on $P_{a,c}$ and $P_{b,d}$ to a vertex gives a $W_{4}(X)$-minor. 
\end{proof}

Now that we have shown that all planar $3$-connected graphs have a $W_{4}(X)$ or a $K_{4}(X)$-minor, we prove a lemma reducing the non-planar case to the planar case.

\begin{lemma}
\label{planarreduction}
Let $G$ be a $3$-connected graph and $X = \{a,b,c,d\} \subseteq V(G)$. Suppose $G$ is a spanning subgraph of an $\{a,b,c,d\}$-web, $H^{+} = (H,F)$. Then there is a $3$-connected planar graph $K$ such that $K$ is a minor of $G$.
\end{lemma}

\begin{proof}
For every triangle $T \in H$, consider the graph $G[V(F_{T})]$ (here recall $F_{T}$ is the clique attached to $T$ in $H^{+}$). Let $C_{1}, \ldots, C_{n}$ be the connected components of $G[V(F_{T})]$. For every triangle $T \in H$, contract $C_{1}$ down to a vertex, which we will call $v_{T}$, and contract all of $C_{2}, \ldots, C_{n}$ to an arbitrary vertex of $T$. Let $K$ be the graph obtained from $G$ after applying the above construction. We claim $K$ is planar and $3$-connected. 

First, we show $G$ is $3$-connected. Notice for each triangle $T$ where $V(T) = \{x_{1},x_{2},x_{3}\}$, for each $x_{i}$, $i \in \{1,2,3\}$ there is a vertex $v \in V(C_{1})$ such that $vx_{i} \in E(G)$. Thus in $K$, $v_{T}$ is adjacent to $x_{i}$ for all $i \in \{1,2,3\}$. It follows that $v_{T}$ has three internally disjoint paths to any other vertex of $K$.  Now let $x,y \in V(H)$ and let $P_{1},P_{2}$ and $P_{3}$ be three internally disjoint $(x,y)$-paths in $G$. Notice that for every triangle $T \in H$, at most one of $P_{1},P_{2}$ or $P_{3}$ uses vertices from $V(F_{T})$, since $|V(T)| = 3$. Therefore in $K$, if necessary, we can reroute the path using vertices from $V(F_{T})$ to use $v_{T}$, and thus in $K$ there are three internally disjoint $(x,y)$-paths. Therefore $K$ is $3$-connected.
 
 So it suffices to show that $K$ is planar. Notice that if given a some planar embedding of $H$, to all the faces bounded by a triangle, we can add a vertex to the interior of the face and make the vertex adjacent to every vertex in the triangle and remain planar. The graph from that construction contains $K$ as a subgraph, so $K$ is planar, completing the proof.  
\end{proof}

Now it follows easily that every $3$-connected graph has either a $W_{4}(X)$ or $K_{4}(X)$-minor. 
\begin{theorem}
\label{3connectivityW4K4}
Let $G$ be a $3$-connected graph and $X = \{a,b,c,d\} \subseteq V(G)$. Then $G$ either has a $K_{4}(X)$-minor or a $W_{4}(X)$-minor.
\end{theorem}

\begin{proof}
We may assume $G$ does not have a $K_{4}(X)$-minor. By Corollary \ref{k4(x)3conn}, $G$ is a spanning subgraph of an $\{a,b,c,d\}$-web. Then by Lemma \ref{planarreduction}, $G$ has a $3$-connected planar minor $K$. By Lemma \ref{Planar $3$-connected}, $K$ has $W_{4}(X)$-minor and thus $G$ has a $W_{4}(X)$-minor. 
\end{proof}

\section{$2$-connected graphs without $K_{4}(X)$ or $W_{4}(X)$-minors}
The goal of this section will be to give a spanning subgraph characterization of $2$-connected graphs without $K_{4}(X)$ or $W_{4}(X)$-minors, and some connectivity reductions for $1$-connected graphs. We first give some low order connectivity reductions. 

\subsection{Connectivity reductions}

Here we give easy connectivity reductions for $H(X)$-minors, which we could not find written anywhere. Throughout this section, if we do not say what the underlying family of maps $ \pi_{1},\pi_{2},\ldots,\pi_{n}$ is, it is assumed that we have an arbitrary family of maps. The following observation is obvious.

\begin{observation}
\label{connectedfromdisconnected}
Let $H$ be a connected graph. Let $G$ be a graph and $X \subseteq V(G)$. Then $G$ has an $H(X)$-minor if and only if $X$ is contained in a connected component of $G$, and the connected component has an $H(X)$-minor. 
\end{observation}

Then since all of our forbidden minors we will discuss are connected,  we will assume all graphs are connected. Now we have a few definitions.

 Given a graph $G$, for some positive integer $k$, a \textit{$k$-separation} of $G$ is a pair $(A,B)$ such that $A \subseteq V(G)$, $B \subseteq V(G)$, $A \cup B = V(G)$, $|A \cap B| \leq k$, and if $v \in B \setminus A$, and $u \in A \setminus B$, then $uv \not \in E(G)$. The vertices in $A \cap B$ are called the \textit{vertex boundary} of the separation. We say a $k$-separation is \textit{proper} if $A \setminus (A \cap B) \neq \emptyset$ and $B \setminus (A \cap B) \neq \emptyset$.  A proper $k$-separation is \textit{tight} if for all subsets $X \subsetneq A \cap B$, the set $X$ is not the vertex boundary of a separation. 

Let $G_{1}$ and $G_{2}$ be graphs a $k$-cliques as subgraphs. A \textit{$k$-clique-sum} or just \textit{$k$-sum} of $G_{1}$ and $G_{2}$ is a bijective identification of pairs of vertices in the two $k$-cliques with, if desired, removal of some edges from the new $k$-clique. We note sometimes it is enforced that all edges in the new $k$-clique are removed in a $k$-sum. In practice, under the assumption we can have parallel edges, this is equivalent to the above definition, as one simply adds parallel edges as desired.

Now the remainder of the section is dedicated to generalizing the following well-known lemma (see, for example, \cite{Smallminors}) to rooted minors.

\begin{lemma}
\label{lowordersepsminors}
Let $H$ be a $3$-connected graph. Let $G$ be a $k$-sum of $G_{1}$ and $G_{2}$ where $k \in \{0,1,2\}$. Then $G$ has an $H$-minor if and only if $G_{1}$ or $G_{2}$ has an $H$-minor. 
\end{lemma}

\subsection{Cut vertices}
Throughout this subsection, suppose that we have a simple $2$-connected graph $H$, and a connected graph $G$ where $G$ has a $1$-separation $(A,B)$ where $A \cap B = \{v\}$. Furthermore, let $X = \{a,b,c,d\} \subseteq V(G)$, and without loss of generality suppose that $|A \cap X| \geq |B \cap X|$. Let $\mathcal{F}$ be an arbitrary family of injective maps from $X$ to $V(H)$. Figure \ref{Oneconnectedreductionspic} gives a pictorial representation of the upcoming lemmas. 

\begin{lemma}
\label{allonesidecutvertex}
If $X \subseteq A$, then $G$ has an $H(X)$-minor if and only if $G[A]$ has an $H(X)$-minor.
\end{lemma}

\begin{proof}
First suppose $G[A]$ has an $H(X)$-minor. Then $G$ has an $H(X)$-minor by extending the branch set containing $v$ to contain all of $G[B]$.

Conversely, suppose $G$ has an $H(X)$-minor and let $\{G_{x} \ | \ x \in V(H)\}$ be a model of an $H(X)$-minor in $G$. As $X \subseteq A$, $H$ is $2$-connected, and $v$ is a cut vertex, there is no branch set which is  strictly contained inside $G[B \setminus \{v\}]$. Notice if all of the branch sets are contained inside $G[A]$, then we are done since $\{G[V(G_{x}) \cap A]\ | \ x \in V(H)\}$ would be the desired $H(X)$-model in $G[A]$. Therefore we assume at least one branch set contains vertices from $B$. Since $v$ is a cut vertex, there is only one branch set containing vertices from $B$. Let $G_{z}$, for some $z \in V(H)$, be such a branch set. Notice that $G_{z} \cap G[A]$ is a connected subgraph of $G[A]$, and as $G_{z}$ was the only branch set containing vertices in $B$, $\{G_{x} \cap G[A] \ | \  x \in V(H)\}$ is an $H(X)$-model in $G[A]$. 
\end{proof}

\begin{lemma}
Suppose $a,b,c \in A \setminus \{v\}$, and $d \in B \setminus \{v\}$. Let $X_{A} = \{a,b,c,v\}$ and for each $\pi \in \mathcal{F}$, define $\pi': X_{A} \to V(H)$ such that $\pi' = \pi$ except that $\pi'(v) = \pi(d)$.  Then $G$ has an $H$-minor if and only if $G[A]$ has an $H(X_{A})$-minor.
\end{lemma}

\begin{proof}
Let $\{G_{x} \ |\  x \in V(H)\}$ be an $H(X_{A})$-model in $G[A]$. Suppose that $G_{d}$ is the branch set where $v \in V(G_{d})$. Then we obtain an $H(X)$-model of $G$ by extending $G_{d}$ to contain all of $G[B]$.  

Conversely, let $\{G_{x} \ | \ x \in V(H)\}$ be a model of an $H(X)$-minor in $G$. Let $G_{d}$ be the branch set where $d \in V(G_{d})$. As $H$ is $2$-connected, and $v$ is a cut vertex, we have that $v \in G_{d}$. Therefore all other branch sets are contained inside $G[A]$, and thus all required adjacencies for the $H(X)$-minor exist in $G[A]$. Therefore $\{G_{x} \cap G[A] \ | \ x \in V(H)\}$ is an $H(X_{A})$-minor of $G[A]$.  
\end{proof}

\begin{lemma}
Suppose $v=a$. Then $G$ has an $H(X)$-minor if and only if $X \subseteq A$ and $G[A]$ has an $H(X)$-minor.
\end{lemma}

\begin{proof}
Sufficiency follows from Lemma \ref{allonesidecutvertex}.

Conversely, let $\{G_{x} \ | \ x \in V(H)\}$ be a model of an $H(X)$-minor in $G$. Towards a contradiction,  we consider the case where $d,c \in A \setminus \{v\}$ and $b \in B \setminus \{v\}$. Let $G_{v_{1}}$, $G_{v_{n}}$, $G_{z}$ be the branch sets for which $b \in V(G_{v_{1}})$, $c \in G_{v_{n}}$, and $a \in G_{z}$.  As $H$ is $2$-connected, there is a path $P = v_{1},v_{2},\ldots,v_{n}$ in $H$ such that $z \not \in V(P)$. Then there is a sequence of branch sets, $G_{v_{1}},\ldots,G_{v_{n}}$, such that $G_{v_{i}}$ has a vertex which is adjacent to a vertex in $G_{v_{i+1}}$ for all $i \in \{1,\ldots,n-1\}$. But $G_{v_{1}} \subseteq G[A-v]$, $G_{v_{n}} \subseteq G[B-v]$,  $G_{z} \neq G_{v_{i}}$ for any $i \in \{1,\ldots,n\}$ and $a \in V(G_{z})$, which a contradiction. The other cases follow similarly. 
\end{proof}

\begin{lemma}
If exactly two vertices of $X$ are in $B \setminus \{v\}$ and exactly two vertices of $X$ are in $A \setminus \{v\}$, then $G$ does not have an $H(X)$-minor.
\end{lemma}

\begin{proof}

Let $\{G_{x} \ | \ x \in V(H)\}$ be an $H$-model of $G$. Let $G_{y}$ be the branch set containing $v$.  Suppose $a \in A \setminus \{v\}$ and $b \in B \setminus \{v\}$ and $a \in G_{a}$ and $b \in G_{b}$. Then if we contract each branch set to a vertex to obtain the $H$-minor, all $(a,b)$-paths in $H$ contain $y$ since $v \in G_{y}$ and $v$ is a cut vertex. But then $H$ is not $2$-connected, a contradiction. 
\end{proof}

\begin{figure}
\begin{center}
\includegraphics[clip,trim=0.5cm 4cm 0.5cm 0.5cm, scale=0.5]{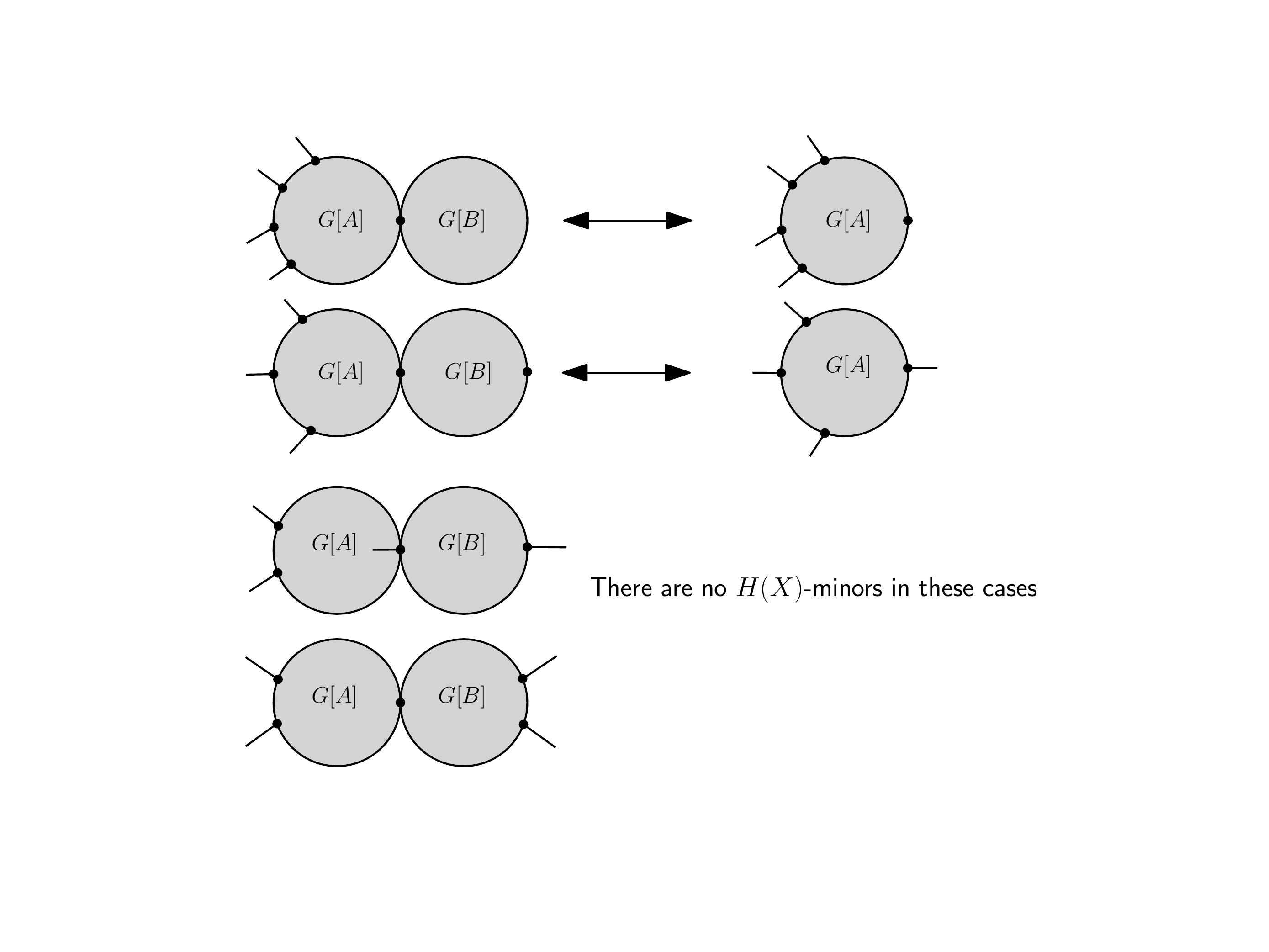}
\caption{Cut vertex reductions. Edges with only one endpoint represent the root vertices.  }
\label{Oneconnectedreductionspic}
\end{center}
\end{figure}

We note that these are all of the possibilities for how the roots can be distributed across a cut vertex. Now we restrict our attentions to $2$-connected graphs.

\subsection{$2$-connected reductions}

For this section suppose that $H$ is a $3$-connected simple graph and that $G$ is a $2$-connected graph. Furthermore, assume that $G$ has a $2$-separation $(A,B)$ such that $A \cap B = \{u,v\}$, and let $X = \{a,b,c,d\} \subseteq V(G)$. We define $G_{A} = G[A] \cup \{uv\}$ and $G_{B} = G[B] \cup \{uv\}$. Let $\mathcal{F}$ be a family of injective maps from $X$ to $V(H)$. By the discussion in the previous section, we may assume $G$ is $2$-connected.

\begin{lemma}
\label{2sepW4}
Let $L \subseteq X$ such that $|L| =2$ and suppose $L \subseteq A \setminus \{u,v\}$. Furthermore, suppose that $X \setminus L \subseteq B \setminus \{u,v\}$. Let $X_{A} = \{L,u,v\}$ and $X_{B} = \{(X \setminus L),u,v\}$. Suppose that for any $\pi_{1} \in \mathcal{F}$, there exists a $\pi_{2} \in \mathcal{F}$ such that $\pi_{1}(c) = \pi_{2}(d)$ and $\pi_{1}(d) = \pi_{2}(c)$. For each $\pi \in \mathcal{F}$, define $\pi_{A}$ and $\pi_{B}$ in the natural way so that $u,v$ replace the vertices in $L$ and $X \setminus L$ respectively. Then $G$ has an $H(X)$-minor if and only if either $G_{A}$ has an $H(X_{A})$-minor or $G_{B}$ has an $H(X_{B})$-minor. 

\end{lemma}

\begin{proof}
Suppose that $G_{A}$ contains an $H(X_{A})$-minor and suppose that $c,d \in B \setminus \{u,v\}$. As $G$ is $2$-connected, there are two disjoint paths between $\{u,v\}$ and $\{c,d\}$. Since $c,d \in B \setminus \{u,v\}$, these paths are contained inside of $G[B]$. Thus we can contract $G[B]$ to $\{u,v\}$ in such a way that $c$ and $d$ do not get identified together (by an easy application of Mengers Theorem). Since we supposed that for any $\pi_{1} \in \mathcal{F}$, there exists a $\pi_{2} \in \mathcal{F}$ such that $\pi_{1}(c) = \pi_{2}(d)$ and $\pi_{1}(d) = \pi_{2}(c)$, the graph $G$ has an $H(X)$-minor. The other cases follow similarly. 

Conversely, let $\{G_{x} | x \in V(H)\}$ be a model of an $H(X)$-minor. Suppose $a,b \in A \setminus \{u,v\}$ and $c,d \in B \setminus \{u,v\}$. Let $a \in G_{a}$, $b \in G_{b}$, $c \in G_{c}$, and $d \in G_{d}$. First, suppose there are branch sets $G_{y}$ and $G_{z}$ such that $G_{y} \subseteq G[A - \{u,v\}]$ and $G_{z} \subseteq G[B- \{u,v\}]$. Then if we contract each branch set down to a vertex, there would be at most two internally disjoint $(y,z)$-paths, contradicting that $H$ is $3$-connected.  

Therefore we can assume that either $u \in G_{a}$, $v \in G_{b}$ and $A \subseteq V(G_{a}) \cup V(G_{b})$  or $u \in G_{c}$, $v \in G_{d}$ and $B \subseteq V(G_{c}) \cup V(G_{d})$. Suppose $u \in G_{a}$ and $v \in G_{b}$ and $A \subseteq V(G_{a}) \cup V(G_{b})$. Then all other branch sets are contained in $G[B- \{u,v\}]$. Then since $uv \in E(G_{b})$,  $\{G'_{x} = G_{B}[V(G_{x}) \cap V(G_{B})] \; | x \in V(H)\}$ is an $H(X)$-model in $G_{B}$. The other case follows similarly.
\end{proof}

We remark that all rooted graph minors we will see throughout this paper satisfy the technical condition in the above lemma.

\begin{lemma}
\label{oneterminalonesideothersother}
Suppose  $a \in A \setminus \{u,v\}$ and $b,c,d \in B \setminus \{u,v\}$.  Let $X_{1} = \{u,b,c,d\}$ and $X_{2} = \{v,b,c,d\}$. For each $\pi \in \mathcal{F}$, let $\pi_{1}$ satisfy $\pi_{1} = \pi$ on $X_{2} \setminus \{u\}$ and $\pi_{1}(u) = \pi(a)$. Let $\pi_{2} = \pi$ on $X_{1} \setminus \{v\}$ and $\pi_{2}(v) = \pi(a)$. Then $G$ has an $H(X)$-minor if and only if either $G_{B}$ has an $H(X_{1})$-minor or an $H(X_{2})$-minor. 
\end{lemma}

\begin{proof}
Suppose $G_{B}$ has an $H(X_{1})$-minor. By Menger's Theorem there exists a path from $a$ to $u$ which does not contain $v$, and therefore we can contract $G[A]$ to $\{u,v\}$ in such a way that $a$ gets contracted onto $u$. Therefore $G$ has an $H(X)$-minor. The case where $G_{B}$ has an $H(X_{2})$-minor follows similarly.

Conversely, let $\{G_{x} | x \in V(H)\}$ be a model of an $H(X)$-minor in $G$. Let $a \in G_{a}$, $b \in G_{b}, c \in G_{c}$, and $d \in G_{d}$. First, suppose for some $y \in V(H)$, $G_{y} \subseteq G[A - \{u,v\}]$. If $yb \in E(H)$, then one of $u$ or $v$ is in $V(G_{b})$. Then at least two of the following occur: $V(G_{a}) \subseteq A \setminus \{u,v\}$, $V(G_{c}) \subseteq B \setminus \{u,v\}$ or $V(G_{d}) \subseteq B \setminus \{u,v\}$. If $V(G_{a}) \subseteq A \setminus \{u,v\}$ and $V(G_{c}) \subseteq B \setminus \{u,v\}$, then if we contract all the branch sets down to a vertex, there are at most two internally disjoint $(a,c)$-paths contradicting that $H$ is $3$-connected. The other case when $V(G_{a}) \subseteq A \setminus \{u,v\}$ follows similarly, and thus we can assume that $G_{a}$ contains one of $u$ or $v$.   But then, contracting all branch sets to a vertex there are at most two internally disjoint $(c,y)$-paths contracting that $H$ is $3$-connected. 

Thus we may assume that $yb \not \in E(H)$. Then by Menger's Theorem there are three internally disjoint $(b,y)$-paths. If either $u$ or $v$ is in $V(G_{b})$, then the above argument can be applied to derive a contradiction. Therefore we assume that $u,v \not \in V(G_{b})$. But then contradicting all the branch sets down to a vertex, every $(b,y)$-path uses the vertex which was obtained by contracting the branch sets that $u$ or $v$ were in down to a single vertex.  But that implies there are at most two internally disjoint $(b,y)$-paths, a contradiction. Therefore for every $y \in V(H)$, $G_{y} \not \subseteq G[A - \{u,v\}]$. 

Then since $a \in A \setminus \{u,v\}$, at least one of $u$ or $v$ is contained in $G_{a}$. Therefore at most one other branch set contains vertices from $A$. Then since $uv \in G_{B}$, $\{G'_{x} = G_{B}[V(G_{x}) \cap V(G_{B})] \; | x \in V(H)\}$ is a model for either an $H(X_{1})$ or $H(X_{2})$-minor in $G_{B}$, depending on which of $u$ and $v$ is in $G_{a}$. 
\end{proof}

\begin{lemma}
\label{allonesidegeneralH}
Suppose that $X \subseteq A$. Then $G$ has an $H(X)$-minor if and only if $G_{A}$ has an $H(X)$-minor.
\end{lemma}

\begin{proof}
Suppose $G_{A}$ has an $H(X)$-minor. Then contracting $G[B]$ onto $\{u,v\}$ gives $G_{A}$, and thus $G$ has a $W_{4}(X)$-minor.

Let $\{G_{x} | x \in V(H)\}$ be a model of an $H(X)$-minor. Let $a \in G_{a}, b \in G_{b}, c \in G_{c}$ and $d \in G_{d}$. Suppose there is a $y \in V(H)$ such that $G_{y}$ is contained in $G[B \setminus \{u,v\}]$. Note that $y \neq a,b,c$ or $d$. Then since $\{u,v\}$ is a $2$-vertex cut, at least two of the following occur: $G_{a} \in G[A \setminus \{u,v\}]$, $G_{b} \in G[A \setminus \{u,v\}]$, $G_{c} \in G[A \setminus \{u,v\}]$ and $G_{d} \in G[A \setminus \{u,v\}]$. Without loss of generality, suppose that $G_{a} \in G[A \setminus \{u,v\}]$. But then if we contract each branch set down to a vertex, there is at most two internally disjoint $(a,y)$-paths in $H$, contradicting that $H$ is $3$-connected. Therefore there are no branch sets contained in $G[B \setminus \{u,v\}]$. Then since $\{u,v\}$ is a $2$-vertex cut, there are at most two branch sets using vertices in $B$. If there is only one branch set using vertices from $B$, then easily $\{G'_{x} = G[V(G_{x}) \cap V(G_{A})] \; | x \in V(H)\}$ is an $H(X)$-minor of $G_{A}$. If two branch sets contain vertices from $B$, then since $uv \in E(G_{A})$,  $\{G'_{x} = G_{A}[V(G_{x}) \cap V(G_{A})] \; | x \in V(H)\}$ is a model of an $H(X)$-minor of $G_{A}$. 
\end{proof}

\begin{lemma}
\label{noW4minor}
Suppose $L \subseteq X$ where $L = \{u,v\}$. Furthermore, suppose that there is a vertex of $X \setminus L$ in $A \setminus \{u,v\}$ and a vertex of $X \setminus L$ in $B \setminus \{u,v\}$. Then $G$ does not have an $H(X)$-minor.
\end{lemma}

\begin{proof}

Consider any $L \subseteq X$. Suppose for a contradiction that $\{G_{x} \ | \ x \in V(H)\}$ is a model of an $H(X)$-minor.  Consider the case where $b,c \in L$ and $a \in A \setminus \{u,v\}$, $d \in B \setminus \{u,v\}$, and $a \in G_{a}$, and $d \in G_{d}$. Then $V(G_{a}) \subseteq A \setminus \{u,v\}$ and $V(G_{d}) \subseteq B \setminus \{u,v\}$. But then contracting each branch set down to a vertex, there are at most two internally disjoint $(a,d)$-paths, contradicting that $H$ is $3$-connected. The other cases follow similarly. 
\end{proof}

\begin{figure}
\begin{center}
\includegraphics[scale =0.5]{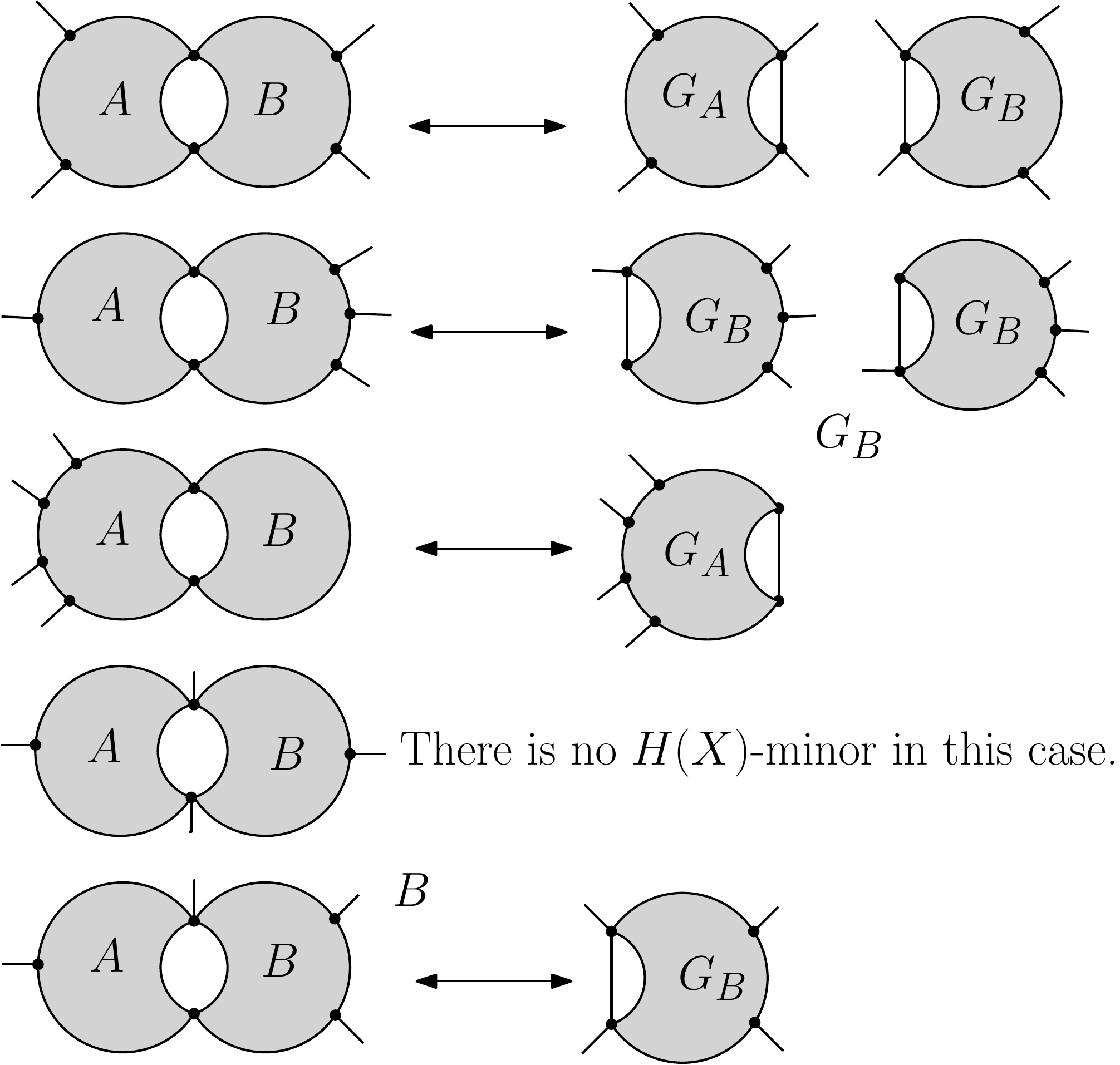}
\caption{The $2$-connected reductions. Edges with only one endpoint represent vertices from $X$.}
\label{2connectivitylemmas}
\end{center}
\end{figure}

\begin{lemma}
\label{oneterminalinthecut2conn}
Suppose $u = a$, and that there is exactly one vertex of $X$ in $A \setminus \{u,v\}$ and two vertices of $X$ in $B \setminus \{u,v\}$. Let $X_{A} = \{u,v,(X \cap B \setminus \{u,v\})\}$. For each $\pi \in \mathcal{F}$, define $\pi_{A} = \pi$ on $X_{A} \setminus \{v\}$ and let $\pi_{A}(v) = \pi(x)$ where $x \in X \setminus \{u, (X \cap A \setminus \{u,v\})\}$. Then $G$ has an $H(X)$-minor if and only if $G_{A}$ has an $H(X_{A})$-minor. 
\end{lemma}

\begin{proof}

 Suppose $G_{A}$ has an $H(X_{A})$-minor. Then by Menger's Theorem there is a path from the vertex of $X \cap A \setminus \{u,v\}$ to $v$ not containing $u$. Thus we can contract $G[A]$ to  $\{u,v\}$ such that we obtain the edge $uv$ and the vertex of $X$ in $A \setminus \{u,v\}$ is contracted to $v$. The resulting graph is isomorphic to $G_{A}$ and thus $G$ has an $H(X)$-minor.

Conversely, let $\{G_{x} | x \in V(H)\}$ be a model of an $H(X)$-minor in $G$. Let $b \in A \setminus \{u,v\}$ and suppose $a \in G_{a}$, $b \in G_{b}$, $c \in G_{c}$ and $d \in G_{d}$. First, suppose for some $y \in V(H)$,  $V(G_{y}) \subseteq A \setminus \{u,v\}$.  As $a =u$,  either $V(G_{c}) \subseteq B \setminus \{u,v\}$ or $V(G_{d}) \subseteq B \setminus \{u,v\}$. Without loss of generality, suppose that $V(G_{c}) \subseteq B \setminus \{u,v\}$.  Then if we contract all of the branch sets to a vertex, there are at most two internally disjoint $(y,c)$-paths, contradicting that $H$ is $3$-connected. Thus the only two branch sets with vertices in $A$ are $G_{a}$ and $G_{b}$. Since $uv \in E(G_{B})$, we have that $\{G'_{x} = G_{B}[V(G_{x}) \cap V(G_{B})] \; | x \in V(H)\}$ is a model of an $H(X)$-minor of $G_{B}$. 
 
\end{proof}

We note that this is every possible way to distribute four roots across $2$-separations. Working through the lemmas, one obtains a characterization of all graphs not containing a $W_{4}(X)$ or $K_{4}(X)$-minor. 

\subsection{A characterization of $2$-connected graphs without $W_{4}(X)$ and $K_{4}(X)$-minors}

Now we give a spanning subgraph characterization of graphs without $W_{4}(X)$ and $K_{4}(X)$-minors in somewhat similar vein to Theorem \ref{k4free}. It turns out the only interesting case are webs. 

\begin{lemma}
Let $G$ be a $2$-connected graph such that $G$ is a spanning subgraph of a class $\mathcal{A},\mathcal{B}$ or $\mathcal{C}$ graph (see Theorem \ref{k4free}). Then $G$ does not have a $W_{4}(X)$-minor. 
\end{lemma}

\begin{proof}
We treat each case separately.

If $G$ is the spanning subgraph of a class $\mathcal{A}$ graph, then $\{d,e\}$ is a $2$-vertex cut. Applying Lemma \ref{oneterminalinthecut2conn} and Lemma \ref{noW4minor} successively to the separation induced by $\{d,e\}$, we see $G$ does not have a $W_{4}(X)$-minor.

If $G$ is the spanning subgraph of a class $\mathcal{B}$ graph, then $\{e,f\}$ is a $2$-vertex cut. Applying  Lemma \ref{2sepW4} and Lemma \ref{noW4minor} successively to the separation induced by $\{e,f\}$, we see $G$ does not have a $W_{4}(X)$-minor.

 If $G$ is the spanning subgraph of a class $\mathcal{C}$ graph, then $\{g,f\}$ is a $2$-vertex cut. Apply Lemma \ref{2sepW4} to the separation induced by $\{g,f\}$ and let $G_{1}$ and $G_{2}$ be the graphs obtained from Lemma \ref{2sepW4}. Without loss of generality, let $G_{1}$ be the graph such that  $\{f,g\}$ induces a separation satisfying Lemma \ref{noW4minor}. Then $G_{1}$ does not have a $W_{4}(X)$-minor.  Then in $G_{2}$, to the separation induced by the $2$-vertex cut $\{g,e\}$, apply Lemma \ref{oneterminalinthecut2conn} to obtain a graph $G_{3}$. Then in $G_{3}$, notice that Lemma \ref{noW4minor} applies, thus $G_{3}$ does not have a $W_{4}(X)$-minor, and thus we get that $G$ does not have a $W_{4}(X)$-minor. 
\end{proof}

Notice that by applying Lemma \ref{2sepW4} to $2$-connected spanning subgraphs of class $\mathcal{E}$ and $\mathcal{F}$ graphs, we see that one of these graphs has a $W_{4}(X)$-minor if and only if the corresponding web from class $\mathcal{D}$ has a $W_{4}(X)$-minor. So now we restrict ourselves to looking at $\{a,b,c,d\}$-webs. First we show that $\{a,b,c,d\}$-webs always have a cycle which contains $\{a,b,c,d\}$.

\begin{observation}
\label{planarface}
Suppose $G$ is a planar spanning subgraph of some $\{a,b,c,d\}$-web $H^{+} = (H,F)$. Then the graph $G'$ defined by $V(G') = V(G)$ and $E(G') = E(G) \cup \{ab,bc,cd,da\}$ is planar, and furthermore the cycle $C$ with edge set $ab,bc,cd,da$ is the boundary of a face in $G'$. 
\end{observation}

\begin{proof}

Fix a planar embedding $\tilde{H}$ of $H$. As $G$ is planar, for each triangle $T \in H$ the graph $F_{T} \cup T$ is planar. For each triangle $T \in H$, fix a  planar embedding of $F_{T} \cup T$ where $T$ is the boundary of the outerface. Then we can combine the planar embedding of $H$ with the planar embeddings of $F_{T} \cup T$ by joining $F_{T} \cup T$ to the appropriate triangle. This implies that the graph $H^{+}$ is planar, and thus $G'$ is planar. Additionally, notice in  $H^{+}$, the cycle with edge set $ab,bc,cd,da$ is the boundary of a face in $H^{+}$, and thus the cycle with edge set $ab,bc,cd,da$ in $G'$ is the boundary of a face in $G'$. 
\end{proof}

\begin{lemma}
\label{webcycleplanar}
Let $G$ be a $2$-connected planar graph and let $X = \{a,b,c,d\} \subseteq V(G)$. If $G$ is the spanning subgraph of an $\{a,b,c,d\}$-web, then there is a cycle, $C$, such that $X \subseteq V(C)$.
\end{lemma}

\begin{proof}
Let $G_{1},\ldots,G_{n}$ be a sequence of graphs where $G_{1} = H^{+}$, $G_{n} = G$ and $G_{i+1} = G_{i} \setminus \{e\}$ where $e$ is some edge of $G_{i}$. We proceed by induction on $i$.  When $i =1$, $G_{1} = H^{+}$ and the $4$-cycle on $a,b,c,d$ in $H$ is our desired cycle.

 Now consider $G_{i}$, $i \geq 2$ and let $e=xy \in E(G)$ be the edge such that $G_{i} = G_{i-1} \setminus \{e\}$. By induction, $G_{i-1}$ contains a cycle $C$ containing $X$. We may assume that $e \in E(C)$ as otherwise $C$ completes the claim. Let $P= C \setminus \{e\}$.   Without loss of generality, suppose that $a,b,c,d$ appear in that order in $C$, and that $x$ and $y$ lie on the $(a,d)$-path, $P_{a,d}$, in $C$ in $G_{i-1}$ which does not contain $c$ and $d$, such that $a,x,y,d$ appear in that order. Similarly define paths $P_{a,b}, P_{b,c}$ and $P_{d,a}$. Additionally define $P_{a,x}$ to be the $(a,x)$-subpath on $P_{a,d}$ and $P_{y,d}$ to be the $(y,d)$-subpath on $P_{a,d}$.
 
  By Observation \ref{planarface}, in the graph $G'_{i-1}$, we have that $P_{a,d} \cup \{ad\}$, $P_{a,b} \cup \{ab\}, P_{b,c} \cup \{bc\}$ and $P_{c,d} \cup \{cd\}$ are cycles. Furthermore, we may assume that in a planar embedding of $G'_{i-1}$, no edges cross $ad,ab,bc,$ or $cd$. We define the interior of $C$ to be the component of $G'_{i-1} -C$ which does not contain any of $ab,bc,cd$ or $ab$, and the exterior is the component which is not the interior. We abuse notation and will refer also interior and exterior of $C$ in $G_{i-1}$. Notice that if we have a path whose two endpoints are on $P_{a,d}$ and whose vertices only use exterior vertices, then that path does not contain any vertices from $P_{a,b} \cup P_{b,c} \cup P_{c,d} \setminus \{a,d\}$.
  
As $G_{i}$ is $2$-connected, there is an $(x,y)$-path, $P'$ such that $P' \neq P$. If $V(P') \cap V(P) = \{x,y\}$, then $P' \cup P$ is our desired cycle. Therefore we may assume that every $(x,y)$-path intersects $P$.
   If there is any path $P''$ from a vertex $x' \in V(P_{a,x})$ to a vertex $y' \in  V(P_{y,d})$ using only vertices from the exterior, then $P_{a,x'} \cup P'' \cup P_{y',d} \cup P_{c,d} \cup P_{b,c} \cup P_{a,d}$ is a cycle, since no edges cross the edge $ad \in E(G'_{i-1})$. Here $P_{a,x'}$ is the $(a,x')$-subpath on $P_{a,x}$ and $P_{y',d}$ is the $(y',d)$-subpath on $P_{y,d}$. Therefore we assume no such path of that form exists. By essentially the same argument, we can assume no path of that form exists with vertices in the interior which does not intersect any of $P_{a,b}, P_{b,c}$ and $P_{c,d}$.
   
Since $G_{i}$ is $2$-connected, there are two internally disjoint $(x,y)$-paths, say $P'$ and $P''$. By our previous discussion, we may assume that both $P'$ and $P''$ are not $P$, and that both $P'$ and $P''$ intersect $P$. Suppose that $P'$ intersects all of $P_{b,c}$, $P_{c,d}$ and $P_{a,b}$. Notice that by planarity, these paths cannot cross, so it is well defined to say that one of $P' - \{x,y\}$ or $P'' - \{x,y\}$ lies on the interior of the cycle $P' \cup \{xy\}$ or $P'' \cup \{xy\}$.  Without loss of generality, suppose that $P' -\{x,y\}$ lies on the interior of $P'' \cup \{xy\}$. Then by the previous discussion, and planarity, the only way for $P''$ to be an internally disjoint $(x,y)$-path is for $P''$ to intersect $P_{a,b}$, go through the exterior of $C$ and intersect $P_{a,b}$ again, do this some finite number of times, then intersect $P_{b,c}$, go through the exterior of $C$ and intersect $P_{b,c}$ again, do this some finite number of times, then intersect $P_{c,d}$, go through the exterior of $C$ and intersect $P_{c,d}$ again. Then we can reroute $P$ along paths in the exterior of $C$ along $P''$ to get a new path $P'''$ such that $P'''$ contains all of $a,b,c,d$ and $V(P''') \cap V(P') \setminus \{x,y\} = \emptyset$. But then $P''' \cup P'$ a cycle satisfying the claim. We note the same strategy holds if $P'$ intersects any subset of $P_{b,c}$, $P_{c,d}$ and $P_{a,b}$. Therefore there is a cycle containing $a,b,c$, and $d$ in $G$. 
\end{proof}

\begin{figure}
\begin{center}
\includegraphics[scale = 0.5]{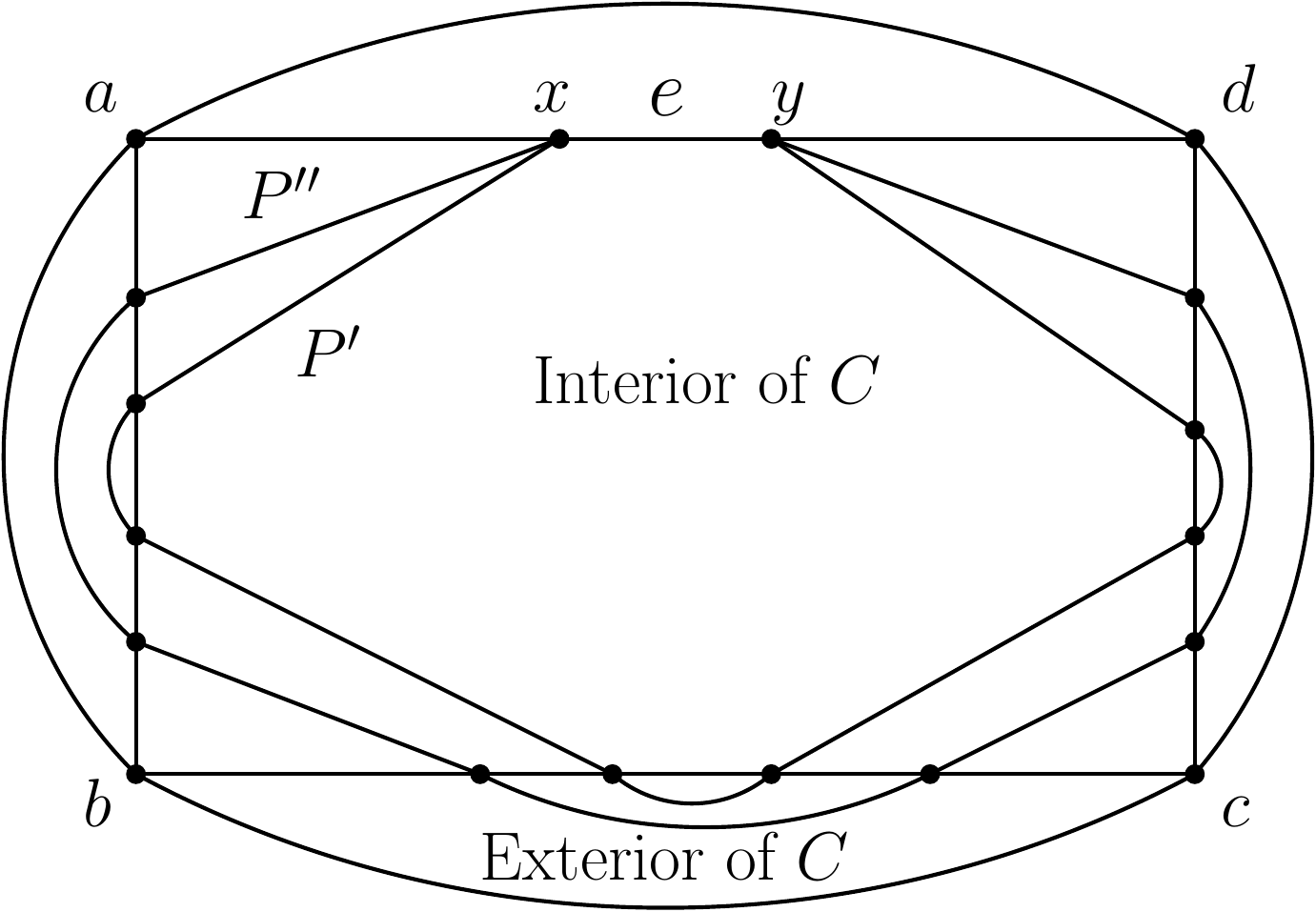}
\caption{The situation in Lemma \ref{webcycleplanar}. The edges $ab,bc,cd$ and $ab$ exist only in $G'_{i}$. The path $P'''$ is obtained by rerouting along paths in the exterior of $C$ which are subpaths of $P''$.}
\end{center}
\end{figure}

\begin{corollary}
\label{webcycle}
Let $G$ be a $2$-connected graph and let $X = \{a,b,c,d\} \subseteq V(G)$. If $G$ is the spanning subgraph of an $\{a,b,c,d\}$-web, $H^{+} = (H,F)$, then there is a cycle, $C$, such that $X \subseteq V(C)$.
\end{corollary}

\begin{proof}
By Lemma \ref{webcycleplanar}, we may assume that $G$ is non-planar. For each triangle $T \in H$, consider the graph $G[V(F_{T})]$ and let $M_{1},\ldots,M_{n}$ be the connected components of $G[V(F_{T})]$. Now for each triangle $T \in H$ let $G'$ be the graph obtained by contracting each connected component, $M_{i}$, down to a vertex, call it $v^{i}_{T}$, $i \in \{1,\ldots,n\}$.

 Now consider some triangle $T \in H$ such that $V(T) = \{x_{1},x_{2},x_{3}\}$. First suppose there exists an $i \in \{1,\ldots,n\}$ such that $v^{i}_{T}$ is adjacent to $x_{j}$ for all $j \in \{1,2,3\}$. Then for all $v^{k}_{T}$, $k \in \{1,\ldots,n\}, k \neq i$, contract $v^{k}_{T}$ to any vertex of $T$.
 
 Now suppose that there was no $i \in \{1,\ldots,n\}$ such that $v^{i}_{T}$ is adjacent to $x_{j}$ for all $j \in \{1,2,3\}$. Since $G$ is $2$-connected, that means that for all $i \in \{1,\ldots,n\}$, $v^{i}_{T}$ is adjacent to exactly two of $x_{1},x_{2}$ and $x_{3}$. Let $v^{i}_{T}, v^{j}_{T}, v^{k}_{T}$ be vertices such that $v^{i}_{T}$ is adjacent to $x_{1},x_{2}$, and $v^{j}_{T}$ is adjacent to $x_{1},x_{3}$, and $v^{k}_{T}$ is adjacent to $x_{2},x_{3}$ for $i,j,k \in \{1,\ldots,n\}$, $i \neq j \neq k$ (note such vertices may not exist). Then for all $v^{l}_{T}$, $l \neq i,j,k$, contract $v^{l}_{T}$ to an arbitrary vertex of $T$. Let $G'$ be the resulting graph after applying the above procedure to every triangle $T \in H$ to $G$. We note that some subset of the vertices $v^{i}_{T}, v^{j}_{T}, v^{k}_{T}$ may not exist, but in this case we just do not have that subset of vertices in $G'$ . We claim that $G'$ is planar and $2$-connected.
 
First we show that $G'$ is $2$-connected. Notice that since $G$ is $2$-connected, for every $T \in H$, all of the $v^{i}_{T}$ have $2$ internally disjoint paths to every other vertex. Now consider two vertices $x,y \in V(H)$ and let $P_{1},P_{2}$ be two internally disjoint $(x,y)$-paths in $G$. Notice that for each triangle $T \in H$, at most one of $P_{1}$ or $P_{2}$ uses vertices from any connected component of $G[V(F_{T})]$, so these paths exist in $G'$ by possibly augmenting them to the appropriate vertex $v^{i}_{T},v^{j}_{T}$ or $v^{k}_{T}$. Therefore $G'$ is $2$-connected. 

So it suffices to show that $G'$ is planar. Take any planar embedding of $H$, and for every face bounded by a triangle $T$, either add a vertex adjacent to all of the vertices of $T$ to the interior of the face, or add three vertices $v_{1},v_{2},v_{3}$ to the interior of the face such that for all $i \in \{1,2,3\}$, $v_{i}$ is adjacent to two vertices of $T$, and $N(v_{i}) \neq N(v_{j})$ if $i \neq j$. Note that  the resulting graph is planar. Furthermore, using this construction, we can obtain a planar graph $K$ such that $G'$ is a subgraph of $K$, and thus $G'$ is planar. 

Now since $G'$ is $2$-connected, planar, and by construction we did not contract any of $a,b,c$ or $d$ together, we can apply Lemma \ref{webcycleplanar}. Thus $G'$ has a cycle $C'$ containing $X$. But then $G$ has a cycle $C$ containing $X$, obtained by extending $C'$ along the contracted edges, if necessary. 
\end{proof}

We note it is easy to show that none of the non-web $K_{4}(X)$-free classes have a cycle containing $X$, so the above lemma cannot be extended. Before we state our characterization, we give some definitions.

A common idea which appears in the study of graph minors is the notion of a $k$-dissection, which is simply a sequence of nested $k$-separations. Formally, a sequence $((A_{1},B_{1}), \ldots,(A_{n},B_{n}))$ is a \textit{$k$-dissection} if for all $i \in \{1,\ldots,n\}$,  $(A_{i},B_{i})$ is a $k$-separation, and for all $i \neq n$,  $A_{i} \subseteq A_{i+1}$, and $B_{i+1} \subseteq B_{i}$. We will use special types of $k$-dissections. 

\begin{definition}
Let $G$ be a $2$-connected graph and $X = \{a,b,c,d\} \subseteq V(G)$. Let $((A_{1},B_{1}), \ldots,(A_{n},B_{n}))$ be a $2$-dissection for some integer $n$. If $A_{i} \cap  B_{i} \cap  A_{i+1} \cap B_{i+1} \neq \emptyset$  for all $i \in \{1,\ldots, n-1\}$, then we will say the $2$-dissection is a \emph{$2$-chain}. Let $((A_{1},B_{1}),\ldots,(A_{n},B_{n}))$ be a $2$-chain. Suppose both $A_{1} \cap B_{1}$ and $A_{n} \cap B_{n}$ contain at least one vertex of $X$, and there is exactly one vertex of $X$ in $A_{1} \setminus B_{1}$, and exactly one vertex of $X$ in $B_{n} \setminus A_{n}$. Then we say $((A_{1},B_{1}),\ldots, (A_{n},B_{n}))$ is a \emph{terminal separating $2$-chain}. 
\end{definition}

Additionally, we will require the idea of a ``triangle" of separations. 

\begin{definition}
Let $G$ be a $2$-connected graph and $X = \{a,b,c,d\} \subseteq V(G)$. Suppose there are three distinct $2$-separations $(A_{1},B_{1}),(A_{2},B_{2}),(A_{3},B_{3})$. We say these separations form a \emph{triangle} if $A_{1} \cap B_{1} = \{x,y\}$, $A_{2} \cap B_{2} = \{x,v\}$ and $A_{3} \cap B_{3} = \{v,y\}$ for distinct vertices $x,y,v \in V(G)$. For notational convenience, we will enforce that in a triangle, $(A_{i} \setminus B_{i}) \cap (A_{j} \setminus B_{j}) = \emptyset$ for any $i,j \in \{1,2,3\}$, $i \neq j$.  We say a triangle $(A_{1},B_{1}),(A_{2},B_{2}),(A_{3},B_{3})$, is \emph{terminal separating} if exactly two vertices of $X$ are contained in $A_{1}$, exactly one vertex of $X$ is contained in $A_{2} \setminus B_{2}$ and exactly one vertex of $X$ is contained in $A_{3} \setminus  B_{3}$. Two triangles  $((A^{1}_{1},B^{1}_{1}), (A^{1}_{2},B^{1}_{2}), (A^{1}_{3},B^{1}_{3}))$, $((A^{2}_{1},B^{2}_{1}),(A^{2}_{2},B^{2}_{2}),(A^{2}_{3},B^{2}_{3}))$, are distinct if there exists an $i \in \{1,2,3\}$, such that for $A^{2}_{i}$, $A^{1}_{j} \cap B^{1}_{j} \subseteq A^{2}_{i}$ for all $j \in \{1,2,3\}$ and there exists an $i \in \{1,2,3\}$ such that for  $A^{1}_{i}$, $A^{2}_{j} \cap B^{2}_{j} \subseteq A^{1}_{i}$ for all $j \in \{1,2,3\}$.  
\end{definition}

Now we can state the characterization.

\begin{figure}
\begin{center}
\includegraphics[scale=0.5]{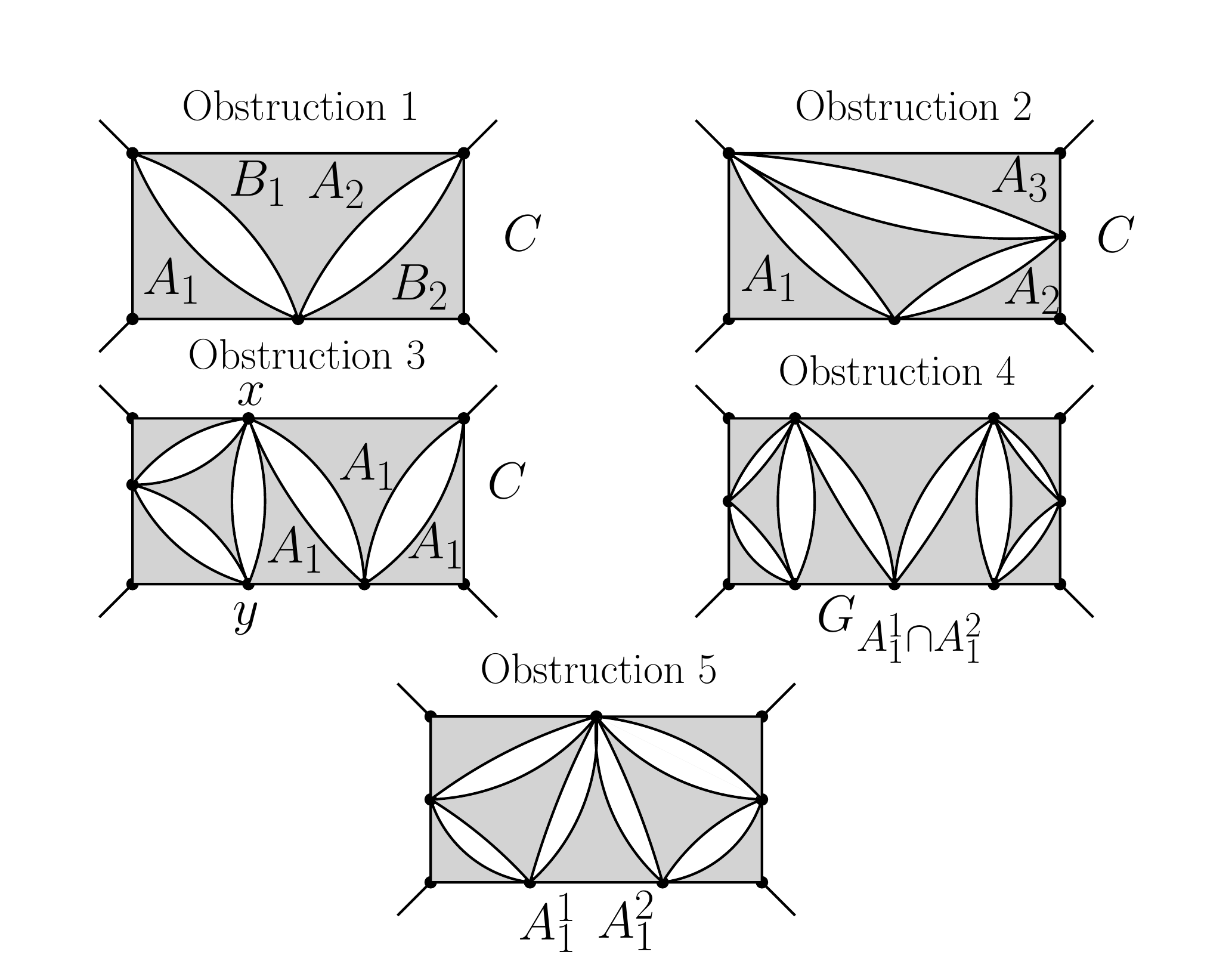}
\caption{The obstructions for Theorem \ref{w4cuts}. The vertices of $X$ are represented by vertices which have an edge not adjacent to a vertex. Curved lines represent a $2$-separation. Sections not apart of triangles are apart of a $2$-chain.}
\label{obstructionset}
\end{center}
\end{figure}

\begin{theorem}
\label{w4cuts}
Let $G$ be a $2$-connected graph and $X = \{a,b,c,d\} \subseteq V(G)$. Suppose $G$ is a spanning subgraph of an $\{a,b,c,d\}$-web. Then $G$ is $W_{4}(X)$-minor free if and only if for every cycle $C$ where $X \subseteq V(C)$, we have one of the following obstructions (see Figure \ref{obstructionset}).  
\begin{enumerate}
\item{There is a terminal separating $2$-chain $((A_{1},B_{1}),\ldots,(A_{n},B_{n}))$ such that $A_{i} \cap B_{i} \subseteq V(C)$ for all $i \in \{1,\ldots, n\}$.}
\item{There is a terminal separating triangle $(A_{1},B_{1}),(A_{2},B_{2}),(A_{3},B_{3})$, such that $A_{1} \cap B_{1}$ contains a vertex from $X$ and $A_{i} \cap B_{i} \subseteq V(C)$, for all $i \in \{1,2,3\}$.}
\item{There is a terminal separating triangle $(A_{1},B_{1}),(A_{2},B_{2}),(A_{3},B_{3})$ such that $A_{1} \cap B_{1} = \{x,y\}$, where $x,y \not \in X$. Furthermore, the graph $G_{A_{1}} = G[A_{1}] \cup \{xy\}$ and $C_{A} = G[V(C) \cap A] \cup \{xy\}$  has a terminal separating $2$-chain $((A'_{1},B'_{1}),\ldots,(A'_{n},B'_{n}))$ where we let $x$ and $y$ replace the two vertices in $X$ from $G$ not in $G_{A_{1}}$. Additionally, $A_{i} \cap B_{i} \subseteq V(C)$, for all $i \in \{1,2,3\}$, and $A'_{i} \cap B'_{i} \subseteq V(C_{A})$ for all $i \in \{1,\ldots,n\}$.}
\item{There are two distinct terminal separating triangles $((A^{1}_{1},B^{1}_{1}),(A^{1}_{2},B^{1}_{2}), (A^{1}_{3},B^{1}_{3}))$, and $((A^{2}_{1},B^{2}_{1}),(A^{2}_{2},B^{2}_{2}),\\ (A^{2}_{3},B^{2}_{3}))$  where for all $i \in \{1,2,3\}$,  $A^{1}_{i} \cap B^{1}_{i} \subseteq A^{2}_{1}$ and $A^{2}_{i} \cap B^{2}_{i} \subseteq A^{1}_{3}$. Furthermore, consider the graph $G_{A^{2}_{1} \cap A^{1}_{1}} = G[A^{2}_{1} \cap A^{1}_{1}] \cup \{xy | x,y \in A^{i}_{1} \cap B^{i}_{1}, i \in \{1,2\}\}$ and the cycle $C' = G[V(C) \cap A^{2}_{1} \cap A^{1}_{1}] \cup \{xy | x,y \in A^{i}_{1} \cap B^{i}_{1}, i \in \{1,2\}\}$. Let $X'$ be defined to be the vertices $A^{2}_{1} \cap B^{2}_{1}$ and $A^{2}_{3} \cap B^{2}_{3}$. Then there is a terminal separating $2$-chain with respect to $X'$, $((A_{1},B_{1}),\ldots,(A_{n},B_{n}))$, in $G_{A^{2}_{1} \cap A^{1}_{1}}$ such that $A_{i} \cap B_{i} \subseteq V(C')$.}
\item{There are $2$ distinct terminal separating triangles $((A^{1}_{1},B^{1}_{1}), (A^{1}_{2},B^{1}_{2}), (A^{1}_{3},B^{1}_{3})),\\  ((A^{2}_{1},B^{2}_{1}), (A^{2}_{2},B^{2}_{2}),(A^{2}_{3},B^{2}_{3}))$ where for all $i \in \{1,2,3\}$,  $A^{1}_{i} \cap B^{1}_{i} \subseteq A^{2}_{1}$  and $A^{2}_{i} \cap B^{2}_{i} \subseteq A^{1}_{1}$, the set $A^{2}_{1} \cap B^{2}_{1} \cap A^{1}_{1} \cap B^{1}_{1}$ is not empty and $A^{j}_{i} \cap B^{j}_{i} \subseteq V(C)$ for all $i \in \{1,2,3\}$, and $j \in \{1,2\}$.}
\end{enumerate}
\end{theorem}

Before proving this, we prove some lemmas to make the proof cleaner.

\begin{lemma}
\label{obstructionsareobstructions}
Let $G$ be a $2$-connected graph, $X = \{a,b,c,d\} \subseteq V(G)$. Let $C$ be a cycle in $G$ such that $X \subseteq V(C)$. If any of the obstructions in Theorem \ref{w4cuts} occur, then $G$ does not have a $W_{4}(X)$-minor.
\end{lemma}

\begin{proof}
We deal with each case separately. In each case we suppose $G$ is a minimal counterexample with respect to the number of vertices. 

\textbf{Case 1:} Suppose we have a terminal separating $2$-chain $((A_{1},B_{1}),\ldots,(A_{n},B_{n}))$. If $n=1$, then $(A_{1},B_{1})$ is a $2$-separation satisfying the conditions in Lemma \ref{noW4minor} and thus $G$ does not have a $W_{4}(X)$-minor. Therefore we assume $n \geq 2$. Then $(A_{1},B_{1})$ satisfies the conditions in Lemma \ref{oneterminalinthecut2conn}. Let $G'$ be the graph obtained after applying Lemma \ref{oneterminalinthecut2conn} to $(A_{1},B_{1})$. Then $G$ has a $W_{4}(X)$-minor if and only if $G'$ has a $W_{4}(X_{1})$-minor, where $X_{1}$ is defined from Lemma \ref{oneterminalinthecut2conn}. Notice in $G'$, $((A_{2},B_{2}),\ldots,(A_{n},B_{n}))$ is a terminal separating $2$-chain satisfying the properties of obstruction $1$ for the cycle $G'[V(C) \cap B_{1}]$ when we replace the vertex of $X$ in $A_{1} \setminus B_{1}$ with the vertex in $(A_{1} \cap B_{1}) \setminus X$. Since $G$ is a vertex minimal counterexample, $G'$ has no $W_{4}(X)$-minor, and thus $G$ has no $W_{4}(X)$-minor.

\textbf{Case 2:} Suppose there is a terminal separating triangle $(A_{1},B_{1}),(A_{2},B_{2}),(A_{3},B_{3})$ satisfying properties in the second obstruction.  Then we apply Lemma \ref{oneterminalinthecut2conn} to $(A_{1},B_{1})$ giving a new graph $G'$ which has a $W_{4}(X)$-minor if and only if $G'$ has a $W_{4}(X_{1})$-minor, where $X_{1}$ is defined from Lemma \ref{oneterminalinthecut2conn}. In $G'$, apply Lemma \ref{oneterminalinthecut2conn} to $(A_{2},B_{2})$ giving a graph $G''$. Then $G''$ has a $W_{4}(X_{2})$-minor if and only if $G$ has a $W_{4}(X)$-minor. Observe in $G''$,  $(A_{3},B_{3})$ is a separation satisfying Lemma \ref{noW4minor} and thus $G''$ does not have a $W_{4}(X_{2})$-minor. But $G''$ has a $W_{4}(X)$-minor if and only if $G$ has a $W_{4}(X)$-minor, so therefore $G$ has no $W_{4}(X)$-minor.

\textbf{Case 3:} Suppose there is a terminal separating triangle $(A_{1},B_{1}),(A_{2},B_{2}),(A_{3},B_{3})$ and a terminal separating $2$-chain $((A'_{1},B'_{1}),\ldots (A'_{n},B'_{n}))$ in the graph $G_{A_{1}}$, as in the third obstruction. Apply Lemma \ref{oneterminalonesideothersother} to $(A_{2},B_{2})$ to obtain two new reduced graphs $G'_{1}$ and $G'_{2}$. Now in one of $G'_{1}$ and $G'_{2}$, we can apply Lemma \ref{oneterminalinthecut2conn} to $(A_{1},B_{1})$ to obtain a graph $G''$, and in one of $G'_{1}$ and $G'_{2}$, we can apply Lemma \ref{oneterminalinthecut2conn} twice to $(A_{3},B_{3})$ and $(A_{1},B_{1})$ to obtain the graph $G''$ (note that the graph $G''$ obtained from both $G'_{1}$ and $G'_{2}$ is indeed the same graph). Then $G''$ has a $W_{4}(X)$-minor if and only if $G$ has a $W_{4}(X)$-minor. Notice that in the graph $G''$, $((A'_{1},B'_{1}),
\ldots,(A'_{n},B'_{n}))$ is a terminal separating $2$-chain satisfying obstruction $1$. Then by case one, $G''$ has no $W_{4}(X)$-minor, and thus $G$ has no $W_{4}(X)$-minor.

\textbf{Case 4:} Suppose we have the fourth obstruction in Theorem \ref{w4cuts} and let $(A_{1},B_{1}), (A_{2},B_{2}),(A_{3},B_{3})$ be one of the terminal separating triangles. If we apply Lemma \ref{oneterminalonesideothersother} and Lemma \ref{oneterminalinthecut2conn} to $(A_{1},B_{1}),(A_{2},B_{2}),(A_{3},B_{3})$ as we did in case $3$, we obtain a graph $G'$ which has a $W_{4}(X)$-minor if and only if $G$ does. Furthermore, in the graph $G'$, the other terminal separating triangle and terminal separating $2$-chain given by the fourth obstruction for $G$ and $C$ satisfy the properties of obstruction $3$. But then by case $3$, $G'$ has no $W_{4}(X)$-minor, and thus $G$ has no $W_{4}(X)$-minor. 

\textbf{Case 5:} Let $(A^{1}_{1},B^{1}_{1}), (A^{1}_{2},B^{1}_{2}), (A^{1}_{3},B^{1}_{3})$ and $(A^{2}_{1},B^{2}_{1}), (A^{2}_{2},B^{2}_{2}),(A^{2}_{3},B^{2}_{3})$ be the two distinct triangles satisfying the properties in obstruction five. Applying Lemma \ref{oneterminalonesideothersother} and Lemma \ref{oneterminalinthecut2conn} to $(A^{1}_{1},B^{1}_{1}), (A^{1}_{2},B^{1}_{2}), (A^{1}_{3},B^{1}_{3})$ as in case $3$, we obtain a graph $G'$ which has a $W_{4}(X)$-minor if and only if $G$ has a $W_{4}(X)$-minor. If $A^{1}_{1} \cap B^{1}_{1} = A^{2}_{1} \cap B^{2}_{1}$, then in $G'$, the separations $(A^{2}_{2}, B^{2}_{2})$ and $(A^{2}_{3},B^{2}_{3})$ are a terminal separating $2$-chain in $G'$ satisfying obstruction $1$ on the cycle $C' = G'[V(C) \cap B^{1}_{1}]$. Then by case $1$, $G'$ has no $W_{4}(X)$-minor. Therefore we assume that $A^{1}_{1} \cap B^{1}_{1} \neq A^{2}_{1} \cap B^{2}_{1}$. Then in $G'$ on the cycle $C'$, $(A^{2}_{1},B^{2}_{1}), (A^{2}_{2},B^{2}_{2}),(A^{2}_{3},B^{2}_{3})$ is a terminal separating triangle, so by case $2$, $G'$ has no $W_{4}(X)$-minor, and thus $G$ has no $W_{4}(X)$-minor.    
\end{proof}

\begin{lemma}
\label{nicesplittinglemma}
Let $G$ be a $2$-connected graph, $X = \{a,b,c,d\} \subseteq V(C)$. Let $C$ be a cycle in $G$ such that $X \subseteq V(C)$. Suppose that $a,b,c,d$ appear in that order on $C$, and suppose that $cd \in E(C)$. Then,
\begin{itemize}
  \item{If there is a terminal separating triangle satisfying the properties of obstruction $2$, then  $A_{1} \cap B_{1}$ contains exactly one of vertices $c$ or $d$.}
  \item{If $C$ has a terminal separating triangle and a terminal separating $2$-chain $((A_{1},B_{1}),\ldots, (A_{n},B_{n}))$ as in obstruction $3$, then either $A_{1} \cap B_{1}$ contains $c$ or $d$ or $A_{n} \cap B_{n}$ contains $c$ or $d$.}
  \item{ Obstructions $4$ and $5$ do not occur on $C$.}
  \end{itemize}  
\end{lemma}

\begin{proof}

Let $P_{a,b}$ be the $(a,b)$-path on $C$ such that $c,d \not \in V(P_{a,b})$. Similarly define $P_{b,c}$, $P_{c,d}$ and $P_{d,a}$. Then by construction, $P_{a,b} \cup P_{b,c} \cup P_{c,d} \cup P_{d,a} = C$. 

We first show that if we have a terminal separating triangle $(A_{1},B_{1}),(A_{2},B_{2}),(A_{3},B_{3})$ as in obstruction $2$, that $A_{1} \cap B_{1}$ contains one of $c$ or $d$.

 If $A_{1} \cap B_{1}$ contains $c$ or $d$ we are done. Therefore we assume that $A_{1} \cap B_{1}$ contains $a$. Let $(A_{1} \cap B_{1}) \setminus \{a\} = v$. Then to satisfy the definition of a terminal separating triangle we have that $v \in V(P_{b,c})$ or $v \in V(P_{c,d})$. This follows since if $v$ lies on either of $P_{a,b} \setminus \{a,b\}$ or $P_{d,a} \setminus \{d,a\}$, then either there would not be two exactly two vertices of $X$ contained in $A_{1}$, or we could not satisfy the condition that $A_{j} \not \subseteq A_{i}$ for all $i,j \in \{1,2,3\}$, $i \neq j$ and maintain that $A_{2} \setminus  B_{2}$ and $A_{3} \setminus B_{3}$ both contain a vertex of $X$. 

If $v=b$, then notice that the vertices of $A_{2} \cap B_{2}$ and $A_{3} \cap B_{3}$ that are not $a$ or $b$ lie on $P_{c,d}$. Since $cd \in E(C)$, without loss of generality $A_{2} \cap B_{2} = \{b,c\}$ and $A_{3} \cap B_{3} = \{a,c\}$. But then we do not have a terminal separating triangle, a contradiction. 

Therefore $v \neq b$. If $v = c$ or $v = d$ we are done. Since $cd \in E(C)$ it suffices to consider the case when $v \in P_{b,c} \setminus \{b,c\}$. Suppose $A_{2} \cap B_{2} = \{v,x\}$. If $x \not \in V(P_{a,d})$, then either $A_{2} \setminus B_{2}$ does not contain a vertex of $X$, or $A_{3} \setminus  B_{3}$ does not contain a vertex of $X$ which is a contradiction. If $x = d$, then $A_{3} \cap B_{3} = \{a,d\}$ which implies that $A_{3} \setminus B_{3}$ does not contain a vertex from $X$, a contradiction. Therefore $A_{1} \cap B_{1}$ cannot contain $a$. Mirroring the above argument, $A_{1} \cap B_{1}$ cannot contain $b$, and thus since by definition $A_{1} \cap B_{1}$ contains a vertex of $X$, it contains $c$ or $d$. 

Now suppose we have a terminal separating triangle $(A^{1}_{1},B^{1}_{1}), (A^{1}_{2}, B^{1}_{2}), (A^{1}_{3}, B^{1}_{3})$ and a terminal separating $2$-chain $((A_{1},B_{1}),\ldots,(A_{n},B_{n}))$ in $G_{A_{1}}$ satisfying the properties of obstruction $3$. We will show that either $A_{1} \cap B_{1}$ or $A_{n} \cap B_{n}$ contains $c$ or $d$. 

 Since $A^{1}_{1} \cap B^{1}_{1} = \{x,y\}$ and $x,y \not \in X$, to be a terminal separating triangle, for all $i \in \{1,2,3\}$,  $A^{1}_{i} \cap B^{1}_{i}$ does not contain a vertex from $X$. This follows from the definition of terminal separating triangle as two vertices of $X$ lie in $A^{1}_{1} \setminus  B^{1}_{1}$, so if one of $A^{1}_{i} \cap B^{1}_{i}$, for $i \in \{2,3\}$ contained a vertex of $X$, then at least one of $A^{1}_{2} \setminus B^{1}_{2}$ or $A^{1}_{3} \setminus B^{1}_{3}$ does not contain a vertex from $X$, a contradiction. Now notice that since $A^{1}_{1} \setminus \{x,y\}$ contains two vertices of $X$, and $cd \in E(C)$, either $c,d \in A^{1}_{1} \setminus \{x,y\}$ or $a,b \in A^{1}_{1} \setminus \{x,y\}$. If $a,b \in A^{1}_{1} \setminus \{x,y\}$ then to remain a terminal separating triangle, one of $A^{1}_{2} \cap B^{1}_{2}$ or $A^{1}_{3} \cap B^{1}_{3}$ contains a vertex from $P_{c,d}$. But since $cd \in E(C)$, and we know that for all $i \in \{1,2,3\}$,  no vertex of $X$ is contained in $A^{1}_{i} \cap B^{1}_{i}$, a contradiction. Therefore $c,d \in A^{1}_{1} \setminus \{x,y\}$. Then by definition of terminal separating $2$-chain and since $a,b,c,d$ appear in that order on $C$, the terminal separating $2$-chain in $G_{A_{1}}$ contains either $d$ or $c$ in $A_{n} \cap B_{n}$ or $A_{1} \cap B_{1}$.

Now suppose we have two distinct terminal separating triangles $(A^{1}_{1}, B^{1}_{1}), (A^{1}_{2},B^{1}_{2}), (A^{1}_{3}, B^{1}_{3})$ and $(A^{2}_{1},B^{2}_{1}), \\ (A^{2}_{2},B^{2}_{2}), (A^{2}_{3}, B^{2}_{3})$, and a terminal separating $2$-chain in $G[A_{1} \cap A_{2}]$ satisfying the properties of obstruction $4$. We will show that this obstruction does not exist since $cd \in E(C)$.

 Notice from the assumptions that $A^{i}_{1} \cap B^{i}_{1}$ does not contain any vertices from $X$ for $i \in \{1,2\}$. Then since $cd \in E(C)$, without loss of generality we may assume that $a,b \in A^{1}_{1}$ and $c,d \in A^{2}_{1}$. But then by the previous discussion, one of $A^{2}_{2} \cap B^{2}_{2}$ and $A^{2}_{3} \cap B^{2}_{3}$ contains a vertex from $P_{c,d}$. But then one of $A^{2}_{2} \cap B^{2}_{2}$ and $A^{2}_{3} \cap B^{2}_{3}$ contains $c$ or $d$, which implies that $(A^{2}_{1}, B^{2}_{1}), (A^{2}_{2},B^{2}_{2}), (A^{2}_{3}, B^{2}_{3})$ is not a terminal separating triangle, a contradiction. 

Finally, suppose that the fifth obstruction occurred. Then we have two distinct terminal separating triangles $((A^{1}_{1},B^{1}_{1}), (A^{1}_{2},B^{1}_{2}), (A^{1}_{3},B^{1}_{3})), ((A^{2}_{1},B^{2}_{1}),  (A^{2}_{2},B^{2}_{2}),(A^{2}_{3},B^{2}_{3}))$ satisfying the properties of obstruction $5$. Then from our assumptions and since $cd \in E(C)$, without loss of generality we may assume  $c,d \in A^{1}_{1} \setminus B^{1}_{1}$ and $a,b \in A^{2}_{1} \setminus B^{2}_{1}$. By the same argument as for the obstruction $4$ case, this gives a contradiction, completing the proof.    
\end{proof}

We also note a well known observation on the submodularity of separations. 

\begin{proposition}
Let $G$ be a graph with $2$-separations $(A_{1},B_{1})$ and $(A_{2},B_{2})$. Let $A_{1} \cap B_{1} = \{u,v\}$ and $A_{2} \cap B_{2} = \{x,y\}$ where $x,y,u$ and $v$ are distinct vertices. Furthermore, suppose that $x \in A_{1} \setminus B_{1}$ and $y \in B_{1} \setminus A_{1}$. Then let $u \in A_{2} \setminus  B_{2}$ and $v \in B_{2} \setminus A_{2} $. Then there is a separation $(A',B')$ such that $A' \cap B' = \{u,x\}$ and $A' \subseteq A_{1}$ and $B_{1} \subseteq B'$. 
\end{proposition}

\begin{proof}
Let $A' = A_{1} \cap A_{2}$ and $B' = V(G) \setminus (A_{1} \cap A_{2})$. Let $z$ be any vertex in $A_{1} \cap A_{2}$ and consider a path $P$ from $z$ to any vertex not in $A_{1} \cap A_{2}$. Consider the first vertex in $P$ which is not in $A_{1} \cap A_{2}$. If this vertex is in $B_{1} \cap A_{2}$, then since $u \in A_{2} \setminus B_{2}$, this vertex is $u$, and thus $u \in V(P)$. If this vertex is in $A_{1} \cap B_{2}$ then since $x \in A_{1} \setminus B_{1}$ this vertex is $x$ and thus $x \in V(P)$. Notice that these are the only options, and thus $(A',B')$ is a $2$-separation with $A' \cap B' = \{u,x\}$. Additionally, it is immediate that $A' \subseteq A_{1}$ and $B_{1} \subseteq B'$.
\end{proof}

Now we prove the theorem. Throughout the proof we will abuse the notation of separations slightly.  Suppose we have a graph $G$ and a $2$-separation $(A,B)$. Consider the graph  $G_{B} = G[B] \cup \{xy\}$ where $xy \in A \cap B$. If there is a $2$-separation $(A',B')$ in $G_{B}$, then we will refer to the $2$-separation $(A' \cup A,B)$ in $G$ as $(A',B')$ to avoid notational clutter. Thus essentially if a separation $(A',B')$ in $G_{B}$ induces a natural separation in $G$, then we refer to the separation in $G$ as $(A',B')$.

\begin{proof}[Proof of Theorem \ref{w4cuts}]
Lemma \ref{obstructionsareobstructions} proves one direction of the theorem.

For the other direction, consider a graph $G$ which is a minimal counterexample with respect to $|V(G)|$. That is, we consider a graph $G$ such that there exists a cycle $C$ such that $X \subseteq V(C)$ where none of the five above obstructions exist on the cycle $C$, and $G$ is $W_{4}(X)$-minor free. Note such cycle always exists by Corollary \ref{webcycle}. Without loss of generality let $a,b,c,d$ appear in that order on $C$. Let $P_{a,b}$ be the $(a,b)$-path on $C$ such that $c,d \not \in V(P_{a,b})$. Similarly define $P_{b,c},P_{c,d}$ and $P_{d,a}$.  The goal will be to show that $G$ must be $3$-connected, which contradicts Theorem \ref{3connectivityW4K4}. We go through all the different possibilities for where the vertices of $X$ can be in relation to a $2$-separation.

\textbf{Claim 1:} There is no $2$-separation $(A,B)$ such that $X \subseteq A$. 

Suppose we had such separation and first suppose $V(C) \subseteq A$. By applying Lemma \ref{allonesidegeneralH} to $(A,B)$, the graph $G$ has a $W_{4}(X)$-minor if and only if $G_{A}$ has a $W_{4}(X)$-minor. Since $G$ is a vertex minimal counterexample, $G_{A}$ has one of the obstructions on $C$. But then the obstruction exists in $G$, a contradiction. 

Therefore we can assume that there are vertices of $C$ in $B \setminus A$. Since $X \subseteq V(C)$ and $X \subseteq A$, all of the vertices in $V(C) \cap B$ lie in exactly one of $P_{a,b}$, $P_{b,c}$,  $P_{c,d}$ or $P_{d,a}$. Without loss of generality, suppose all the vertices in $V(C) \cap B$ lie on $P_{a,b}$. Furthermore, since there is a vertex in $B \setminus A$, this implies that $A \cap B \subseteq V(P_{a,b})$. Then after applying Lemma \ref{allonesidegeneralH} to $(A,B)$ notice that in the graph $G_{A}$, that $C_{A} = G_{A}[V(C) \cap A]$ is a cycle. As $G$ is a minimal counterexample, $C_{A}$ has one of the five obstructions. Notice that regardless of the obstruction, since all the vertices of $C$ that were in $B$ were on $P_{a,b}$, the obstruction for $C_{A}$ in $G_{A}$ exists in $G$ for $C$. But this is a contradiction. 

\textbf{Claim 2:} There is no $2$-separation $(A,B)$ such that two vertices of $X$ lie in $A \setminus B$ and two vertices of $X$ that lie in $B \setminus A$.

Suppose such separation existed and let $A \cap B = \{x,y\}$. Notice that for such a separation to exist we have that $x \in V(C)$ and $y \in V(C)$.  Consider the graphs $G_{A}$ and $G_{B}$ inherited from Lemma \ref{2sepW4} and let $C_{A}$ and $C_{B}$ be the cycles where $C_{A} = G_{A}[V(C) \cap A]$ and $C_{B} = G_{B}[V(C) \cap A]$. By minimality, both $C_{A}$ and $C_{B}$ have one of the five obstructions. We consider the various cases.
 
 \textbf{Case 1:} The cycle $C_{B}$ has a terminal separating $2$-chain as in obstruction $1$, say $((A^{2}_{1},B^{2}_{1}),\ldots,(A^{2}_{n},B^{2}_{n}))$.
 
 \textbf{Subcase 1:} The cycle $C_{A}$ has a terminal separating $2$-chain as in obstruction $1$, $((A^{1}_{1},B^{1}_{1}),\ldots,(A^{1}_{n},B^{1}_{n}))$.

  Then if $A^{1}_{n} \cap B^{1}_{n} \cap A^{2}_{1} \cap B^{2}_{1} \neq \emptyset$ we may concatenate the two terminal separating $2$-chains giving a terminal separating $2$-chain of $G$ on $C$ satisfying obstruction $1$. Therefore we can assume $A^{1}_{n} \cap B^{1}_{n}$ does not share a vertex with $A^{2}_{1} \cap B^{2}_{1}$. But then the terminal separating $2$-chain $((A^{1}_{1},B^{1}_{1}),\ldots,(A^{1}_{n},B^{1}_{n}),(A,B), (A^{2}_{1},B^{2}_{1}),\ldots,(A^{2}_{n},B^{2}_{n}))$  satisfies obstruction $1$ on $C$, a contradiction. 
 
 \textbf{Subcase 2:} The cycle $C_{A}$ has a terminal separating triangle, $(A^{1}_{1},B^{1}_{1}), (A^{1}_{2},B^{1}_{2}),(A^{1}_{3},B^{1}_{3})$, satisfying obstruction two.
 
  Since $xy \in E(C_{A})$, the terminal vertex in $A^{1}_{1} \cap B^{1}_{1}$ is either $x$ or $y$, by Lemma \ref{nicesplittinglemma}. But then either $(A^{2}_{1},B^{2}_{1}), (A^{2}_{2},B^{2}_{2}), (A^{2}_{3},B^{2}_{3})$ combines with the terminal separating $2$-chain in $G_{B}$ or the terminal separating $2$ chain plus the separation $(A,B)$ to form the third obstruction in $G$, a contradiction. 

\textbf{Subcase 3:}  The cycle $C_{A}$ has a terminal separating triangle $(A_{1},B_{1}),(A_{2},B_{2}),(A_{3},B_{3})$ and a terminal separating $2$-chain $((A^{1}_{1},B^{1}_{1}),\ldots,(A^{1}_{n},B^{1}_{n}))$ as in obstruction $3$. 

By Lemma \ref{nicesplittinglemma}, we can assume that either $x$ or $y$ is in $A^{1}_{1} \cap B^{1}_{1}$. But then either we can concatenate the two terminal separating $2$-chains such that there resulting terminal separating $2$-chain and the terminal separating triangle satisfy obstruction $3$, or we can add in the separation $(A,B)$ to the terminal separating $2$-chains such that the terminal separating $2$-chains plus $(A,B)$ plus the terminal separating triangle satisfy obstruction $3$. 

\textbf{Subcase 4:} Obstruction $4$ or $5$ occurs on $C_{A}$.

By Lemma \ref{nicesplittinglemma} the fourth and fifth obstructions do not occur in $G_{A}$, since $xy \in E(C_{A})$. Now we consider cases where $G_{B}$ does not have a terminal separating $2$-chain.

\textbf{Case 2:} The cycle $C_{B}$ has a terminal separating triangle $(A^{2}_{1},B^{2}_{1}),(A^{2}_{2},B^{2}_{2}),(A^{2}_{3},B^{2}_{3})$ satisfying obstruction $2$.
 
 Notice by Lemma \ref{nicesplittinglemma} the terminal vertex contained in $A^{2}_{1} \cap B^{2}_{1}$ is $x$ or $y$. 
 
 \textbf{Subcase 1:} The cycle $C_{A}$ has a terminal separating triangle $(A^{1}_{1},B^{1}_{1}),(A^{1}_{2},B^{1}_{2}),(A^{1}_{3},B^{1}_{3})$ satisfying obstruction $2$.
 
  Again, by Lemma \ref{nicesplittinglemma} the terminal vertex contained in $A^{1}_{1} \cap B^{1}_{1}$ is $x$ or $y$. If the terminal vertex in $A^{i}_{1} \cap B^{i}_{1}$ is $x$ for both $i =1,2$, then notice in $G$ the two triangles form obstruction five. A similar statement holds if they are both $y$. Then we must have that one of the triangles contains $x$ as the terminal vertex in $A^{i}_{1} \cap B^{i}_{1}$, $i \in \{1,2\}$ and the other contains $y$. But then the two triangles plus the separation $(A,B)$ form obstruction $4$ in $G$, a contradiction. 

\textbf{Subcase 2:} The cycle  $C_{A}$ has a terminal separating triangle and a terminal separating $2$-chain as in obstruction $3$. 

By Lemma \ref{nicesplittinglemma} the terminal separating $2$-chain contains one of the vertices $x$ or $y$. Then by possibly adding in the separation $(A,B)$ to the existing terminal separating $2$-chains, we can extend this to a triangle plus a terminal separating $2$-chain and another terminal separating triangle, as in obstruction $4$, a contradiction. 

\textbf{Case 3:} Both of the cycles $C_{A}$ and $C_{B}$ have obstruction $3$.

But then by Lemma \ref{nicesplittinglemma} the terminal separating $2$-chains in both $C_{A}$ and $C_{B}$ both contain $x$ or $y$, and thus after possibly adding in $(A,B)$, we get in $G$, the fourth obstruction exists on $C$, a contradiction.  

Obstruction $4$ and $5$ cannot occur, and therefore there is no $2$-separation $(A,B)$ such that two vertices of $X$ lie in $A \setminus  B$ and two vertices of $X$ that lie in $B \setminus A$, and $A \cap B \subseteq V(C)$. 
 
\textbf{Claim 3:}  There is no $2$-separation $(A,B)$ such that $b \in A \cap B$, $a \in A \setminus B$, and $c,d \in B \setminus A$. 

Suppose there is such a separation, $(A,B)$, and let $A \cap B = \{v,b\}$. First notice that since $C$ is a cycle containing $X$, for such separation to exist $v \in V(C)$. Applying Lemma \ref{oneterminalinthecut2conn} to $(A,B)$ we get that $G$ has a $W_{4}(X)$-minor if and only if $G_{B}$ has a $W_{4}(X)$-minor. Since we picked $G$ to be a minimal counterexample, $G_{B}$ has one of the five obstructions occuring on the cycle $C_{B} = G_{B}[V(C) \cap B]$. We consider the various cases.

\textbf{Case 1:} The cycle $C_{B}$ has a terminal separating $2$-chain $((A_{1},B_{1}),\ldots,(A_{n},B_{n}))$ satisfying obstruction $1$.

If $A_{1} \cap B_{1}$ or $A_{n} \cap B_{n}$ contains $b$ then it is a terminal separating $2$-chain in $G$ for $C$. Otherwise without loss of generality we have that $A_{1} \cap B_{1}$ contains $v$. But then $((A,B),(A_{1},B_{1}),\ldots,(A_{n},B_{n}))$ is a terminal separating $2$-chain in $G$ for $C$, contradicting that $G$ is a minimal counterexample. 
 
\textbf{Case 2:} The cycle $C_{B}$ has a terminal separating triangle $(A_{1},B_{1}), (A_{2},B_{2}),(A_{3},B_{3})$ satisfying obstruction $2$.

 First suppose that  $b \in A_{1} \cap B_{1}$. Then the terminal separating triangle exists in $G$, contradicting that we have a minimal counterexample. Then $v \in A_{1} \cap B_{1}$. But then in $G$, the separation $(A,B)$ plus $(A_{1},B_{1}), (A_{2},B_{2}),(A_{3},B_{3})$ satisfies obstruction $3$, contradicting that $G$ is a minimal counterexample. 

\textbf{Case 3:} The cycle $C_{B}$ has a terminal separating triangle plus a terminal separating $2$-chain as in obstruction $3$.

Then by Lemma \ref{nicesplittinglemma}, then the terminal separating $2$-chain contains either $v$ or $b$. In the case where the terminal separating $2$-chain contains $v$, then by adding in the separation $(A,B)$, we get a terminal separating $2$-chain plus terminal separating triangle in $G$. In the case where the terminal separating $2$-chain contains $b$, the terminal separating $2$-chain and terminal separating triangle were already an obstruction in $G$, a contradiction.

\textbf{Case 4:} Obstructions $4$ or $5$ occur on $C_{B}$. 

By Lemma \ref{nicesplittinglemma}, these obstructions do not occur, a contradiction. We have considered all possible cases, so therefore we can assume there is no $2$-separation $(A,B)$ such that $b \in A \cap B$, $a \in A \setminus B$, and $c,d \in B \setminus A$. 
 
 \textbf{Claim 4:} There are no $2$-separations of the form $(A,B)$ such that $A \cap B = \{u,v\}$, $a \in A \setminus \{u,v\}$, and $b,c,d \in B \setminus \{u,v\}$.

 Without loss of generality, we may assume that $u \in V(P_{a,b})$ and $v \in V(P_{a,d})$. By applying Lemma \ref{oneterminalonesideothersother} we get that $G$ has a $W_{4}(X)$-minor if and only if $G_{B}$ has either a $W_{4}(X_{1})$-minor or a $W_{4}(X_{2})$-minor, where $X_{1} = \{b,c,d,u\}$ and $X_{2} = \{b,c,d,v\}$. Since we have a minimal counterexample with respect to the number of vertices, we get that $C_{B} = G_{B}[V(C) \cap B]$ has one of the obstructions when we consider $X_{1}$ and when we consider $X_{2}$. First notice that if under either $X_{1}$ or $X_{2}$, we get obstruction $4$ or $5$, then the obstruction exists in $G$, a contradiction. We consider the various other cases.
 
 \textbf{Case 1:} Under $X_{2}$, we get a terminal separating $2$-chain  $((A_{1},B_{1}),\ldots,(A_{n},B_{n}))$.
 
 Notice that we may assume that $v \in A_{1} \cap B_{1}$ or $v \in A_{n} \cap B_{n}$, as otherwise $((A_{1},B_{1}),\ldots, (A_{n},B_{n}))$ exists in $G$ for $C$, a contradiction. Therefore without loss of generality, we suppose that $v \in A_{1} \cap B_{1}$. 
 
 \textbf{Subcase 1:} Under $X_{1}$, we get a terminal separating $2$-chain $((A'_{1},B'_{1}),\ldots,(A'_{n'},B'_{n'}))$. 
 
 By similar reasoning as above, we may assume that $u \in A'_{1} \cap B'_{1}$. Notice that the vertex in $A'_{1} \cap B'_{1} \setminus \{u\}$ lies in $P_{b,c} \setminus \{b\}$ or in $P_{d,c} \setminus \{d\}$ and and the vertex in $A_{1} \cap B_{1} \setminus \{v\}$ lies in $P_{b,c} \setminus \{b\}$ or $P_{d,c} \setminus \{d\}$. Consider the case where both the vertex in $A'_{1} \cap B'_{1} \setminus \{u\}$ and the vertex in $A_{1} \cap B_{1} \setminus \{v\}$ lie on $P_{b,c} \setminus \{b\}$.

We consider various subcases. Suppose the vertex in $A'_{1} \cap B'_{1} \setminus \{u\}$ is the same vertex as $A_{1} \cap B_{1} \setminus \{v\}$. If this vertex is $c$, then $(A,B), (B_{1},A_{1}), (A'_{1},B'_{1})$ is a terminal separating triangle in $G$ satisfying obstruction $2$, a contradiction. Thus we assume the vertex in $A'_{1} \cap B'_{1} \setminus \{u\}$ is not $c$. But then $(A,B), (A_{1},B_{1}), (A'_{1},B'_{1})$ plus $(A_{2},B_{2}),\ldots,(A_{n},B_{n})$ is a terminal separating triangle and terminal separating $2$-chain satisfying obstruction $3$, a contradiction. 
 
 Therefore we can assume that $A'_{1} \cap B'_{1} \setminus \{u\} \neq A_{1} \cap B_{1} \setminus \{v\}$. First suppose that the vertex in $A'_{1} \cap B'_{1} \setminus \{u\}$ lies in $A_{1}$. Notice that one of the vertices in $A'_{2} \cap B'_{2}$ lies on $P_{u,d}$, and call this vertex $x$  (if $x \not \in V(P_{u,d})$, then either $((A'_{1},B'_{1}),\ldots,(A'_{n},B'_{n}))$ is not a dissection, or not a terminal separating chain).  If  $x =v$, then $(A,B),(A'_{1},B'_{1}),(A'_{2},B'_{2})$ plus $((A_{1},B_{1}),\ldots (A_{n},B_{n}))$ satisfies obstruction $3$ in in $G$, a contradiction. Therefore $x \neq v$. But notice that $uv \in E(C_{B})$, which implies that $x$ lies on the $(v,d)$-subpath of $P_{u,d}$. But then the vertex in $A'_{1} \cap B'_{1} \setminus \{u\}$ and $v$ are a $2$-vertex cut. Let $(A',B')$ be the $2$-separation such that $A' \cap B' = \{u, (A'_{1} \cap B'_{1} \setminus \{u\})\}$. Then $(A,B),(A',B'),(A'_{1},B'_{1})$ plus $((A_{1},B_{1}),\ldots, (A_{n},B_{n}))$ satisfy obstruction $3$ in $G$, a contradiction. 
 
 Therefore we can assume that the vertex in $A'_{1} \cap B'_{1} \setminus \{u\}$ lies in $B_{1}$. Let $x$ be the vertex in $A_{1} \cap B_{1} \setminus \{v\}$. Notice that $x \neq c$ as the vertex in $A'_{1} \cap B'_{1} \setminus \{u\}$ lies in $B_{1}$. Since the vertex in $A'_{1} \cap B'_{1} \setminus \{u\}$ lies in $B_{1}$, we have that $\{x,u\}$ is a $2$-vertex cut. Let $(A',B')$ be this $2$-separation. Then $(A,B),(A',B'),(A_{1},B_{1})$ plus $((A_{2},B_{2}),\ldots (A_{n},B_{n}))$ is a terminal separating triangle plus terminal separating $2$-chain satisfying obstruction $3$ a contradiction. 
 
 Therefore we can assume that the vertex in  $A'_{1} \cap B'_{1} \setminus \{u\}$ and the vertex in $A_{1} \cap B_{1} \setminus \{v\}$ both do not lie on $P_{b,c} \setminus \{b\}$. By essentially the same argument, we can assume that the vertex in $A'_{1} \cap B'_{1} \setminus \{u\}$ and the vertex in $A_{1} \cap B_{1} \setminus \{v\}$ both do not lie on $P_{c,d} \setminus \{d\}$. 
 
 Now consider the case where the vertex in $A_{1} \cap B_{1} \setminus \{v\}$ lies in $P_{b,c} \setminus \{b\}$ and the vertex in $A'_{1} \cap B'_{1} \setminus \{u\}$ lies in $P_{c,d} \setminus \{d\}$. Notice that $c$ is not both in $A_{1} \cap B_{1} \setminus \{v\}$ and $A'_{1} \cap B'_{1} \setminus \{u\}$.  Let $x$ be the vertex in $A_{1} \cap B_{1} \setminus \{v\}$ and suppose $x \neq c$. Then $(A_{2},B_{2})$ exists and $A_{2} \cap B_{2} \setminus \{x\}$ lies on $P_{v,d}$. Then notice that $\{u,x\}$ is a $2$-vertex cut in $G$. Let $(A',B')$ be the separation such that $A' \cap B' = \{u,x\}$. 
Then $(A,B), (A_{1}, B_{1}), (A',B')$ plus $((A_{2},B_{2}),\ldots,(A_{n},B_{n}))$ is a terminal separating triangle and terminal separating chain as in obstruction $3$. The case where $c$ is not in $A'_{1} \cap B'_{1} \setminus \{u\}$ follows similarly.  
 
 Now consider the case where the vertex in $A_{1} \cap B_{1} \setminus \{v\}$ lies in $P_{c,d} \setminus \{d\}$ and the vertex in $A'_{1} \cap B'_{1} \setminus \{u\}$ lies in $P_{b,c} \setminus \{b\}$. Notice that at least one of the vertices in $A'_{1} \cap B'_{1} \setminus \{u\}$ and $A_{1} \cap B_{1} \setminus \{v\}$ is not $c$. Let $x$ be the vertex in $A_{1} \cap B_{1} \setminus \{v\}$ and suppose that $x \neq c$. Then since $x \neq c$, we have that $(A_{2},B_{2})$ exists and $A_{2} \cap B_{2} \setminus \{x\}$ lies on $P_{v,b}$. Then since $uv \in E(C_{B})$, we get that $\{u,x\}$ is a $2$-vertex cut. Let $(A',B')$ be the separation such that $A' \cap B' = \{u,x\}$. Then $(A,B),(A',B'), (A_{1},B_{1})$ plus $((A_{2},B_{2}),\ldots,(A_{n},B_{n}))$ is a terminal separating $2$-chain plus terminal separating triangle satisfying obstruction $3$, a contradiction. The case where $c$ is not in $A'_{1} \cap B'_{1} \setminus \{u\}$ follows similarly.

 \textbf{Subcase 2:} Suppose that under $X_{1}$, we get a terminal separating triangle $(A'_{1},B'_{1}), (A'_{2},B'_{2}),(A'_{3},B'_{3})$ satisfying obstruction $2$.

Suppose that the vertex in $A'_{1} \cap B'_{1} \setminus u$ lies in $A_{1}$.  Call this vertex $x$. First suppose $x \in A_{1} \cap B_{1} \setminus \{v\}$. Notice that $x \neq c$, as if $x =c$ then $(A'_{1},B'_{1}),(A'_{2},B'_{2}),(A'_{3},B'_{3})$ would not satisfy the definition of a terminal separating triangle, a contradiction. Therefore $x \neq c$, but then $(A,B), (A'_{1},B'_{1}),(A_{1},B_{1})$ and $((A_{2},B_{2}),\ldots,(A_{n},B_{n}))$ is a terminal separating $2$-chain satisfying obstruction $3$, a contradiction. 

Now consider when $x \in A_{1} \setminus  B_{1}$ and suppose that $x$ lies in $P_{b,c} \setminus \{b\}$. Then the vertex in $A'_{2} \cap B'_{2} \setminus \{x\}$ lies in either $P_{u,d} \setminus \{u,v\}$ or $P_{c,d}$. In either case, notice that $\{x,v\}$ is a $2$-vertex cut. Let $(A',B')$ be the separation such that $A' \cap B' = \{x,v\}$. But then $(A,B),(A',B'),(A'_{1},B'_{1})$ plus $((A_{1},B_{1}),\ldots,(A_{n},B_{n}))$ satisfies obstruction $3$ in $G$, a contradiction. Now suppose that $x$ lies in $P_{c,d} \setminus \{d\}$. Then since $x \in A_{1} \setminus B_{1}$, we have that $u$ and the vertex in $A_{1} \cap B_{1} \setminus \{v\}$ form a $2$-vertex cut. Let $(A',B')$ be the separation induced by this $2$-vertex cut. Then if $n=1$, we have $(A,B), (A',B'), (A_{1},B_{1})$ form a terminal separating triangle in $G$, a contradiction. Otherwise $n >1$ and $(A,B), (A',B'), (A_{1},B_{1})$ plus $((A_{2},B_{2}), \ldots, (A_{n},B_{n}))$ form a terminal separating triangle plus a terminal separating $2$-chain in $G$, a contradiction. 

Therefore we can assume $x$ lies in $B_{1} \setminus A_{1}$. Let $y$ be the vertex in $A_{1} \cap B_{1} \setminus \{v\}$. Then since $x \in B_{1} \setminus (A_{1} \cap B_{1})$,  notice that $u,y$ induces a $2$-separation. Let $(A',B')$ be the $2$-separation such that $A' \cap B' = \{u,y\}$. If $n=1$, then $(A,B), (A',B'), (A_{1},B_{1})$ is a terminal separating triangle satisfying obstruction $2$ in $G$, a contradiction. Otherwise, $(A,B),(A',B'),(A_{1},B_{1})$ plus $((A_{2},B_{2}),\ldots, (A_{n},B_{n}))$ satisfy obstruction $3$ in $G$, a contradiction. 

\textbf{Subcase 3:} Suppose that under $X_{1}$, we get a terminal separating triangle $(A'_{1},B'_{1}), (A'_{2},B'_{2}),(A'_{3},B'_{3})$ and a terminal separating $2$-chain $((A''_{1},B''_{1}), \ldots, (A''_{n},B''_{n}))$ satisfying obstruction $3$.

Notice that if $u \not \in A''_{1} \cap B''_{1}$ or $u \not \in A''_{n} \cap B''_{n}$, then the obstruction exists already in $G$, a contradiction. Therefore without loss of generality, $u \in A''_{1} \cap B''_{1}$. Now by the same arguments as in subcase $1$, we can look at the position of the vertex in $A''_{1} \cap B''_{1} \setminus \{u\}$ and show that some obstruction $3$ or $2$ always exists in $G$.

\textbf{Case 2:} Under $X_{2}$, we get a terminal separating triangle $(A_{1},B_{1}),(A_{2},B_{2}),(A_{3},B_{3})$ satisfying obstruction $2$.

Notice that $A_{1} \cap B_{1}$ contains $v$ as otherwise the obstruction is an obstruction of $G$, a contradiction. Then since $v \in A_{1} \cap B_{1}$, we may assume that no other vertex of $X_{2}$ is in $A_{1} \cap B_{1}, A_{2} \cap B_{2},$ or $A_{3} \cap B_{3}$.  Also note that $u$ is not in any of $A_{1} \cap B_{1}$, $A_{2} \cap B_{2}$, $A_{3} \cap B_{3}$ as otherwise $(A_{1},B_{1}),(A_{2},B_{2}),(A_{3},B_{3})$ would not be a terminal separating triangle. Then we may assume that $u \in A_{1} \setminus B_{1}$. Also note that we may assume the vertex in $A_{1} \cap B_{1} \setminus \{v\}$ lies in $P_{b,c}$. Notice by symmetry and the above cases, we do not need to consider the case where under $X_{1}$ we get a terminal separating $2$-chain.

\textbf{Subcase 1:} Under $X_{1}$, we get a terminal separating triangle, $(A'_{1},B'_{1}),(A'_{2},B'_{2}),(A'_{3},B'_{3})$, satisfying obstruction $2$.

By a similar argument as above, we may assume that $u \in A'_{1} \cap B'_{1}$ as otherwise the obstructions exists in $G$. Without loss of generality we may assume that the vertex in $A'_{1} \cap B'_{1} \setminus \{u\}$ lies in $P_{b,c} \setminus \{c\}$. Let $x$ be the vertex in $A'_{1} \cap B'_{1} \setminus \{u\}$. If $x \in A_{1} \cap B_{1} \setminus \{v\}$, then $(A,B), (A'_{1},B'_{1}),(A_{1},(B_{1})$ and $(A_{1},B_{1}),(A_{2},B_{2}),(A_{3},B_{3})$ satisfy obstruction $5$ in $G$, a contradiction.

Now suppose that $x$ is in $A_{1} \setminus B_{1}$. Notice that $A'_{2} \cap B'_{2} \setminus \{x\}$ lies on $P_{u,d}$. But then since $uv \in E(C_{B})$, $\{x,v\}$ is a $2$-vertex cut. Let $(A',B')$ be the separation where $A' \cap B' = \{x,v\}$. But then $(A,B), (A',B'), (A'_{1},B'_{1})$ and $(A_{1},B_{1}),(A_{2},B_{2}),(A_{3},B_{3})$ satisfy obstruction $5$ in $G$, a contradiction. 

Now suppose that $x$ lies in $B_{1} \setminus A_{1}$. Let $y \in A_{1} \cap B_{1} \setminus \{v\}$. Then $\{u,y\}$ is a $2$ vertex cut. Let $(A',B')$ be the separation such that $A' \cap B' = \{u,y\}$. Then $(A,B),(A',B'), (A_{1},B_{1})$ and $(A_{1},B_{1}),(A_{2},B_{2}),(A_{3},B_{3})$ are terminal separating triangles satisfying obstruction $5$ in $G$, a contradiction. 

Now suppose that $x$ is in $A_{1} \cap B_{1} \setminus \{v\}$. Then easily $(A,B),(A_{1},B_{1}),(A'_{1},B'_{1})$ and $(A'_{1},B'_{1}),(A'_{2},B'_{2}),(A'_{3},B'_{3})$ satisfies obstruction $5$ in $G$ a contradiction. 
 
\textbf{Subcase 2:} Suppose under $X_{1}$, we get a terminal separating triangle $(A'_{1},B'_{1}),(A'_{2},B'_{2}), (A'_{3},B'_{3})$ and terminal separating $2$-chain $((A''_{1},B''_{1}) \ldots (A''_{n'},B''_{n'}))$ as in obstruction $3$.

 Notice that if $u \not \in A''_{1} \cap B''_{1}$ and $u \not \in A''_{n} \cap B''_{n}$ then the obstruction exists in $G$, a contradiction. Thus without loss of generality, we may assume that $u \in A''_{1} \cap B''_{1}$. Then notice that the vertex in $A''_{1} \cap B''_{1} \setminus \{u\}$ lies on $P_{b,c} \setminus \{b\}$ or $P_{c,d} \setminus \{d\}$. Let $x$ be the vertex in $A''_{1} \cap B''_{1}$.

Now suppose that $x$ is in $A_{1} \setminus B_{1}$. Since we assumed that the other vertex in $A_{1} \cap B_{1} \setminus \{v\}$ was in $P_{b,c}$, this implies that $x \in P_{b,c}$. Then notice that since $c \not \in A_{1} \cap B_{1} \setminus \{v\}$, then either $(A''_{2},B''_{2})$ exists and the vertex in $A''_{2} \cap B''_{2} \setminus \{x\}$ lies on $P_{u,d}$ or $(A''_{2},B''_{2})$ does not exist and the vertex in $A'_{1} \cap B'_{1} \setminus \{x\}$ lies in $P_{u,d}$. First suppose the vertex in $A''_{2} \cap B''_{2} \setminus \{x\}$ lies on $P_{u,d}$. But then since $uv \in E(C_{B})$, $\{x,v\}$ is a $2$-vertex cut. Let $(A',B')$ be the separation where $A' \cap B' = \{x,v\}$. But then $(A,B), (A',B'), (A'_{1},B'_{1})$ and $(A_{1},B_{1}),(A_{2},B_{2}),(A_{3},B_{3})$ satisfy obstruction $5$ in $G$, a contradiction. A similar analysis holds for the case when $(A''_{2},B''_{2})$ does not exist and the vertex in $A'_{1} \cap B'_{1} \setminus \{x\}$ lies in $P_{u,d}$.

Now suppose that $x$ lies in $B_{1} \setminus (A_{1} \cap B_{1})$. Let $y \in A_{1} \cap B_{1} \setminus \{v\}$. Then $\{u,y\}$ is a $2$ vertex cut. Let $(A',B')$ be the separation such that $A' \cap B' = \{u,y\}$. Then $(A,B),(A',B'), (A_{1},B_{1})$ and $(A_{1},B_{1}),(A_{2},B_{2}),(A_{3},B_{3})$ are terminal separating triangles satisfying obstruction $5$, a contradiction.

 \textbf{Case 3:} Under $X_{2}$, we have a terminal separating $2$-chain $((A'_{1},B'_{1}),\ldots,(A'_{n},B'_{n}))$ and a terminal separating triangle $(A_{1}, B_{1}),(A_{2},B_{2}),(A_{3},B_{3})$. 
 
 Notice that if $v \not \in A'_{1} \cap B'_{1}$ and $v \not \in A'_{n} \cap B'_{n}$ then the obstruction exists in $G$, a contradiction. Without loss of generality, assume that $v \in A'_{1} \cap B'_{1}$. Notice that the vertex in $v \in A'_{1} \cap B'_{1} \setminus \{v\}$ lies in either $P_{b,c}$ or $P_{c,d}$. By symmetry and the previous cases, we only need to consider the case when under $X_{1}$ we get obstruction $3$.

\textbf{Subcase 1:} Under $X_{1}$ we get a terminal separating $2$ chain $((A''_{1},B''_{1}), \ldots (A''_{n},B''_{n}))$ and a terminal separating triangle $(A'''_{1},B'''_{1}),(A'''_{2},B'''_{2}),(A'''_{3},B'''_{3})$ satisfying obstruction $3$. 

By the same arguments as above, without loss of generality we may assume that $u \in A''_{1} \cap B''_{1}$. Then let $x$ be the vertex in $A''_{1} \cap B''_{1} \setminus \{u\}$. Notice that $x$ lies in $P_{b,c} \setminus \{b\}$ or $P_{c,d} \setminus \{d\}$. Consider the case where $x$ lies on $P_{b,c} \setminus \{b\}$ and the vertex in $A'_{1} \cap B'_{1} \setminus \{v\}$ lies on $P_{b,c} \setminus \{b\}$.

First suppose that $x \in A_{1} \cap B_{1} \setminus \{v\}$. Then if $n =1$, $(A,B),(A''_{1},B''_{1}),(A'_{1},B'_{1})$ and $(A_{1}, B_{1}),(A_{2},B_{2}),(A_{3},B_{3}))$ satisfies obstruction $5$. Otherwise $n \geq 1$ and $(A,B),(A''_{1},B''_{1}), \\ (A'_{1},B'_{1})$ plus $((A'_{2},B'_{2}), \ldots, (A'_{n},B'_{n}))$ plus $(A_{1}, B_{1}),(A_{2},B_{2}),(A_{3},B_{3})$ satisfy obstruction $4$ in $G$, a contradiction. 

Now suppose that $x \in A'_{1} \setminus B'_{1}$. Then either the vertex in  $A''_{2} \cap B''_{2} \setminus \{x\}$ lies on $P_{u,d}$ or $(A''_{2},B''_{2})$ does not exist and the vertex in $A'''_{1} \cap B'''_{1} \setminus \{x\}$ lies on $P_{u,d}$. Consider the case where $A''_{2} \cap B''_{2} \setminus \{x\}$ lies on $P_{u,d}$. Then since $uv \in E(C_{B})$, we have that $\{x,v\}$ is a $2$-vertex cut. Let $(A',B')$ be the separation where $A' \cap B' = \{x,v\}$. But then $(A,B), (A',B'), (A''_{1},B''_{1})$ plus $((A_{1},B_{1}),\ldots,(A_{n},B_{n}))$ plus $(A'_{1}, B'_{1}),(A'_{2},B'_{2}),(A'_{3},B'_{3})$ satisfies obstruction $4$ in $G$, a contradiction. 

Now suppose that $x \in B'_{1} \setminus A'_{1}$. Let $y \in A'_{1} \cap B'_{1} \setminus \{v\}$. Then notice that $\{y,u\}$ is a $2$-vertex cut in $G_{B}$. Let $(A',B')$ be the separation such that $A' \cap B' = \{y,u\}$. Then if $n =1$,  $(A,B), (A',B'),(A'_{1},B'_{1})$ plus $(A_{1},B_{1}),(A_{2},B_{2}),(A_{3},B_{3})$ satisfies obstruction $5$ in $G$, a contradiction. Otherwise, $n \geq 2$, and $(A,B), (A',B'),(A'_{1},B'_{1})$ plus $((A'_{2},B'_{2}),\ldots,(A'_{n},B'_{n}))$ plus $(A_{1},B_{1}),(A_{2},B_{2}),(A_{3},B_{3})$ satisfies obstruction $4$ in $G$, a contradiction. 

Notice the case where $x$ lies on $P_{c,d} \setminus \{d\}$ and the vertex in $A'_{1} \cap B'_{1} \setminus \{v\}$ lies on $P_{b,d}$ follows in a similar fashion.

 Now suppose that $x$ lies on $P_{b,c} \setminus \{b\}$ and the vertex in $A'_{1} \cap B'_{1} \setminus \{v\}$ lies on $P_{c,d} \setminus \{d\}$. Let $y$ be the vertex in $A'_{1} \cap B'_{1} \setminus \{v\}$. Then $\{u,y\}$ is a $2$-vertex cut. Let $(A',B')$ be the separation induced by this $2$-vertex cut. Then if $n=1$, we have $(A,B),(A',B'), (A'_{1},B'_{1})$ plus $(A_{1},B_{1}),(A_{2},B_{2}),(A_{3},B_{3})$ form obstruction $5$, a contradiction. Otherwise, $n >1$ and $(A,B),(A',B'), (A'_{1},B'_{1})$ plus $((A'_{2},B'_{2}) \ldots (A'_{n},B'_{n}))$ and $(A_{1},B_{1}),(A_{2},B_{2}),(A_{3},B_{3})$ satisfies obstruction $4$, a contradiction. 
 
 Now suppose that $x$ lies on $P_{c,d} \setminus \{d\}$ and the vertex in $A'_{1} \cap B'_{1} \setminus \{v\}$ lies in $P_{b,c} \setminus \{c\}$. We note that a similar argument to when both $x$ and the vertex in $A'_{1} \cap B'_{1} \setminus \{v\}$ were in $P_{b,c}$ and $x \in B'_{1} \setminus A'_{1} $ works in this case.

 Therefore, there are no $2$-separations of the form $(A,B)$ such that $A \cap B = \{u,v\}$, $a \in A \setminus \{u,v\}$, and $b,c,d \in B \setminus \{u,v\}$. 

\textbf{Claim 5:} There are no $2$-separations $(A,B)$ such that $A \cap B$ contain two vertices from $X$, and $A \setminus  B$ contains a vertex from $X$ and $B \setminus A$ contains a vertex from $X$. 

If such separation existed, it would be a terminal separating $2$-chain, a contradiction. 

Therefore, our graph $G$ has no $2$-separations. Therefore $G$ is $3$-connected. But every $3$-connected graph which does not have a $K_{4}(X)$-minor has a $W_{4}(X)$-minor by Theorem \ref{3connectivityW4K4}. But this contradicts our choice of $G$, completing the claim. 
\end{proof}

\section{A characterization of graphs without a $K_{4}(X)$, $W_{4}(X)$, or a $K_{2,4}(X)$-minor}

In this section, we will first show that spanning subgraphs of class $\mathcal{A}$ graphs are $K_{2,4}(X)$-free. Second, we show that $2$-connected spanning subgraphs of class $\mathcal{B}$ and $\mathcal{C}$ graphs always have a $K_{2,4}(X)$-minor. Third, we show that determining if a spanning subgraph class $\mathcal{E}$ or $\mathcal{F}$ graph has a $K_{2,4}(X)$-minor reduces to looking at the underlying web. Finally, we show that an $2$-connected spanning subgraph of an $\{a,b,c,d\}$-web has a $K_{2,4}(X)$-minor if and only if it has an $W_{4}(X)$-minor. 

Now we formally define $K_{2,4}(X)$-minors. Let $V(K_{2,4}) = \{t_{1},t_{2},t_{3},t_{4},s_{1},s_{2}\}$ where $E(K_{2,4}) = \{t_{i}s_{j} \ | \ \forall i \in \{1,2,3,4\}, j \in \{1,2\}\}$.  Let $G$ be a graph and $X = \{a,b,c,d\} \subseteq V(G)$. Let $\mathcal{F}$ be the family of maps from $X$ to $V(K_{2,4})$ such that each vertex of $X$ goes to a distinct vertex in $\{t_{1},t_{2},t_{3},t_{4}\}$. For the purposes of this paper, a $K_{2,4}(X)$ minor refers to the $X$ and family of maps given above.

\begin{figure}
\begin{center}
\includegraphics[clip, trim=2cm 11cm 0.5cm 2cm, scale =0.8]{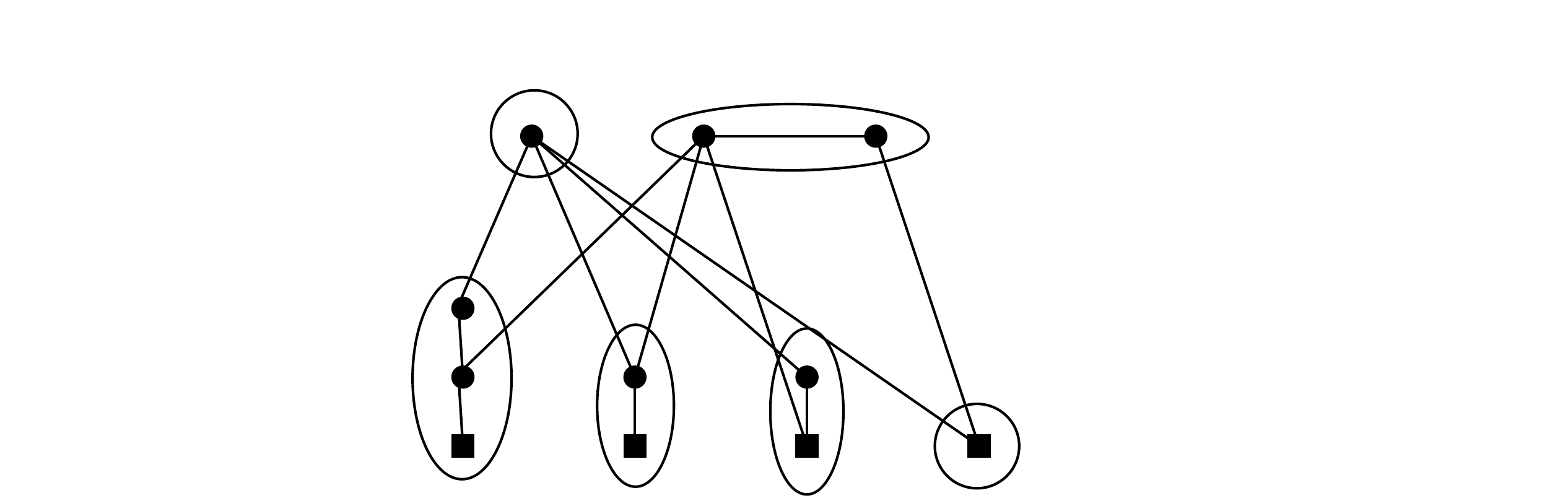}
\caption{A graph $G$ with a model of a $K_{2,4}(X)$-minor. The circles represent connected subgraphs, the vertices represented by squares represent vertices from $X$.}
\end{center}
\end{figure}

We note that $K_{2,4}(X)$-minors have been studied before. They were the subject of Demasi's PhD thesis \cite{linodemasithesis}, where he gave a characterization of $3$-connected planar graphs not containing a $K_{2,4}(X)$-minor. We do not make use of this characterization here, as excluding $K_{4}(X)$ and $W_{4}(X)$-minors already significantly simplifies the problem, but we will appeal to some of his results.

Of particular interest will be $K_{2,2}(X)$-minors due to a useful characterization in \cite{linodemasithesis}. Let $V(K_{2,2}) = \{t_{1},t_{2},s_{1},s_{2}\}$ where $E(K_{2,2}) = \{t_{i}s_{j} \ | \ \forall i,j \in \{1,2\}\}$. Let $G$ be a graph and $X = \{a,b,c,d\} \subseteq V(G)$. Define $\mathcal{F}$ to be the family of maps where $a$ and $b$ are mapped to $\{s_{1},s_{2}\}$ and $c$ and $d$ are mapped to $\{t_{1},t_{2}\}$. For the purpose of this paper, a $K_{2,2}(X)$-minor refers to the $X$ and $\mathcal{F}$ above.

\begin{theorem}[\cite{linodemasithesis}]
\label{k22characterization}
Let $a,b,c,d$ be distinct vertices in a graph $G$. Then $G$ contains a 
$K_{2,2}(X)$-minor if and only if there exists an $(a,c)$-path $P_{a,c}$ and a $(b,d)$-path $P_{b,d}$ such that $P_{a,c} \cap P_{b,d} = \emptyset$ and there exists an $(a,d)$-path $P_{a,d}$ and a $(b,c)$-path $P_{b,c}$ such that $P_{b,c} \cap P_{a,d} = \emptyset$.
\end{theorem}

Since $K_{2,4}$ is $2$-connected, by previous discussion, we may assume we are dealing with $2$-connected graphs. Now we record some of the  $2$-connectivity reductions from \cite{linodemasithesis}. For the next three lemmas, suppose $G$ is a $2$-connected graph, $X = \{a,b,c,d\} \subseteq V(G)$ and $(A,B)$ a $2$-separation such that $A \cap B = \{u,v\}$. Also let $G_{A} = G[A] \cup \{uv\}$ and $G_{B} = G[B] \cup \{uv\}$, as before. 

\begin{lemma}[\cite{linodemasithesis}]
\label{alloneside}
Suppose that $X \subseteq B$. Then $G$ has a $K_{2,4}(X)$ minor if and only if $G_{B}$ has a $K_{2,4}(X)$-minor.
\end{lemma}

\begin{lemma}[\cite{linodemasithesis}]
\label{splitvertices}
If $u, v \in X$, then $G$ has a $K_{2,4}(X)$-minor if and only if $X \subseteq A$ and $G_{A}$ has a $K_{2,4}(X)$-minor or $X \subseteq B$ and $G_{B}$ has a $K_{2,4}(X)$-minor. Suppose $u \in X$ and $v \not \in X$. Furthermore, suppose $a \in A \setminus B$ and $b,c,d \in B$. For each $\pi \in \mathcal{F}$, let $\pi_{B}: \{u,b,c,d\} \to V(K_{2,4})$ be such that $\pi_{B}(u) = \pi(a)$ and on $\{b,c,d\}$, $\pi_{B} = \pi$.  Then $G$ has a $K_{2,4}(X)$-minor if and only if $G_{B}$ has a $K_{2,4}(X)$-minor. 
\end{lemma}

\begin{lemma}[\cite{linodemasithesis}]
\label{thereductionwithk22init}
Suppose $a,b \in A \setminus  B$ and $c,d \in B \setminus A $. For each $\pi \in \mathcal{F}$, let $\pi_{A}: \{a,b,u,v\} \to V(K_{2,4})$ be such that $\pi_{A} = \pi$ on $\{a,b\}$ and $\pi_{A}(u) =\pi(c) $ and $\pi_{A}(v) = \pi(d)$. 
Similarly for each $\pi \in \mathcal{F}$, define $\pi_{B} : \{c,d,u,v\} \to V(K_{2,4})$. Additionally for each $\pi \in \mathcal{F}$, let $\pi'_{A}: \{a,b,u,v\} \to V(K_{2,2})$ be such that $\pi'_{A}= \pi$ on $\{a,b\}$ and $\pi'_{A}(u) = \pi(c)$ and $\pi'_{A}(v) = \pi(d)$. 
Also, for each $\pi \in \mathcal{F}$, let $\pi'_{B}: \{c,d,u,v\} \to V(K_{2,2})$ be such that $\pi'_{B} = \pi$ on $\{c,d\}$ and, $\pi'_{A}(u) = \pi(a)$ and $\pi'_{A}(v) = \pi(b)$.  Then $G$ has a $K_{2,4}(X)$-minor if and only if either $G_{A}$ or $G_{B}$ has a $K_{2,4}(X)$-minor, or both of $G_{A}$ and $G_{B}$ have a $K_{2,2}(X)$ minor. 
\end{lemma}

That completes the $2$-connected reductions from \cite{linodemasithesis} that will be needed. There is one $3$-connected reduction from \cite{linodemasithesis} which is useful. 

\begin{lemma}
\label{3connk24}
Let $G$ be a graph and let $(A,B)$ be a tight $3$-separation where $A \cap B = \{v_{1},v_{2},v_{3}\}$. Suppose that $X \subseteq A$. Let $G' = G[A] \cup \{v_{1}v_{2},v_{1}v_{3},v_{2}v_{3}\}$. Then $G$ has a $K_{2,4}(X)$-minor if and only if $G'$ has a $K_{2,4}(X)$-minor. 
\end{lemma}

Additional reduction lemmas are proven in \cite{linodemasithesis}, but these suffice for our needs.
To avoid repeating the same statements in the next lemmas, we make the following observation.

\begin{lemma}
\label{smallgraphk22}
Let $H$ be the graph where $V(H) = V(K_{2,2})$ and $E(H) = E(K_{2,2}) \cup \{s_{1}s_{2}\}$. Let $G$ be a $2$-connected spanning subgraph of $H^{+}$. From our definition of $K_{2,2}(X)$-minors, let $t_{1},t_{2}$ take the place of $a,b$ and $s_{1}, s_{2}$ take the place of $c$ and $d$. Then $G$ has a $K_{2,2}(X)$-minor, and $G$ does not have a $K_{2,4}(X)$-minor. 
\end{lemma}

\begin{proof}
Notice that $\{s_{1},s_{2}\}$ are the vertex boundary for a $2$-separation $(A,B)$  in $H^{+}$ such that $t_{1} \in A \setminus \{s_{1},s_{2}\}$ and $t_{2} \in B \setminus \{s_{1},s_{2}\}$. Then by Lemma \ref{splitvertices},  $H^{+}$ has no $K_{2,4}(X)$-minor and thus $G$ has no $K_{2,4}(X)$ minor.

 As $G$ is $2$-connected, we can find a $(s_{1},t_{1})$-path in $G[A]$ which does not contain $s_{2}$ and a $(s_{2},t_{2})$-path in $G[B]$ which does not use $s_{1}$. Similarly there exists a $(s_{1},t_{2})$-path in $G[B]$ which does not use $s_{2}$ and a $(s_{2},t_{1})$-path in $G[A]$ which does not use $G[B]$. Therefore by Theorem \ref{k22characterization}, $G$ has a $K_{2,2}(X)$-minor.  
\end{proof}

\begin{corollary}
Let $G$ be a $2$-connected spanning subgraph of a class $\mathcal{A},\mathcal{B},$ or $\mathcal{C}$. If $G$ is the spanning subgraph of a class $\mathcal{A}$ graph, $G$ does not have a $K_{2,4}(X)$-minor. If $G$ is the spanning subgraph of a class $\mathcal{B},$ or $\mathcal{C}$, then $G$ has a $K_{2,4}(X)$-minor.
\end{corollary} 

\begin{proof}
If $G$ is the spanning subgraph of a class $\mathcal{A}$ graph, then $\{d,e\}$ is the vertex boundary of a $2$-separation $(A,B)$ where two vertices of $X$ lie in $A \setminus \{d,e\}$ and one lies in $B \setminus \{d,e\}$. By applying Lemma \ref{splitvertices} we see that $G$ has a $K_{2,4}(X)$-minor if and only if the graph in Lemma \ref{smallgraphk22} has a $K_{2,4}(X)$-minor. By Lemma \ref{smallgraphk22}, it does not. Therefore $G$ is $K_{2,4}(X)$-minor free.

If $G$ is the spanning subgraph of a class $\mathcal{B}$ or $\mathcal{C}$ graph, then $\{e,f\}$ is the vertex boundary of a $2$-separation $(A,B)$ where two vertices of $X$ lie in $A \setminus \{e,f\}$ and two vertices of $X$ lies in $B \setminus \{e,f\}$. Then applying Lemma \ref{thereductionwithk22init}, we see that $G$ has a $K_{2,4}(X)$-minor if the graph in Lemma \ref{smallgraphk22} has a $K_{2,2}(X)$-minor. By Lemma \ref{smallgraphk22}, the graph in question has a $K_{2,2}(X)$-minor and therefore $G$ has a $K_{2,4}(X)$-minor. 
\end{proof}

Now we deal with class $\mathcal{E}$ and $\mathcal{F}$ graphs.

\begin{lemma}
\label{planarreductionk22}
Let $G$ be a $2$-connected spanning subgraph of a $\{t_{1},t_{2},s_{1},s_{2}\}$-web, $H^{+} = (H,F)$. Then there is a planar $2$-connected graph $G'$ such that $G$ has a $K_{2,2}(X)$-minor if and only if $G'$ has a $K_{2,2}(X)$-minor. 
\end{lemma} 

\begin{proof}
We construct $G'$ in the same way that we construct $G'$ in Corollary \ref{webcycle}. We refer the reader to the proof of Corollary \ref{webcycle} for the verification that $G'$ is $2$-connected and planar. Notice that if $G'$ has a $K_{2,2}(X)$-minor, then immediately $G$ has a $K_{2,2}(X)$-minor. 

Therefore we assume that $G$ has a $K_{2,2}(X)$-minor. Then by Theorem \ref{k22characterization}, there is a $(t_{1},s_{1})$-path $P_{t_{1},s_{1}}$ and a $(t_{2},s_{2})$-path, $P_{t_{2},s_{2}}$ such that $P_{t_{1},s_{1}} \cap P_{t_{2},s_{2}} \neq \emptyset$ and there is an $(t_{1},s_{2})$-path, $P_{t_{1},s_{2}}$ and a $(t_{2},s_{1})$-path $P_{t_{2},s_{1}}$ such that $P_{t_{1},s_{2}} \cap P_{t_{2},s_{1}} \neq \emptyset$. 

Suppose $P_{t_{1},s_{1}}$ contains a vertex from $F_{T}$ for some $T$. Then at least two vertices from $T$ are in $P_{t_{1},s_{1}}$. Then since $|V(T)| = 3$, $P_{t_{2},s_{2}}$ does not contain any vertex from $F_{T}$. Therefore if we contract $F_{T}$ down to a vertex, after contracting appropriately      $P_{t_{1},s_{1}}$ is still a $(t_{1},s_{1})$-path, and $P_{t_{1},s_{1}} \cap P_{t_{2},s_{2}} = \emptyset$. A similar statement holds for $P_{t_{2},s_{1}}$ and $P_{t_{1},s_{2}}$. Applying that argument to each triangle $T$ in $H$ and appealing to Theorem \ref{k22characterization}, $G'$ has a $K_{2,2}(X)$-minor. 
\end{proof}

\begin{lemma}
\label{k22minorswebs}
Let $G$ be a $2$-connected spanning subgraph of a $\{t_{1},t_{2},s_{1},s_{2}\}$-web, $H^{+} = (H,F)$. Assume that $t_{1},t_{2},s_{1}$ and $s_{2}$ appear in that order in the outerface of $H$. Then $G$ does not contain a $K_{2,2}(X)$-minor. 
\end{lemma}

\begin{proof}
We note it suffices to show that $H^{+}$ does not have a $K_{2,2}(X)$-minor. By Lemma \ref{planarreductionk22}, we may assume that $H^{+}$ is planar. Then by Observation \ref{planarface} the cycle with edge set $t_{1}t_{2}, t_{2}s_{1}, s_{1}s_{2}, s_{2}t_{1}$ is the boundary of a face. Consider any  $(t_{1},s_{1})$-path $P_{t_{1},s_{1}}$ and any $(t_{2},s_{2})$-path $P_{t_{2},s_{2}}$. We claim that $P_{t_{1},s_{1}} \cap P_{t_{2},s_{2}} \neq \emptyset$. If $t_{2},s_{2} \in P_{t_{1},s_{1}}$ or $t_{1},s_{1} \in P_{t_{2},s_{2}}$ then we are done. Therefore we assume that $t_{2},s_{2} \not \in P_{t_{1},s_{1}}$ and $t_{1},s_{1} \not \in P_{t_{2},s_{2}}$. But then by the Jordan Curve Theorem, $P_{t_{1},s_{1}} \cap P_{t_{2},s_{2}}$ is non-empty, completing the claim. 
\end{proof}

\begin{lemma}
Let $G$ be a $2$-connected spanning subgraph of a class $\mathcal{E}$ or a $\mathcal{F}$ graph. If $G$ is a spanning subgraph of a class $\mathcal{E}$ graph, then $G$ has a $K_{2,4}(X)$-minor if and only if the $\{e,f,c,d\}$-web has a $K_{2,4}(X)$-minor. If $G$ is a spanning subgraph of a class $\mathcal{F}$ graph, then $G$ has a $K_{2,4}(X)$-minor if and only if the $\{e,f,g,h\}$-web has a $K_{2,4}(X)$-minor. 
\end{lemma}

\begin{proof}
First suppose that $G$ is a $2$-connected spanning subgraph of a class $\mathcal{E}$ graph. If the $\{e,f,c,d\}$-web has a $K_{2,4}(X)$-minor, then immediately $G$ has a $K_{2,4}(X)$-minor. 

Therefore assume that $G$ has a $K_{2,4}(X)$-minor. Apply Lemma \ref{thereductionwithk22init} to the $2$-separation with vertex boundary $\{e,f\}$. Then by appealing to Lemma \ref{k22minorswebs} and Lemma \ref{smallgraphk22}, we get that $G$ having a $K_{2,4}(X)$-minor implies the $\{e,f,c,d\}$-web has a $K_{2,4}(X)$-minor, completing the claim. 
Essentially the same argument gives the claim for the class $\mathcal{F}$ graphs. 
\end{proof}

Now it suffices to deal with webs to complete the characterization. Note we can reduce the problem of finding $K_{2,4}(X)$-minors down to the planar case. 

\begin{lemma}
\label{k24planarreduction}
Let $H^{+} = (H,F)$ be an $\{a,b,c,d\}$-web. Let $G$ be $2$-connected spanning subgraph of $H^{+}$. Then there is a planar graph $K$ such that $G$ has a $K_{2,4}(X)$-minor if and only if $K$ has a $K_{2,4}(X)$-minor. 
\end{lemma}

\begin{proof}
For each triangle $T$ in $H$, and every two element subset of $V(T)$ which induces a $2$-separation $(A,B)$ such that $B \setminus A = V(F_{T})$, apply Lemma \ref{allonesidegeneralH}. After doing this to every triangle, notice that for every triangle $T \in H$, $T$ induces a tight $3$-separation $(A,B)$ such that $B = V(F_{T}) \cup V(T)$. Then we may apply Lemma \ref{3connk24} to $(A,B)$. Call the resulting graph $K$. By construction, $K$ has a $K_{2,4}(X)$-minor if and only if $G$ has a $K_{2,4}(X)$-minor. Additionally, notice that $K$ is a subgraph of $H$, and $H$ is planar, so thus $K$ is planar.
\end{proof}

Now we make the main claim, which is that for spanning subgraphs of $2$-connected webs, $K_{2,4}(X)$-minors occur if and only if $W_{4}(X)$-minors occur.

\begin{theorem}
\label{k4w4k24characterization}
Let $G$ be a $2$-connected graph which is the spanning subgraph of an $\{a,b,c,d\}$-web. Let $X = \{a,b,c,d\} \subseteq V(G)$.  If $G$ does not have a $W_{4}(X)$-minor, then $G$ does not have a $K_{2,4}(X)$-minor. 
\end{theorem}

\begin{proof}

Since $G$ does not have a $W_{4}(X)$-minor, by Theorem \ref{w4cuts} for every cycle $C$ such that $X \subseteq V(C)$, we have one of five obstructions. By Corollary \ref{webcycle} we know at least one such cycle exists. Suppose for sake of contradiction, that $G$ is a minimal counterexample with respect to vertices. We proceed by checking the five cases from Theorem \ref{w4cuts}. 

\textbf{Case 1:} Suppose that there is a terminal separating $2$-chain $((A_{1},B_{1}),(A_{2},B_{2}),\ldots, (A_{n},B_{n}))$ such that $A_{i} \cap B_{i} \subseteq V(C)$ for all $i \in \{1,\ldots,n\}$. If $A_{1} \cap B_{1}$ contains two vertices of $X$, then by Lemma \ref{splitvertices} there is no $K_{2,4}(X)$-minor. Therefore we can assume that $A_{1} \cap B_{1}$ contains exactly one vertex from $X$. Then by Lemma \ref{splitvertices}, the graph $G$ has $K_{2,4}(X)$-minor if and only if the graph $G_{B_{1}}$ has a $K_{2,4}(X_{1})$-minor, where $X_{1}$ is defined from Lemma \ref{splitvertices}. Since $G$ did not have a $W_{4}(X)$-minor, $G_{B_{1}}$ does not have a $W_{4}(X_{1})$-minor, and thus by minimality $G_{B_{1}}$ does not have a $K_{2,4}(X_{1})$-minor.

\textbf{Case 2:} Suppose that there is a terminal separating triangle $(A_{1},B_{1}),(A_{2},B_{2}),(A_{3},B_{3})$ such that $A_{i} \cap B_{i} \subseteq V(C)$ and $A_{1} \cap B_{1}$ contains a vertex of $X$. Then applying Lemma \ref{splitvertices} to the separation $(A_{1},B_{1})$ we that $G$ has a $K_{2,4}(X)$-minor if and only if the graph $G_{B_{1}}$ has a $K_{2,4}(X_{1})$-minor, where $X_{1}$ is defined from Lemma \ref{splitvertices}. Since $G$ does not have a $W_{4}(X)$-minor, $G_{B_{1}}$ does not have a $W_{4}(X_{1})$-minor, and thus by minimality, $G_{B_{1}}$ does not have a $K_{2,4}(X_{1})$-minor, and therefore $G$ does not have a $K_{2,4}(X)$-minor. 

\textbf{Case 3:} Suppose that there is a terminal separating triangle $(A^{1}_{1},B^{1}_{1}), (A^{1}_{2},B^{1}_{2}),(A^{1}_{3},B^{1}_{3})$ and a terminal separating $2$-chain $((A_{1},B_{1}),\ldots, (A_{n},B_{n}))$ in $G_{A^{1}_{1}}$. Then suppose that $A_{n} \cap B_{n}$ contains a vertex of $X$. Then applying Lemma \ref{splitvertices} to the separation $(A_{n},B_{n})$, we have that $G$ has a $K_{2,4}(X)$-minor if and only if $G_{A_{n}}$ has a $K_{2,4}(X_{1})$-minor, where $X_{1}$ is defined from Lemma \ref{splitvertices}. Since $G$ does not have a $W_{4}(X)$-minor, $G_{A_{n}}$ does not have a $W_{4}(X_{1})$-minor, and thus by minimality, $G_{A_{n}}$ does not have a $K_{2,4}(X_{1})$-minor so $G$ does not have a $K_{2,4}(X)$-minor. 

\textbf{Case 4:} Suppose we have terminal separating triangles $((A^{1}_{1},B^{1}_{1}), (A^{1}_{2},B^{1}_{2}), (A^{1}_{3},B^{1}_{3}))$, $((A^{2}_{1},B^{2}_{1}),(A^{2}_{2},B^{2}_{2}), \\ (A^{2}_{3},B^{2}_{3}))$  where for all $i \in \{1,2,3\}$,  $A^{1}_{i} \cap B^{1}_{i} \subseteq A^{2}_{1}$  and $A^{2}_{i} \cap B^{2}_{i} \subseteq A^{1}_{3}$, and if we consider the graph $G[A^{2}_{1} \cap A^{1}_{1}]$, and the cycle $C' = G[V(C) \cap A^{2}_{1} \cap A^{1}_{1}] \cup \{xy | x,y \in A^{i}_{1} \cap B^{i}_{1}, i \in \{1,2\}\}$ and we let $X'$ be defined to be the vertices $A^{2}_{1} \cap B^{2}_{1}$ and $A^{2}_{3} \cap B^{2}_{3}$, then there is a terminal separating $2$-chain in $G[A^{2}_{3} \cap A^{1}_{1}]$ with respect to $X'$. 

Consider the two separation $(A^{1}_{1},B^{1}_{1})$. From previous analysis, there are two vertices of $X \in A^{1}_{1} \setminus  B^{1}_{1}$ and two vertices of $X$ are in $B^{1}_{1} \setminus A^{1}_{1}$. Without loss of generality, let $a,b$ be the two vertices in $X$ in $ A^{1}_{1} \setminus B^{1}_{1}$. Therefore we can apply Lemma \ref{thereductionwithk22init} to $(A^{1}_{1},B^{1}_{1})$ and get a two graphs $G_{A^{1}_{1}}$, $G_{B^{1}_{1}}$ such that $G$ has a $K_{2,4}(X)$-minor if and only if either $G_{A^{1}_{1}}$ has a $K_{2,4}(X_{1})$ or both have a $K_{2,2}(X_{1})$-minor, where $X_{1}$ is defined from Lemma \ref{thereductionwithk22init}. Since $G$ did not have a $W_{4}(X)$-minor both $G_{A^{1}_{1}}$ and $G_{B^{1}_{1}}$ do not have a $W_{4}(X_{1})$-minor, and thus by minimality, both $G_{A^{1}_{1}}$ and $G_{B^{1}_{1}}$ do not have $K_{2,4}(X_{1})$-minors. 

Therefore it suffices to show that $G_{B^{1}_{1}}$ does not have a $K_{2,2}(X_{1})$-minor (see Figure \ref{ThegraphGA} for a picture). Let $A^{1}_{1} \cap B^{1}_{1} = \{t_{1},t_{2}\}$. Without loss of generality, let $A^{1}_{2} \cap B^{1}_{2} = \{v,t_{1}\}$ and $A^{1}_{3} \cap B^{1}_{3} = \{v, t_{2}\}$. Now since we had a terminal separating triangle, we have that $v \neq a,b$. So without loss of generality, we may assume $a \in A^{1}_{2} \setminus  B^{1}_{2}$ and $b \in A^{1}_{3} \setminus  B^{1}_{3}$. Now by Theorem \ref{k22characterization}, it suffices to show that for every  $(a,t_{2})$-path, $P_{a,t_{2}}$ and every $(b,t_{1})$-path, $P_{b,t_{1}}$ in $G_{B^{1}_{1}}$ we have $P_{a,t_{2}} \cap P_{b,t_{1}} \neq \emptyset$. Let $P_{a,t_{2}}$ be any $(a,t_{2})$-path. Since $a \in A^{1}_{2} \setminus \{v,t_{1}\}$ and $t_{2} \in B^{1}_{2} \setminus \{v,t_{1}\}$, either $v$ or $t_{1} \in V(P_{a,t_{2}})$. If $t_{1} \in V(P_{a,t_{2}})$, then any $(b,t_{2})$-path $P_{b,t_{2}}$ contains $t_{1}$ by definition, so $P_{a,t_{2}} \cap P_{b,t_{1}} \neq \emptyset$. Therefore we only have to consider when $v \in P_{a,t_{2}}$.
Now since $b \in A^{1}_{3} \setminus \{v,t_{2}\}$ and $t_{1} \in B^{1}_{3} \setminus \{v,t_{2}\}$ every $(b,t_{1})$-path $P_{b,t_{1}}$ contains either $v$ or $t_{2}$. By similar reasoning as above, we may assume that $t_{2} \not \in P_{b,t_{1}}$. Therefore $v \in P_{b,t_{1}}$. But then $P_{b,t_{1}} \cap P_{a,t_{2}} \neq \emptyset$, which implies $G_{B^{1}_{1}}$ does not have a $K_{2,2}(X_{1})$-minor. Combining this with what we already showed, this implies that $G$ does not have a $K_{2,4}(X)$-minor. 

\textbf{Case 5:} Now suppose there are $2$ distinct terminal separating triangles $(A^{1}_{1},B^{1}_{1}), (A^{1}_{2},B^{1}_{2}), (A^{1}_{3},B^{1}_{3}), \\ (A^{2}_{1},B^{2}_{1}), (A^{2}_{2},B^{2}_{2}),(A^{2}_{3},B^{2}_{3})$ where for all $i \in \{1,2,3\}$,  $A^{1}_{i} \cap B^{1}_{i} \subseteq A^{2}_{1}$  and $A^{2}_{i} \cap B^{2}_{i} \subseteq A^{1}_{1}$, $A^{2}_{1} \cap B^{2}_{1} \cap A^{1}_{1} \cap B^{1}_{1}$ is not empty and $A^{j}_{i} \cap B^{j}_{i} \subseteq V(C)$ for all $i \in \{1,2,3\}$, $j \in \{1,2\}$. Now notice if we apply Lemma \ref{thereductionwithk22init} to $(A^{1}_{1},B^{1}_{1})$, we get two graphs $G_{A^{1}_{1}}$ and $G_{B^{1}_{1}}$ such that $G$ has a $K_{2,4}(X)$-minor if and only if either one of $G_{A^{1}_{1}}$ or $G_{B^{1}_{1}}$ has a $K_{2,4}(X_{1})$-minor or both have $K_{2,2}(X_{1})$-minors, where $X_{1}$ is defined from Lemma \ref{thereductionwithk22init}. Now since $G$ does not have a $W_{4}(X)$-minor both of $G_{B^{1}_{1}}$ and $G_{A^{1}_{1}}$  do not have $W_{4}(X_{1})$-minors and thus by minimality both do not have $K_{2,4}(X_{1})$-minors. Observe that we can apply the same argument as case four to $G_{B^{1}_{1}}$ to obtain that $G_{B^{1}_{1}}$ does not have a $K_{2,2}(X_{1})$-minor. Therefore $G$ does not have a $K_{2,4}(X)$-minor, completing the claim. 
\end{proof}

\begin{figure}
\begin{center}
\includegraphics[scale =0.5]{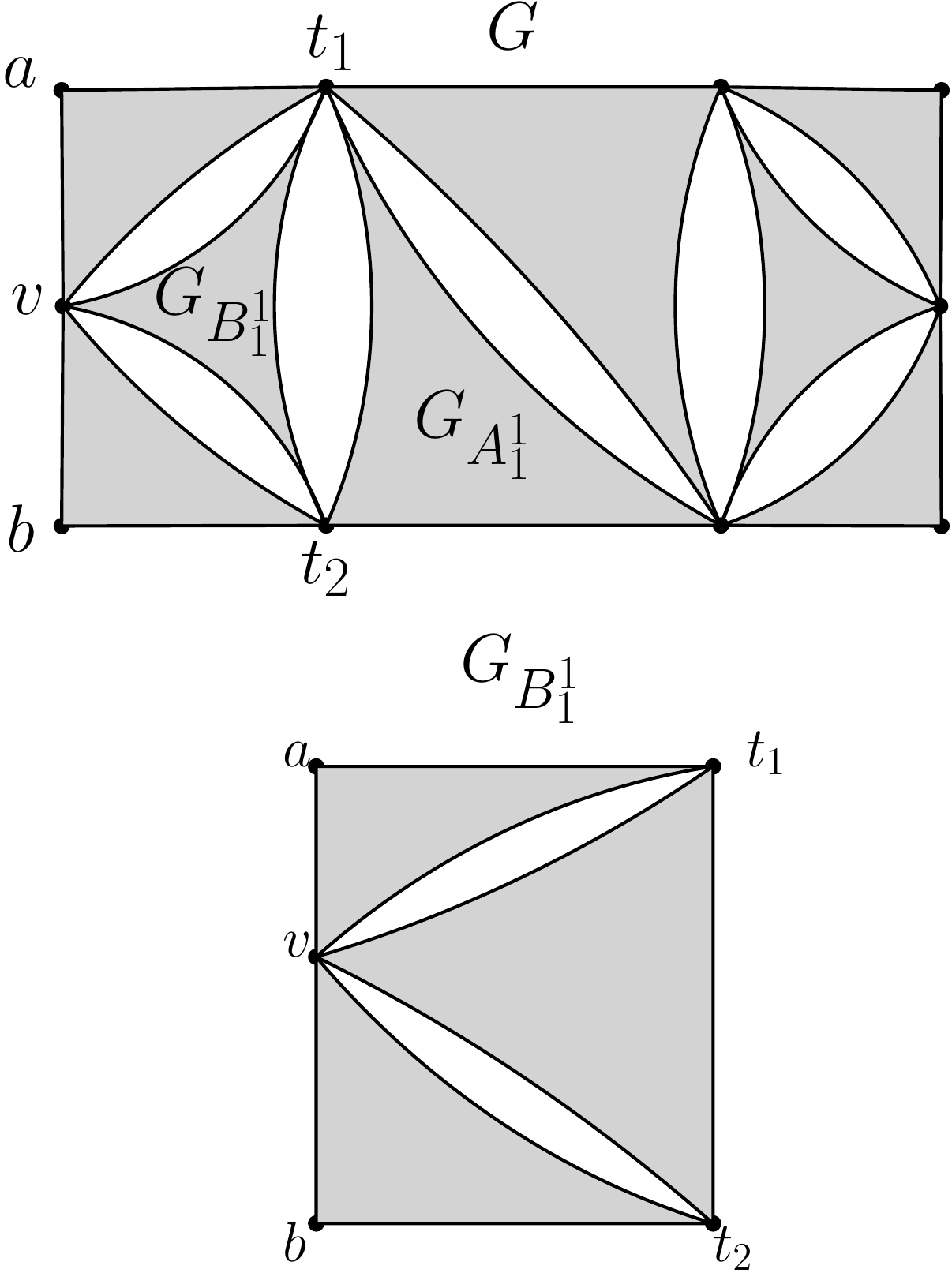}
\caption{The graph $G_{B^{1}_{1}}$ in case four of Theorem \ref{k4w4k24characterization}.}
\label{ThegraphGA}
\end{center}
\end{figure}

As a recap of what we have so far.

\begin{corollary}
Let $G$ be a $2$-connected graph and $X = \{a,b,c,d\} \subseteq V(G)$. The graph $G$ has no $K_{4}(X)$, $K_{2,4}(X)$ or $W_{4}(X)$-minor if and only if $G$ belongs to class $\mathcal{A}$ (see Theorem \ref{k4free}) or $G$ is the spanning subgraph of a class $\mathcal{D}$, $\mathcal{E}$ and $\mathcal{F}$ graph and the corresponding web does not have a $W_{4}(X)$-minor (see Theorem \ref{w4cuts}). 
\end{corollary}

\section{A characterization of graphs without a $K_{4}(X)$, $W_{4}(X)$, $K_{2,4}(X)$ or an $L(X)$-minor}
\label{Lsection}
In this section, we look at the following problem. Suppose $G$ has no $K_{4}(X), W_{4}(X)$ and $K_{2,4}(X)$. When does $G$ have an $L(X)$-minor? We reduce this problem to finding an $L'(X)$-minor where $L'$ is a smaller graph. 

We define the graph $L$ to have vertex set $V(L) = \{v_{1},\ldots,v_{8}\}$ and $E(L) = \{v_{1}v_{2},v_{1}v_{5},v_{2}v_{7}, v_{2}v_{8},v_{2}v_{3},v_{3}v_{4}, \\ v_{4}v_{5}, v_{4}v_{7},v_{5}v_{6},v_{6}v_{7},v_{6}v_{8},v_{7}v_{8}\}$ (see Figure \ref{L(X)labelling}). Let $G$ be a graph and $X = \{a,b,c,d\} \subseteq V(G)$. Let $\mathcal{F}$ be the family of maps from $X$ to $V(L)$ where each vertex of $X$ goes to a distinct vertex in $\{v_{1},v_{3},v_{4},v_{5}\}$. For the purposes of this paper, an $L(X)$-minor refers to the $X$ and family of maps defined above. It is easy to see that the graph $L$ is $2$-connected, so the cut vertex section applies. Therefore we may assume that all graphs are at least $2$-connected. 

We let $L'$ denote the graph induced by $\{v_{2},v_{8},v_{7},v_{6},v_{5},v_{4}\}$ in $L$. Let $G$ be a graph and $X = \{a,b,c\} \subseteq V(G)$. Let $\mathcal{F}$ be the family of surjective maps from $X$ to $\{v_{2},v_{4},v_{5}\}$. An $L'(X)$-minor will refer to the $\mathcal{F}$ and $X$ above. It is easy to see that $L'$ is $2$-connected and thus we may assume all graphs are $2$-connected. 

\begin{figure}
\begin{center}
\includegraphics[scale=0.5]{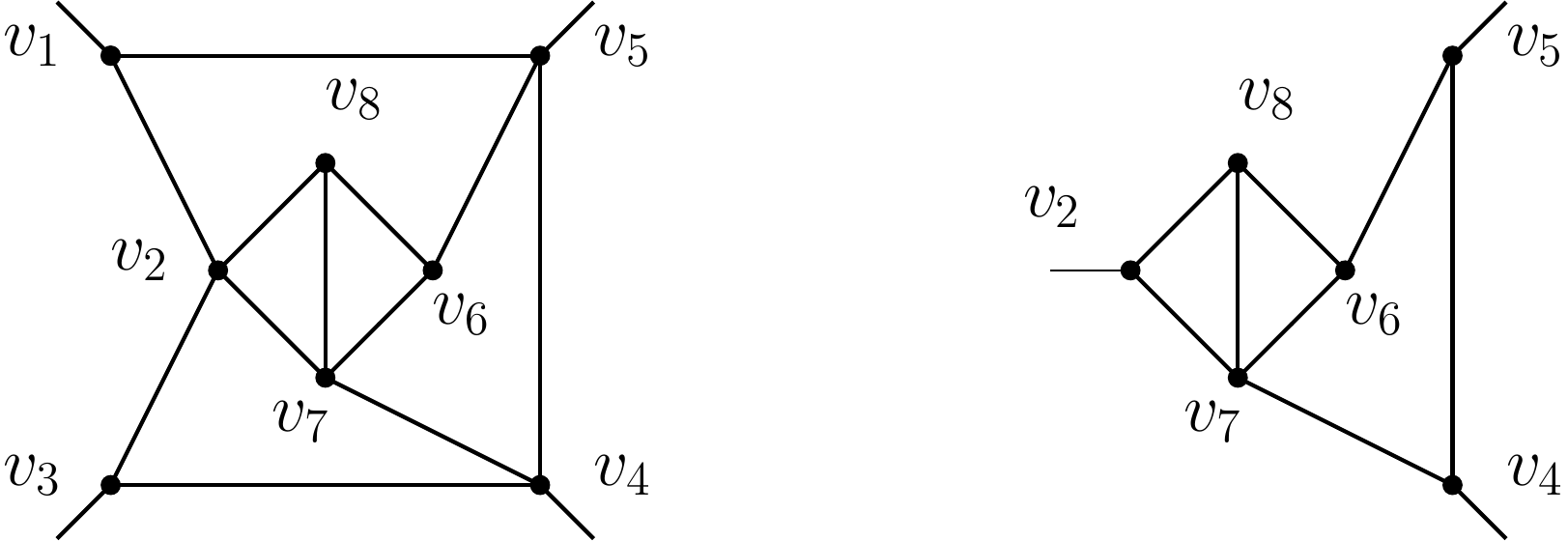}
\caption{The graph $L$ and the graph $L'$. Vertices with lines with only one endpoint indicate the vertices where the roots are being mapped to.}
\label{L(X)labelling}
\end{center}
\end{figure}

\begin{lemma}
\label{L'(X')minors}
Let $G$ be a $2$-connected graph and let $X = \{a,b,c\}$. Then $G$ has an $L'(X)$-minor if and only if there are three distinct cycles $C_{1}$, $C_{2}$, $C_{3}$ and three distinct paths $P_{1},P_{2},P_{3}$ satisfying the following properties: 
\begin{enumerate}
\item{$|V(C_{1}) \cap V(C_{2})| \geq 2$, $|V(C_{2}) \cap V(C_{3})| \geq  2$,  $|C_{3}| \geq 4$, there is at least one edge in $E(C_{2})$ which is not contained in either of $E(C_{1})$ and $E(C_{2})$, and there exists a vertex $v_{1} \in (V(C_{2}) \cap V(C_{3})) \setminus V(C_{1})$ and a vertex $v_{2} \in (V(C_{1}) \cap V(C_{2})) \setminus V(C_{3})$.}

\item{The vertices $a,b$ and $c$ are endpoints of $P_{1},P_{2}$ and $P_{3}$ respectively. Additionally, the other endpoint of $P_{1}$ is in $V(C_{1})$, and the other endpoint of point $P_{2}$ and $P_{3}$ is in  $V(C_{3})$. Furthermore, $P_{1} \cap C_{j} = \emptyset$ for all $j \in \{2,3\}$, and $P_{i} \cap C_{j} = \emptyset$ for all $i \in \{2,3\}, j \in \{1,2\}$.}
\end{enumerate}
\end{lemma}

\begin{proof}

Let $\{G_{x} | x \in V(L')\}$ be a model of an $L'(X)$-minor in $G$. Let $x_{1}$ be a vertex in $G_{v_{2}}$ which is adjacent to a vertex $x_{2} \in G_{v_{7}}$. Let $x_{3}$ be a vertex in $G_{v_{7}}$ which is adjacent to a vertex $x_{4}$ in $G_{v_{8}}$. Let $x_{5}$ be a vertex in $G_{v_{8}}$ which is adjacent to a vertex $x_{6}$ in $G_{v_{2}}$. Then since $G_{v_{2}}$ is connected, there is a $(x_{1},x_{6})$-path, $P_{x_{1},x_{6}}$, contained in $G_{v_{2}}$. Similarly, there is a $(x_{2},x_{3})$-path, $P_{x_{2},x_{3}}$, contained in $G_{v_{7}}$ and a $(x_{4},x_{5})$-path, $P_{x_{4},x_{5}}$, contained in $G_{v_{8}}$. Then $C_{1} = P_{x_{1},x_{6}} \cup P_{x_{2},x_{3}} \cup P_{x_{4},x_{5}} \cup \{x_{1}x_{2}, x_{3}x_{4},x_{5}x_{6}\}$ is a cycle. 

Now there is a vertex $v_{1} \in G_{v_{7}}$ which is adjacent to a vertex $v_{2} \in G_{v_{6}}$. Additionally there is a vertex $v_{3} \in G_{v_{6}}$ which is adjacent to a vertex, $v_{4} \in G_{v_{8}}$. As $G_{v_{6}}$ is connected, there is a $(v_{2},v_{3})$-path, $P_{v_{2},v_{3}}$, contained in $G_{v_{6}}$. As $v_{4}$ and $x_{4}$ are in $G_{v_{8}}$, there is a $(v_{4},x_{4})$-path, $P_{v_{4},x_{4}}$, contained in $G_{v_{8}}$. Similarly, as $v_{1}$ and $x_{3}$ are in $G_{v_{7}}$, there is a $(v_{1},x_{3})$-path, $P_{v_{1},x_{3}}$, contained in $G_{v_{7}}$. Then $C_{2} = P_{v_{2},v_{3}} \cup P_{v_{1},x_{3}} \cup P_{v_{4},x_{4}} \cup \{v_{3}v_{4}, x_{3}x_{4},v_{1}v_{2}\}$ is a cycle.

By definition of an $L(X)$-model, there is a vertex $y_{1} \in G_{v_{5}}$ which is adjacent to a vertex $y_{2} \in G_{v_{4}}$. There is a vertex $y_{3} \in G_{v_{4}}$ which is adjacent to a vertex $y_{4} \in G_{v_{7}}$. There is a vertex $y_{5}$ in $G_{v_{5}}$ which is adjacent to a vertex $y_{6}$ in $G_{v_{6}}$. As $G_{v_{7}}$ is connected, this is a $(v_{1},y_{4})$-path, $P_{v_{1},y_{4}}$, contained in $G_{v_{7}}$. As $G_{v_{6}}$ is connected, there is a $(v_{2},y_{6})$-path, $P_{v_{2},y_{6}}$, which is contained in $G_{v_{6}}$. As $G_{v_{5}}$ is connected, there is a $(y_{1},y_{5})$-path, $P_{y_{1},y_{5}}$ contained in $G_{v_{5}}$. As $G_{v_{4}}$ is connected there is a $(y_{3},y_{2})$-path, $P_{y_{3},y_{2}}$ which is contained inside $G_{v_{4}}$. Then let $C_{3} = P_{y_{3},y_{2}} \cup P_{v_{1},y_{4}} \cup P_{y_{3},y_{2}} \cup P_{y_{1},y_{5}} \cup \{y_{3}y_{4}, v_{1}v_{2},y_{6}y_{5},y_{1}y_{2}\}$ is a cycle. Notice that $|C_{3}| \geq 4$ and $|C_{3} \cap C_{2}| \geq 2$. Now without loss of generality let $a \in G_{v_{2}}$. There is a $(a, x_{1})$-path $P_{a,x_{1}}$ contained in $G_{v_{2}}$ as $G_{v_{2}}$ is connected. Without loss of generality let $b \in G_{v_{5}}$. There is a $(b,y_{1})$-path, $P_{b,y_{1}}$, contained in $G_{v_{5}}$ since $G_{v_{5}}$ is connected. Then $c \in G_{v_{4}}$ and there is a $(c,y_{2})$-path $P_{c,y_{2}}$ contained in $G_{v_{4}}$. Then notice it is easy to see that that $P_{1},P_{2},P_{3}, C_{1},C_{2},C_{3}$ satisfy the claim.

Conversely, if given $P_{1},P_{2},P_{3},C_{1},C_{2}$ and $C_{3}$ satisfying the lemma statement, we simply contract $P_{1}$, $P_{2}$,$P_{3}$ down to a single vertex, contract $C_{1}$ and $C_{2}$ to a diamond, and then $C_{3}$ to a $4$-cycle. 
\end{proof}

As with the other sections, we start off with some lemmas about how the $L(X)$-minor behaves across $2$-separations.

\begin{lemma}
\label{allononesideL(X)}
Let $G$ be a graph and $(A,B)$ be a $2$-separation with vertex boundary $\{x,y\}$. If $X = \{a,b,c,d\} \subseteq A$, then $G$ has an $L(X)$-minor if and only if $G_{A} = G[A] \cup \{xy\}$ has an $L(X)$-minor. 
\end{lemma}

\begin{proof}
If $G_{A}$ has an $L(X)$-minor then since $G_{A}$ is a minor of $G$ by contracting all of $G[B]$ onto $\{x,y\}$, we get that $G$ has an $L(X)$-minor.

Conversely, let $\{G_{x} | x \in V(L)\}$ be a model for an $L(X)$-minor. We claim that $\{G_{A}[V(G_{x}) \cap A] | x \in V(L)\}$ is a model of an $L(X)$-minor in $G_{A}$. If there is only one branch set containing vertices of $B$ then the result trivially holds. Therefore we may assume there are at least two distinct branch sets containing vertices of $B$. If the branch sets containing vertices of $B$ are two of $G_{v_{1}}, G_{v_{3}},G_{v_{4}},G_{v_{5}}$, then since $X \subseteq A$, all other branch sets are contained inside $G[A]$. Then since $xy \in E(G_{A})$, $\{G_{A}[V(G_{x}) \cap A] | x \in V(L)\}$ is a model of an $L(X)$-minor.

 Now suppose that exactly one of $G_{v_{1}},G_{v_{3}},G_{v_{4}},$ and $G_{v_{5}}$ contains a vertex from $B$. Suppose that $G_{v_{1}}$ is the branch set which contains the vertex. Then suppose that $G_{v_{2}}$ is the other branch set which contains a vertex from $B$. Since $\deg(v_{2}) >2$, we may assume that $x \in G_{v_{1}}$ and $y \in G_{v_{2}}$ and all other branch sets are contained in $G[A \setminus \{x,y\}]$. Then since $xy \in E(G_{A})$, $\{G_{A}[V(G_{x}) \cap A] \ | \ x \in V(L)\}$ is a model for an $L(X)$-minor. A similar analysis works for the other cases.

Therefore we can assume that two of $G_{v_{2}}$, $G_{v_{8}}$, $G_{v_{7}}$, and $G_{v_{6}}$ contain vertices from $B$. Suppose $G_{v_{2}}$ and $G_{v_{8}}$ are two branch sets containing vertices from $B$. Then since $\deg(v_{2}) >2$ and $\deg(v_{8}) >2$ in $L$, we may assume that $u \in G_{v_{2}}$ and $x \in G_{v_{8}}$ and that all other branch sets are contained in $A \setminus \{x,y\}$. But then since $xy \in E(G_{A})$, $\{G_{A}[V(G_{x}) \cap A] | x \in V(L)\}$ is a model for an $L(X)$-minor. The other cases follow by essentially the same argument. 
\end{proof}

\begin{lemma}
\label{terminalseparatingL(X)}
Let $G$ be a graph and $(A,B)$ a $2$-separation with vertex boundary $\{x,y\}$. Let $X = \{a,b,c,d\} \subseteq V(G)$. If $x,y \in X$ and one vertex of $X$ lies in $A \setminus \{x,y\}$, and the other lies in $B \setminus \{x,y\}$, then $G$ does not have an $L(X)$-minor. 
\end{lemma}

\begin{proof}
Suppose towards a contradiction that $\{G_{x} | x \in V(L)\}$ was a model of an $L(X)$-minor in $G$. If $x = v_{1}$ and $y = v_{3}$ then without loss of generality we have that $G_{v_{4}} \subseteq G[A \setminus \{x,y\}]$ and $G_{v_{5}} \subseteq G[B \setminus \{x,y\}]$. But this contradicts that there is a vertex in $G_{v_{4}}$ which is adjacent to a vertex in $G_{v_{5}}$. The same argument works if $x= v_{1}$, $y =v_{5}$ or if $x = v_{3}$ or $y = v_{4}$.  

Now assume that $x = v_{1}$ and $y = v_{4}$. Then without loss of generality we may assume that $G_{v_{3}} \subseteq G[A \setminus \{x,y\}]$ and $G_{v_{5}} \subseteq G[B \setminus \{x,y\}]$. Then $G_{v_{2}} \subseteq G[A \setminus \{x,y\}]$ as $G_{v_{2}}$ has a vertex which is adjacent to a vertex in $G_{v_{3}}$. By similar reasoning, $G_{v_{8}}, G_{v_{7}}$ and $G_{v_{6}}$ are all contained in $G[A \setminus \{x,y\}]$. But then there is no vertex in $G_{v_{6}}$ which is adjacent to a vertex in $G_{v_{5}}$, a contradiction. The case where $x = v_{3}$ and $y = v_{5}$ follows similarly. 

Lastly, assume that $x= v_{4}$ and $y = v_{5}$. Without loss of generality, we may assume that $G_{v_{3}}$ is contained in $G[B \setminus \{x,y\}]$ and $G_{v_{1}}$ is contained in $G[A \setminus \{x,y\}]$. Then $G_{v_{2}}$ is contained in either $G[A \setminus \{x,y\}]$ or $G[B \setminus \{x,y\}]$. Suppose that $G_{v_{2}}$ is contained in $G[A \setminus \{x,y\}]$. But then there is no vertex in $G_{v_{3}}$ which is adjacent to a vertex in $G_{v_{2}}$. Then $G_{v_{2}}$ is contained in $G[B \setminus \{x,y\}]$. But then there is no vertex in $G_{v_{1}}$ which is adjacent to a vertex in $G_{v_{2}}$, a contradiction. 
\end{proof}

\begin{lemma}
\label{subdivisionlabellingL(X)}
Let $G$ be a graph and $(A,B)$ be a $2$-separation with vertex boundary $\{u,v\}$. Let $X = \{a,b,c,d\} \subseteq V(G)$.  Suppose that there is exactly one vertex, $z$, such that $z \in X \cap (A \setminus \{u,v\})$, and exactly two vertices from $X$ in $B \setminus \{x,y\}$ and $u \in X$. Let $X_{1} = X \setminus \{z\} \cup \{v\}$. For each $\pi \in \mathcal{F}$, define $\pi': X_{1} \to V(L)$ such that $\pi' = \pi$ on $X \setminus \{z\}$ and $\pi'(v) = \pi(z)$. If there is a model of an $L(X)$-minor, $\{G_{x} | x \in L(X)\}$, then either $\{G_{B}[V(G_{x}) \cap B)]| x \in V(L)\}$ is a model of an $L(X_{1})$-minor in $G_{B}$, or the vertex from $X$ in $ A \setminus \{u,v\}$ is not in branch sets $G_{v_{4}}$ or $G_{v_{5}}$. 
\end{lemma}

\begin{proof}
Suppose $\{G_{x} | x \in L(X)\}$ is a model of an $L(X)$-minor, and suppose that $\{G_{B}[V(G_{x}) \cap B)] \ | \  x \in V(L)\}$ is not a model of an $L(X_{1})$-minor. Furthermore, suppose the vertex from $X$ in $A \setminus \{u,v\}$ is in $G_{v_{4}}$.

 First consider the case when $v \in G_{v_{4}}$. Suppose any of $G_{v_{2}}$, $G_{v_{7}}$, $G_{v_{6}}$, or $G_{v_{8}}$ 
is contained in $G[A \setminus \{u,v\}]$. Since $v_{2},v_{7},v_{6}$ and $v_{8}$ induce a diamond in $L$, and $v \in G_{v_{4}}$, and $u \in X$, each of $G_{v_{2}}$, $G_{v_{7}}$, $G_{v_{6}}$ and $G_{v_{8}}$ are contained in $G[A \setminus \{u,v\}]$. But then at least two of $G_{v_{1}}, G_{v_{3}}$, and $G_{v_{5}}$ are contained in $G[B \setminus \{u,v\}]$.  But this is a contradiction, since in $L$, all of $v_{1},v_{5}$ and $v_{4}$ are adjacent to at least one of $v_{2},v_{7}$ and $v_{6}$. Therefore we can assume that none of $G_{v_{2}}$, $G_{v_{7}}$, $G_{v_{6}}$, or $G_{v_{8}}$ are in $G[A \setminus \{u,v\}]$. But then since two vertices of $X$ lie in $B \setminus \{u,v\}$, there are at most two branch sets containing vertices from $A$. But then $\{G_{B}[V(G_{x}) \cap B)]| x \in V(L)\}$ is a model of an $L(X_{1})$-minor in $G_{B}$, a contradiction. 

Therefore we can assume that $v \not \in V(G_{v_{4}})$, and thus $G_{v_{4}} \subseteq G[A \setminus \{u,v\}]$. Now at least one of $G_{v_{3}}$ and $G_{v_{5}}$ contains a vertex from $B$ which is not $u$, and thus either $v \in G_{v_{3}}$ or $v \in G_{v_{5}}$. In either case, this implies that $G_{v_{7}}$ is contained in $G[A \setminus \{u,v\}]$. By the same reasoning as before, this implies that all of $G_{v_{2}}, G_{v_{8}}$ and $G_{v_{6}}$ are contained in $G[A \setminus \{u,v\}]$. But then since at least one of $G_{v_{1}}, G_{v_{3}}$ and $G_{v_{5}}$ are contained in $G[B \setminus \{u,v\}]$, which contradicts that $\{G_{x} | x \in L(X)\}$ is an $L(X)$-model. The case where the vertex from $X$ in $ A \setminus \{u,v\}$ is in $G_{v_{5}}$ follows similarly. 
 \end{proof}

\begin{lemma}
\label{twooneachsideL(X)}
Let $G$ be a $2$-connected graph and $(A,B)$ be a $2$-separation with vertex boundary $\{x,y\}$. Let $X = \{a,b,c,d\} \subseteq V(G)$.  Suppose  $a, b \in A \setminus (A \cap B)$ and $c,d \in B \setminus (B \cap A)$. Let $X_{1}= (X \cap A) \cup \{x,y\}$. For each $\pi \in \mathcal{F}$, define $\pi_{1}: X_{1} \to V(L)$ such that $\pi_{1} = \pi$ on $a,b$ and $\pi_{1}(x) = \pi(c)$ and $\pi_{1}(d) = \pi$. Let $X_{2} = (X \cap B) \cup \{x,y\}$. For each $\pi \in \mathcal{F}$, define $\pi_{2}: X \to V(L)$ such that $\pi_{2} = \pi$ on $\{c,d\}$ and $\pi_{2}(x) = \pi(a)$ and $\pi_{2}(y) = \pi(b)$.  Then $G$ has an $L(X)$-minor if and only if $G_{A} = G[A] \cup \{xy\}$ has an $L(X_{1})$-minor or $G_{B} = G[B] \cup \{xy\}$ has an $L(X_{2})$-minor. 
\end{lemma}

\begin{proof}
If $G_{A}$ has an $L(X_{1})$-minor, then we can contract $B$ onto $\{x,y\}$ such that the vertices of $X$ do not get contracted together. This is possible since $G$ is $2$-connected. Then $G$ has an $L(X)$-minor. A similar argument works when $G_{B}$ has an $L(X_{2})$-minor. Now assume that $\{G_{x} | x \in V(L)\}$ is a model of an $L(X)$-minor in $G$.

Suppose $x$ is in one of $G_{v_{2}}$, $G_{v_{7}}, G_{v_{6}},$ or $G_{v_{8}}$.  We consider cases based on which branch sets contain the vertices of $X$.

First, suppose that the vertices from $X$ in $A \setminus \{x,y\}$ are in branch sets $G_{v_{1}}$ and $G_{v_{3}}$. Then the vertices in $X$ in $B \setminus \{x,y\}$ are in branch sets $G_{v_{5}}$ and $G_{v_{4}}$. But then since at most one branch set contains $y$, either there is no vertex in $G_{v_{3}}$ adjacent to a vertex in $G_{v_{4}}$ or there is no vertex in $G_{v_{1}}$ which is adjacent to a vertex in $G_{v_{5}}$.

Now suppose  the vertices from $X$ in $A \setminus \{x,y\}$ are in branch sets $G_{v_{1}}$ and $G_{v_{4}}$. Then the vertices in $X$ in $B \setminus \{x,y\}$ are in branch sets $G_{v_{5}}$ and $G_{v_{3}}$. But then since at most one branch set contains $y$, either there is no vertex in $G_{v_{3}}$ adjacent to a vertex in $G_{v_{4}}$ or there is no vertex in $G_{v_{1}}$ which is adjacent to a vertex in $G_{v_{5}}$. In either case, this is a contradiction. 

Now suppose the vertices from $X$ in $A \setminus \{x,y\}$ are in branch sets $G_{v_{1}}$ and $G_{v_{5}}$. Then the vertices from $X$ in $B \setminus \{x,y\}$ are in branch sets $G_{v_{3}}$ and $G_{v_{4}}$. First suppose that $x \not \in V(G_{v_{2}})$. Then since $v_{1}$ and $v_{3}$ are adjacent to $v_{2}$ in $L$, we have that $y \in V(G_{v_{2}})$. But then there is no vertex in $G_{v_{4}}$ which is adjacent to a vertex in $G_{v_{5}}$, a contradiction. Therefore $x \in V(G_{v_{2}})$. Then $y \in V(G_{v_{4}})$ or $V(G_{v_{5}})$. Then $G_{v_{8}},G_{v_{7}}$ and $G_{v_{6}}$ all have vertex sets in either $B \setminus \{x,y\}$ or $A \setminus \{x,y\}$. But then either $G_{v_{6}}$ does not have a vertex adjacent to a vertex in $G_{v_{5}}$ or $G_{v_{4}}$ does not have a vertex adjacent to a vertex in $G_{v_{7}}$. In either case, this is a contradiction.

Therefore $x$ is not in $V(G_{v_{2}})$, $V(G_{v_{6}}), V(G_{v_{7}}),$ or $V(G_{v_{8}})$. Similarly, $y$ is not in $V(G_{v_{2}})$, $V(G_{v_{6}}), V(G_{v_{7}}),$ or $V(G_{v_{8}})$.  Then $x$ and $y$ belong to two of $G_{v_{1}},G_{v_{3}},G_{v_{4}}$ and $G_{v_{5}}$. Now suppose that $G_{v_{1}}$ is contained in $G[A \setminus \{x,y\}]$ and $G_{v_{3}}$ is contained in $G[B \setminus \{x,y\}]$. Then since $x,y$ belong to two of $G_{v_{1}},G_{v_{3}},G_{v_{4}}$ and $G_{v_{5}}$, we have that $G_{v_{2}}$ is contained in one of $G[B \setminus \{x,y\}]$ or $G[A \setminus \{x,y\}]$. But then without loss of generality there is no vertex in $G_{v_{2}}$ which is adjacent to a vertex in $G_{v_{1}}$, a contradiction. A similar analysis shows that for any two of $G_{v_{1}},G_{v_{3}},G_{v_{4}}$ and $G_{v_{5}}$, if one of the branch sets is contained in $G[B \setminus \{x,y\}]$ and the other in $G[A \setminus \{x,y\}]$ we get a contradiction. Therefore the two branch sets from $G_{v_{1}},G_{v_{3}},G_{v_{4}}$ and $G_{v_{5}}$ which do not contain $x,y$ are contained on the same side of the $2$-separation. Suppose the two branch sets from $G_{v_{1}},G_{v_{3}},G_{v_{4}}$ and $G_{v_{5}}$ which do not contain $x,y$ are contained in $G[A \setminus \{x,y\}]$. Then by the same reasoning as above, all of $G_{v_{2}},G_{v_{8}},G_{v_{6}},$ and $G_{v_{7}}$ are contained in $G[A \setminus \{x,y\}]$. But then $\{G_{A}[V(G_{x}) \cap A] | x \in V(L)\}$ is a model of an $L(X_{1})$-minor in $G_{A}$. In the other case, by the same argument we get an $L(X_{2})$-minor in $G_{B}$.
\end{proof}

Now we show that class $\mathcal{A}$ graphs do not have an $L(X)$-minors.

\begin{lemma}
\label{noclassAgraphsL(X)}
Let $G$ be $2$-connected spanning subgraph of a class $\mathcal{A}$ graph. Then $G$ does not have an $L(X)$-minor. 
\end{lemma}

\begin{proof}
Suppose towards a contradiction that we have a model of an $L(X)$-minor in $G$, $\{G_{x} | x \in V(L)\}$. Then observe that $\{e,d\}$ is the vertex boundary of a separation $(A,B)$ such that $a \in A \setminus \{d,e\}$ and $b,c \in B \setminus \{d,e\}$ so the hypotheses of Lemma \ref{subdivisionlabellingL(X)} are satisfied. Then we consider two cases. If $\{G_{B}[V(G_{x} \cap B)]| x \in V(L)\}$ is a model of an $L(X_{1})$-minor, then notice that this means that a spanning subgraph of a graph $H^{+}$, where $H$ is the graph defined in Lemma \ref{smallgraphk22} has an $L(X_{1})$-minor. But in this graph $H^{+}$, $\{d,e\}$ is the vertex boundary of a $2$-separation satisfying Lemma \ref{terminalseparatingL(X)}, and thus $H^{+}$ does not have an $L(X_{1})$-minor, a contradiction. 

Therefore we must be in the case where $a \in V(G_{v_{1}})$ or $a \in V(G_{v_{3}})$. Notice that $\{e,d\}$  induces a $2$-separation in $G$, $(A',B')$ such that $c \in A' \setminus \{e,d\}$ and $a,b \in B' \setminus \{e,d\}$. Additionally, $\{e,d\}$ is the vertex boundary of a $2$-separation $(A'',B'')$ such that $b \in A'' \setminus \{e,d\}$, and $a,c \in B'' \setminus \{e,d\}$. Therefore by applying the exact same analysis as the the $(A,B)$ separation, we get that all of $a,b,$ and $c$ are in $G_{v_{1}}$ or $G_{v_{3}}$, which is a contradiction.  
\end{proof}

Now we reduce class $\mathcal{E}$ and $\mathcal{F}$ down to looking at webs. 

\begin{lemma}
Let $G$ be a $2$-connected spanning subgraph of a class $\mathcal{E}$ or a $\mathcal{F}$ graph. If $G$ is a spanning subgraph of a class $\mathcal{E}$ graph, then $G$ has an $L(X)$-minor if and only if the $\{e,f,c,d\}$-web has an $L(X)$-minor. If $G$ is a spanning subgraph of a class $\mathcal{F}$ graph, then $G$ has an $L(X)$-minor if and only if the $\{e,f,g,h\}$-web has an $L(X)$-minor.
\end{lemma}

\begin{proof}
First suppose that $G$ is a $2$-connected spanning subgraph of a class $\mathcal{E}$ graph. Then apply Lemma \ref{twooneachsideL(X)} to the separation whose vertex boundary is $\{e,f\}$. Notice that one of the graphs we get from Lemma \ref{twooneachsideL(X)} is the graph $H^{+}$ where $H$ is the graph from Lemma \ref{smallgraphk22}. Notice that $H^{+}$ does not have an $L(X)$-minor since $\{e,f\}$ forms a separation satisfying Lemma \ref{terminalseparatingL(X)}. Notice that the other graph we obtain from Lemma \ref{twooneachsideL(X)} is the $\{e,f,c,d\}$-web, which completes the claim. The argument for the class $\mathcal{F}$ graphs is essentially the same.
\end{proof}

Therefore it suffices to look at graphs which are spanning subgraphs of webs satisfying the obstructions given in Theorem \ref{w4cuts}. We will look at each case separately, but it turns out that essentially all the obstructions reduce to instances of terminal separating $2$-chains having $L(X)$-minors. First we look at the terminal separating $2$-chain obstruction. 

\begin{figure}
\begin{center}
\includegraphics[scale =0.5]{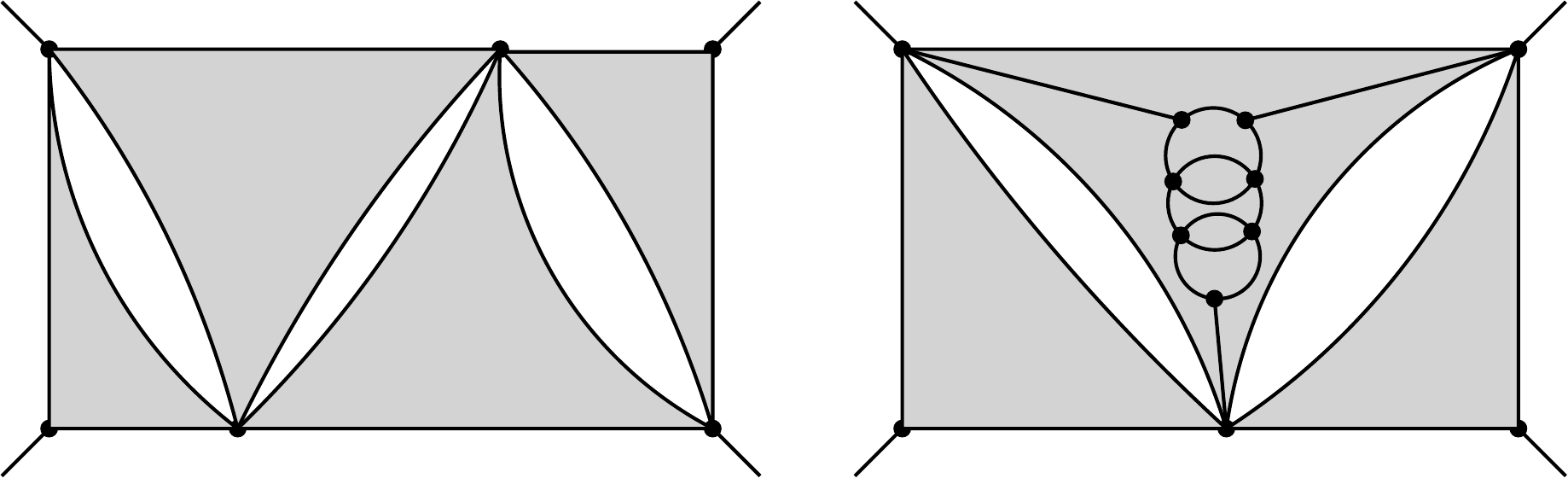}
\caption{Graph on the left has no $L(X)$-minor by Lemma \ref{terminalseparting2chainsL(X)} and the graph on the right has an $L(X)$-minor and satisfies the conditions of Lemma \ref{terminalseparting2chainsL(X)}. Curved lines indicate $2$-separations.}
\end{center}
\end{figure}

\begin{lemma}
\label{terminalseparting2chainsL(X)}
Let $G$ be a $2$-connected spanning subgraph of an $\{a,b,c,d\}$-web. Let $X =\{a,b,c,d\} \subseteq V(G)$ and suppose that $G$ does not have a $W_{4}(X)$-minor. Consider any cycle $C$ for which $X \subseteq V(C)$. Suppose $a,b,c,d$ appears in that order on $C$. Suppose there is a terminal separating $2$-chain $((A_{1},B_{1}),\ldots, (A_{n},B_{n}))$ satisfying obstruction $1$ from Theorem \ref{w4cuts}. Furthermore, suppose that $A_{1} \cap B_{1} \cap X = \{a\}$, and $b = A_{1} \cap X \setminus (A_{1} \cap B_{1})$. If $n \equiv 1 \pmod 2$, then $G$ has no $L(X)$-minor. Now consider when $n \equiv 0 \pmod 2$. Then $c \in A_{n} \cap X \setminus B_{n}$. Then $G$ has an $L(X)$-minor if and only if the following occurs. There exists an $i \in \{1,\ldots,n-1\}$, $i \equiv 1 \pmod 2$, such that the graph $G_{B_{i} \cap A_{i+1}} =  G[B_{i} \cap A_{i+1}] \cup \{xy | x,y \in A_{i} \cap B_{i}\} \cup \{xy | x,y \in A_{i+1} \cap B_{i+1}\}$ has an $L'(X')$-minor, where $X' = \{(A_{1} \cap B_{1}), (A_{2} \cap B_{2})\}$. Furthermore, the vertex from $X'$ which is contained in the branch set $G_{v_{2}}$ lies on the $(b,c)$-path $P_{b,c}$ where $a,d \not \in V(P_{b,c})$ and $V(P_{b,c}) \subseteq V(C)$.
\end{lemma}

\begin{proof}
First suppose that $n \equiv 1 \pmod 2$. If $n=1$, then $(A_{1},B_{1})$ satisfies Lemma \ref{terminalseparatingL(X)} and thus $G$ does not have an $L(X)$-minor. Now suppose $n \geq 3$ and $G$ is a vertex minimal counterexample to the claim. Then apply Lemma \ref{twooneachsideL(X)} to $(A_{2},B_{2})$ gives two graphs $G_{A_{2}}$ and $G_{B_{2}}$ such that $G$ has an $L(X)$-minor if and only if $G_{A_{2}}$ has an $L(X_{1})$-minor or $G_{B_{2}}$ has an $L(X_{2})$-minor.   Now both $G_{A_{2}}$ and $G_{B_{2}}$ have terminal separating $2$-chains of odd length. Namely, $(A_{1},B_{1})$ is a terminal separating $2$-chain in $G_{A_{2}}$ and $(A_{3},B_{3}),\ldots,(A_{n},B_{n})$ is a terminal separating $2$-chain in $G_{B_{2}}$ which has length $n-2$, which is odd.  Then since $G$ is a vertex minimal counterexample, both $G_{A_{2}}$ and $G_{B_{2}}$ do not have $L(X)$-minors and therefore $G$ does not have an $L(X)$-minor. 

Now suppose that $n \equiv 0 \pmod 2$. First suppose that $n \geq 4$ and suppose that $G$ is a vertex minimal counterexample. Then apply Lemma \ref{twooneachsideL(X)} to $(A_{2},B_{2})$. Then we get two graphs $G_{A_{2}}$ and $G_{B_{2}}$ such that $G$ has an $L(X)$-minor if and only if either $G_{A_{2}}$ has an $L(X_{1})$-minor or $G_{B_{2}}$ has an $L(X_{2})$-minor. We note that $G_{A_{2}}$ does not have an $L(X_{1})$-minor since $(A_{1},B_{1})$ satisfies Lemma \ref{terminalseparatingL(X)}. Therefore since $G$ has an $L(X)$-minor, $G_{B_{2}}$ has an $L(X_{2})$-minor. Furthermore, $((A_{3},B_{3}),\ldots,(A_{n},B_{n}))$ is an even length terminal separating $2$-chain in $G_{B_{2}}$. Then as $G$ is a vertex minimal counterexample, $G_{B_{2}}$ satisfies the claim. But then notice that the vertex in $X_{2}$ which takes the place of $b$ in the claim lies on $P_{b,c}$ since $b \in A_{1} \setminus (A_{1} \cap B_{1})$ implies that the vertex in $(A_{1} \cap B_{1}) \setminus \{a\}$ lies on $P_{b,c}$. Therefore $G$ satisfies the claim, a contradiction.

Therefore we can assume that $n =2$, and suppose $\{G_{x} \ | \ x \in V(L)\}$ is a model of an $L(X)$-minor in $G$. Apply Lemma \ref{subdivisionlabellingL(X)} to $(A_{1},B_{1})$. Then we have two cases. First suppose we have an $L(X_{1})$-minor in $G_{B_{1}}$ where $X_{1} = X \setminus \{b\} \cup (A_{1} \cap B_{1})$. Then in $G_{B_{1}}$ under $X_{1}$, the $2$-separation $(A_{2},B_{2})$ satisfies Lemma \ref{terminalseparatingL(X)}, implying that $G_{B_{1}}$ does not have an $L(X_{1})$-minor, a contradiction. Therefore $b \in V(G_{v_{1}})$ or $b \in V(G_{v_{3}})$. First suppose that $b \in V(G_{v_{1}})$. Notice that $(A_{2},B_{2})$ satisfies Lemma \ref{subdivisionlabellingL(X)}, and that by symmetry we may assume that $c \in V(G_{v_{1}})$ or $c \in V(G_{v_{3}})$. As we assumed $b \in V(G_{v_{1}})$, we get that $c \in V(G_{v_{3}})$. Suppose that $G_{v_{1}}$ is not contained in $G[A_{1} \setminus B_{1}]$. Then the vertex in $A_{1} \cap B_{1} \setminus \{a\}$ is in $G_{v_{1}}$. Then $G_{v_{3}}$ is contained in $G[B_{2}]$. Then since $G_{v_{3}}$ has a vertex adjacent to a vertex in $G_{v_{2}}$, we have that $G_{v_{2}}$ is contained in $G[B_{1}]$. Then this implies that $G_{v_{8}}, G_{v_{6}},$ and $G_{v_{7}}$ are contained in $G[B_{2} \setminus A_{2}]$. But $a \in V(G_{v_{5}})$ or $a \in V(G_{v_{4}})$, so either $G_{v_{4}}$ is contained in $G[A_{2} \setminus B_{2}]$ or $G_{v_{5}}$ is contained in $G[A_{2} \setminus B_{2}]$. But that contradicts that $\{G_{x} \ | \ x \in L(X)\}$ is an $L(X)$-minor. Therefore $G_{v_{1}}$ is contained in $G[A_{1} \setminus (A_{1} \cap B_{1})]$. By symmetry, $G_{v_{3}}$ is contained in $G[B_{2} \setminus (A_{2} \cap B_{2})]$. That implies that $G_{v_{2}}$ contains the vertex in $A_{1} \cap B_{1} \setminus \{a\}$. Suppose that $G_{v_{8}}$ is contained in $G[A_{1} \setminus (A_{1} \cap B_{1})]$. Then since $G_{v_{2}}$ contains the vertex in $A_{1} \cap B_{1} \setminus \{a\}$, the branch sets $G_{v_{7}}$ and $G_{v_{6}}$ are contained in $G[A_{1} \setminus (A_{1} \cap B_{1})]$. But then since $d \in V(G_{v_{5}})$ or $d \in V(G_{v_{4}})$, this implies that either $G_{v_{5}}$ is contained in $G[B_{1} \setminus A_{1}]$ or $G_{v_{4}}$ is contained in $G[B_{1} \setminus A_{1}]$. But this contradicts that $\{G_{x} \ | \ x \in V(L)\}$ is a model of an $L(X)$-minor. By essentially the same argument, none of $G_{v_{8}}, G_{v_{7}}$ or $G_{v_{6}}$ are contained in $G[B_{2} \setminus A_{2} ]$. Therefore all of $G_{v_{8}}, G_{v_{7}}$ and $G_{v_{6}}$ are contained in $G[B_{1} \cap A_{2}]$. But then the set $\{G_{B_{1} \cap A_{2}}[G_{x} \cap B_{1} \cap A_{2}] \ | \ x \in V(L')\}$ is a model of $L'(X')$ minor satisfying the properties of the lemma.  

Now we prove the converse.  Suppose there exists an $i \in \{1,\ldots,n-1\}$ and $i \equiv 1 \pmod 2$  such that the graph $G_{B_{i} \cap A_{i+1}}$ has an $L'(X')$-minor satisfying the properties in the lemma. Let $\{G_{x} | x \in V(L')\}$ be a model of an $L'(X')$-minor satisfying the above properties. Let $v$ be the vertex in $X'$ which is also in $G_{v_{2}}$ in the $L'(X)$ model. Observe that since $i \equiv 1 \pmod 2$, there is exactly one vertex in $X'$ which lies on $P_{b,c}$, so therefore $v \in P_{b,c}$ and both other vertices of $X'$ lie on $P_{a,d}$, the $(a,d)$-path such that $V(P_{a,d}) \subseteq V(C)$ and $b,c \not \in V(P_{a,d})$. Then to get an $L(X)$-minor in $G$, first let $x$ be the vertex in $X'$ such that the $(a,x)$-subpath on $P_{a,d}$ does not contain the other vertex from $X'$ on $P_{a,d}$, and without loss of generality, let $x \in G_{v_{5}}$. Then extend $G_{v_{5}}$ to include the subpath from $(a,x)$ on $P_{a,d}$. Now let $y$ be the vertex in $X$ such that the $(y,d)$-subpath on $P_{a,d}$ does not contain $x$. Since $x \in G_{v_{5}}$, we have $y \in G_{v_{4}}$. Extend $G_{v_{4}}$ along the $(y,d)$-subpath on $P_{a,d}$. Now we create $G_{v_{1}}$  by letting it be the $(a,v)$-path, $P_{a,v}$ such that $V(P_{a,v}) \subseteq V(C)$ and $d,c \not \in P_{a,v}$ and do not include either $a$ or $v$. Note that $b \in V(P_{a,v})$. Similarly, let $G_{v_{3}}$ be the $(d,v)$-path, $P_{d,v}$ such that $V(P_{d,v}) \subseteq V(C)$ and $a,b \not \in P_{d,v}$, not including either $d$ or $v$.
 we $G_{v_{1}}$ contain the path from $v$ to $b$ on $P_{b,c}$, not including $v$, we let $G_{v_{3}}$ contain the path from $v$ to $c$ on $P_{b,c}$ not including $v$. Then by construction and since we already had an $L'(X')$-minor, we have an $L(X)$-minor in $G$.
\end{proof}

Now we look at when our graph has a terminal separating triangle as in obstruction $2$.

\begin{lemma}
\label{trianglereductionL(X)}
Let $G$ be a $2$-connected spanning subgraph of an $\{a,b,c,d\}$-web. Let $X =\{a,b,c,d\} \subseteq V(G)$ and suppose that $G$ does not have a $W_{4}(X)$-minor. Consider any cycle $C$ for which $X \subseteq V(C)$, and suppose that there is a terminal separating triangle $(A_{1},B_{1}),(A_{2},B_{2}),(A_{3},B_{3})$ satisfying obstruction $2$ of Theorem \ref{w4cuts}. Then $G$ has an $L(X)$-minor if and only if either the graph $G_{B_{1}}$ has an $L(X_{1})$-minor where $X_{1} = (X \cap B_{1}) \cup (A_{1} \cap B_{1})$ or the graph $G_{B_{3}}$ has an $L(X_{2})$-minor  where $X_{2} = (X \cap B_{3}) \cup (A_{3} \cap B_{3})$. 
\end{lemma}

\begin{figure}
\begin{center}
\includegraphics[scale =0.5]{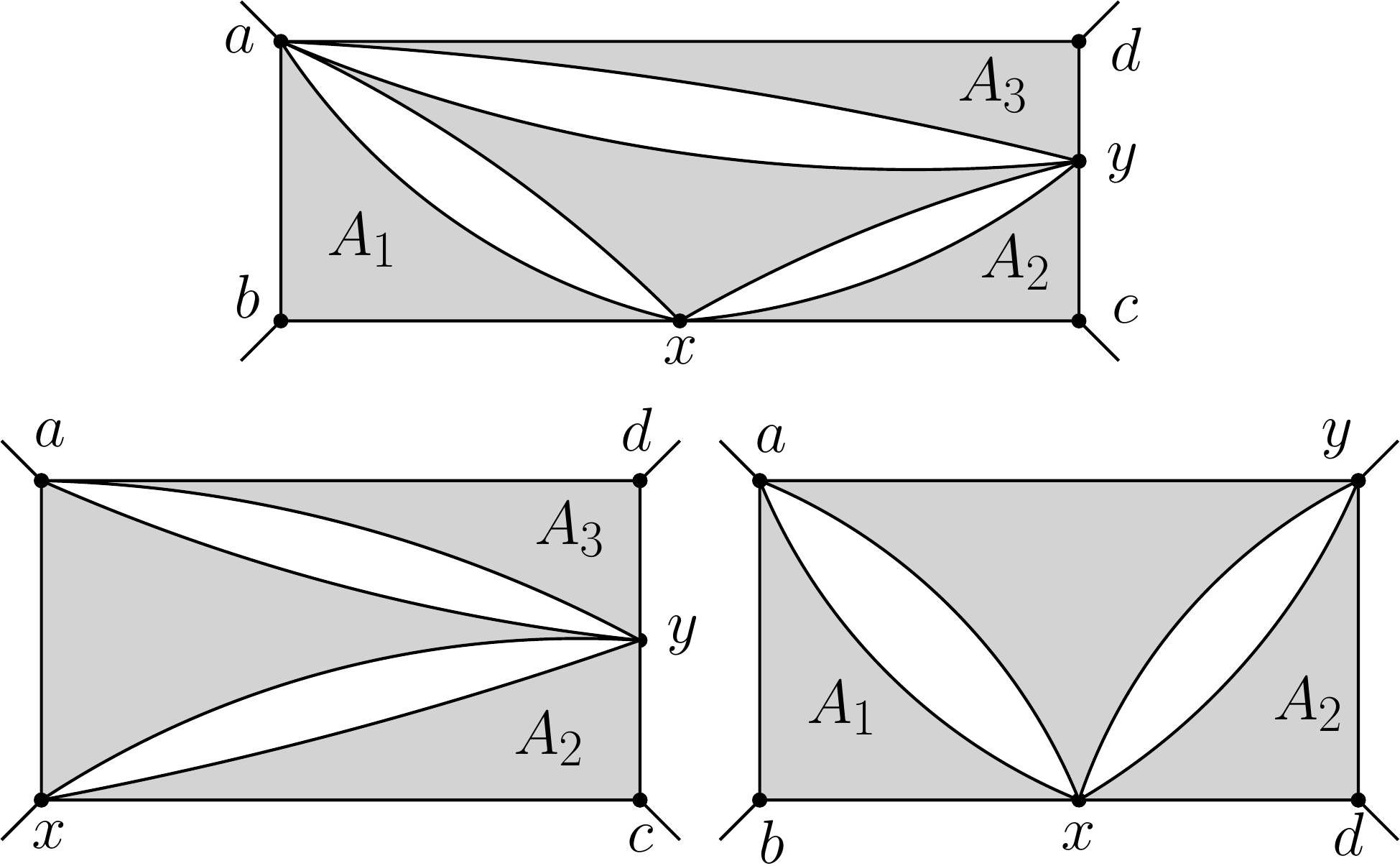}
\caption{A terminal separating triangle and the graphs $G_{B_{1}}$ and $G_{B_{3}}$ as in Lemma \ref{trianglereductionL(X)}. Note both $G_{B_{1}}$ and $G_{B_{3}}$ have terminal separating $2$-chain obstructions.}
\end{center}
\end{figure}

\begin{proof}
If $G_{B_{1}}$ has an $L(X_{1})$-minor then since $G$ is $2$-connected we can contract the vertex in $X$ in $A_{1}$ to $A_{1} \cap B_{1}$ such that it does not get contracted together with another vertex of $X$. But then $G$ has an $L(X)$-minor. A similar argument holds for $G_{B_{3}}$.

Now suppose that $\{G_{x} | x \in V(L)\}$ is a model of an $L(X)$-minor in $G$, and furthermore suppose we have a vertex minimal counterexample to the claim. Without loss of generality let $a,b,c,d$ appear in that order on $C$, and without loss of generality suppose that $a$ is the vertex from $X$ in $A_{1} \cap B_{1}$, and that $b \in  A_{1} \setminus B_{1}$. Similarly, without loss of generality we may assume that $c \in A_{2} \setminus  B_{2}$ and $d \in A_{3} \setminus B_{3}$. Then notice that $(A_{1},B_{1})$ satisfies Lemma \ref{subdivisionlabellingL(X)}. We consider two cases. 

First consider the case where $b \in G_{v_{1}}$ or $b \in G_{v_{3}}$. We assume that $b \in G_{v_{1}}$. Now notice that since we assumed $d \in A_{3}$, $(A_{3},B_{3})$ is a $2$-separation satisfying Lemma \ref{subdivisionlabellingL(X)}. First suppose that $d \in G_{v_{3}}$ (note this is without loss of generality, since $b \in G_{v_{1}}$). First consider the case where $G_{v_{1}}$ is contained in $G[A_{1} \setminus B_{1}]$. Then $G_{v_{2}}$ contains the vertex in $(A_{1} \cap B_{1})$ which is not $a$. Similarly, if $G_{v_{3}}$ is contained in $G[A_{3} \setminus B_{3}]$, then $G_{v_{2}}$ contains the vertex in $A_{3} \cap B_{3} \setminus \{a\}$. But then $a \in B_{2} \setminus A_{2}$ and $c \in A_{2} \setminus B_{2}$, which implies that there is no vertex in $G_{v_{5}}$ which is adjacent to a vertex in $G_{v_{4}}$, a contradiction. Therefore $G_{v_{3}}$ is not contained in $G[A_{3} \setminus  B_{3}]$ and the vertex in $A_{3} \cap B_{3} \setminus \{a\}$ is in $G_{v_{3}}$. But then as before, $a \in B_{2} \setminus A_{2}$ and $c \in A_{2} \setminus  B_{2}$, contradicting that there is a vertex in $G_{v_{5}}$ which is adjacent to a vertex in $G_{v_{4}}$. Therefore we can assume that $d \not \in V(G_{v_{3}})$. But then by Lemma \ref{subdivisionlabellingL(X)}, $G_{B_{3}}$ has an $L(X_{2})$-minor, which satisfies the claim. 

Therefore we can assume that $G_{v_{1}}$ is not contained in $G[A_{1} \setminus  B_{1}]$. Since the cases are symmetric, the only situation left to consider is when  $d \in G_{v_{3}}$ and not contained in $G[A_{3} \setminus  B_{3}]$. Then the vertex in $(A_{1} \cap B_{1}) \setminus \{a\}$ and the vertex in $(A_{3} \cap B_{3}) \setminus \{a\}$ are in $G_{v_{1}}$ and $G_{v_{3}}$ respectively. But then there is no vertex in $G_{v_{4}}$ which is adjacent to $G_{v_{5}}$, a contradiction. 
The same situation happens when $b \in G_{v_{3}}$. Therefore we assume that $b \not \in G_{v_{3}}$ and $b \not \in G_{v_{1}}$. But then by Lemma \ref{subdivisionlabellingL(X)}, $G_{B_{1}}$ has an $L(X_{1})$-minor, which completes the claim.
\end{proof}

Observe that both the graphs $G_{B_{1}}$ and $G_{B_{2}}$ have terminal separating $2$-chains as obstructions, and thus when we have a terminal separating triangle as an obstruction, the problem reduces to appealing to Lemma \ref{terminalseparting2chainsL(X)}. Now we look at what happens when the graph has a terminal separating $2$-chain and a terminal separating triangle as in obstruction $3$. 

\begin{lemma}
\label{terminalseparatingchainplustriangleL(X)}
Let $G$ be a $2$-connected spanning subgraph of an $\{a,b,c,d\}$-web. Let $X = \{a,b,c,d\} \subseteq V(G)$ and suppose that $G$ does not have a $W_{4}(X)$-minor. Consider any cycle $C$ for which $X \subseteq V(C)$, and suppose there is a terminal separating triangle $(A_{1},B_{1}),(A_{2},B_{2}), (A_{3},B_{3})$ such that $A_{1} \cap B_{1} = \{x,y\}$, where $x,y \not \in X$. Furthermore, the graph $G_{A_{1}} = G[A_{1}] \cup \{xy\}$ and $C_{A} = G[V(C) \cap A] \cup \{x,y\}$  has a terminal separating $2$-chain $(A'_{1},B'_{1}),\ldots,(A'_{n},B'_{n})$ where we let $x$ and $y$ replace the two vertices in $X$ from $G$ not in $G_{A_{1}}$ and $A_{i} \cap B_{i} \subseteq V(C)$, $i \in \{1,2,3\}$, and $A'_{i} \cap B'_{i} \subseteq V(C_{A})$ for all $i \in \{1,\ldots,n\}$. Then $G$ has an $L(X)$-minor if and only if the graph $G_{B_{1}} = G[B_{1}] \cup \{xy | x,y \in A_{1} \cap B_{1}\}$ has an $L(X_{1})$-minor where $L(X_{1}) = X \cap B_{1} \cup (A_{1} \cap B_{1})$ or $G_{A_{1}} =  G[B_{1}] \cup \{xy | x,y \in A_{1} \cap B_{1}\}$ has an $L(X_{2})$-minor where $X_{2} = X \cap A_{1} \cup (A_{1} \cap B_{1})$.
\end{lemma}

\begin{proof}
This follows immediately from applying Lemma \ref{twooneachsideL(X)} to $(A_{1},B_{1})$.
\end{proof}

The point of the above observation is that the graphs  $G_{A_{1}}$ and $G_{B_{1}}$ have terminal separating $2$-chain obstructions, and thus the problem of finding $L(X)$-minors in that case reduces to the problem of finding $L(X)$-minors in terminal separating  $2$-chains case, which is done in Lemma \ref{terminalseparatingL(X)}. Now we look at what happens when we have two terminal separting triangles and a terminal separating $2$-chain as in

\begin{lemma}
Let $G$ be a $2$-connected spanning subgraph of an $\{a,b,c,d\}$-web. Let $X = \{a,b,c,d\} \subseteq V(G)$ and suppose that $G$ does not have a $W_{4}(X)$-minor. Consider any cycle $C$ for which $X \subseteq V(C)$, and suppose there are two distinct terminal separating triangles $((A^{1}_{1},B^{1}_{1}), (A^{1}_{2},B^{1}_{2}), (A^{1}_{3},B^{1}_{3}))$, $((A^{2}_{1},B^{2}_{1}),(A^{2}_{2},B^{2}_{2}),(A^{2}_{3},B^{2}_{3}))$  where for all $i \in \{1,2,3\}$,  $A^{1}_{i} \cap B^{1}_{i} \subseteq A^{2}_{1}$  and $A^{2}_{i} \cap B^{2}_{i} \subseteq A^{1}_{3}$. Furthermore, if we consider the graph $G[A^{2}_{1} \cap A^{1}_{1}]$ and the cycle $C' = G[V(C) \cap A^{2}_{1} \cap A^{1}_{1}] \cup \{xy | x,y \in A^{i}_{1} \cap B^{i}_{1}, i \in \{1,2\}\}$ and we let $X'$ be defined to be the vertices $A^{2}_{1} \cap B^{2}_{1}$ and $A^{2}_{3} \cap B^{2}_{3}$, then there is a terminal separating $2$-chain in $G[A^{2}_{3} \cap A^{1}_{1}]$ with respect to $X'$. Then $G$ has an $L(X)$-minor if and only if either $G_{B^{1}_{1}}$ has an $L(X_{1})$-minor where $X_{1} = X \cap B^{1}_{1} \cup (A^{1}_{1} \cap B^{1}_{1})$ or $G_{B^{2}_{1}}$ has an $L(X_{2})$-minor where $X_{2} = X \cap B^{2}_{1} \cup (A^{1}_{1} \cap B^{1}_{1})$ or $G[A^{2}_{3} \cap A^{1}_{1}]$ has an $L(X')$-minor. 
\end{lemma}

\begin{proof}
Apply Lemma \ref{twooneachsideL(X)} to $(A^{1}_{1},B^{1}_{1})$. Then $G$ has an $L(X)$-minor if and only if $G_{B^{1}_{1}}$ has an $L(X_{1})$ minor where $X_{1} = X \cap B^{1}_{1} \cup (A^{1}_{1} \cap B^{1}_{1})$ or $G_{A^{1}_{1}}$ has an $L(X'_{1})$-minor where $X'_{1} = X \cap A^{1}_{1} \cup (A^{1}_{1} \cap B^{1}_{1})$. Notice that in $G_{A^{1}_{1}}$, we have a terminal separating $2$-chain and terminal separating triangle as in obstruction $3$. But then we can apply Lemma \ref{terminalseparatingchainplustriangleL(X)} which gives us exactly the claim.  
\end{proof}

\begin{lemma}
\label{2disjointtrianglesL(X)}
Let $G$ be a $2$-connected spanning subgraph of an $\{a,b,c,d\}$-web. Let $X = \{a,b,c,d\} \subseteq V(G)$ and suppose that $G$ does not have a $W_{4}(X)$-minor. Consider any cycle $C$ for which $X \subseteq V(C)$, and suppose that there two distinct terminal separating triangles $((A^{1}_{1},B^{1}_{1}), (A^{1}_{2},B^{1}_{2}), (A^{1}_{3},B^{1}_{3})), ((A^{2}_{1},B^{2}_{1}), (A^{2}_{2},B^{2}_{2}),(A^{2}_{3},B^{2}_{3}))$ which satisfy obstruction $5$ in Theorem \ref{w4cuts}. Then $G$ has an $L(X)$-minor if and only if either the graph $G_{B^{1}_{1}} = G[B^{1}_{1}] \cup \{xy | x,y \in A^{1}_{1} \cap B^{1}_{1}\}$ has an $L(X_{1})$-minor where $X_{1} = (X \cap B^{1}_{1}) \cup (A^{1}_{1} \cap B^{1}_{1})$ or the graph $G_{B^{2}_{1}}= G[B^{2}_{1}] \cup \{xy | x,y \in A^{2}_{1} \cap B^{2}_{1}\}$ has an $L(X_{2})$-minor where $X_{2} = (X \cap B^{2}_{1}) \cup (A^{2}_{1} \cap B^{2}_{1})$.  
\end{lemma}

\begin{proof}
If $G_{B^{1}_{1}}$ has an $L(X_{1})$-minor, then we simply contract all of $A^{1}_{1}$ onto $B^{1}_{1} \cap A^{1}_{1}$ such that we do not contract two vertices of $X$ together. This is possible as $G$ is $2$-connected. The same strategy applies to $G_{B^{2}_{1}}$.

Now suppose that $G$ has an $L(X)$-minor. Then $(A^{1}_{1},B^{1}_{1})$ satisfies Lemma \ref{twooneachsideL(X)}. Therefore $G$ has an $L(X)$-minor if and only if $G_{B^{1}_{1}}$ has an $L(X_{1})$-minor where $X_{1} = X \cap B^{1}_{1} \cup (B^{1}_{1} \cap A^{1}_{1})$  or $G_{A^{1}_{1}}$ has an $L(X'_{1})$-minor where $X'_{1} = A^{1}_{1} \cap X \cup (A^{1}_{1} \cap B^{1}_{1})$. If $(A^{1}_{1},B^{1}_{1}) = (A^{2}_{1},B^{2}_{1})$, then the claim follows immediately. Thus we assume that $(A^{1}_{1},B^{1}_{1}) \neq (A^{2}_{1},B^{2}_{1})$. Notice that in $G_{A^{1}_{1}}$, the triangle  $(A^{2}_{1},B^{2}_{1}), (A^{2}_{2},B^{2}_{2}),(A^{2}_{3},B^{2}_{3})$ is a terminal separating triangle satisfying obstruction $2$. Then by Lemma \ref{trianglereductionL(X)}, $G_{A^{1}_{1}}$ has an $L(X'_{1})$-minor if and only if the graph $G_{B^{2}_{1}}$ has an $L(X_{2})$-minor, which completes the claim.
\end{proof}

With that, one can determine exactly which graphs do not have an $K_{4}(X)$-minor, $W_{4}(X)$-minor, $K_{2,4}(X)$-minor and an $L(X)$-minor. To summarize, a $2$-connected graph does not have one of the four above minors if and only if $G$ is a class $\mathcal{A}$ graph, or it is the spanning subgraph of a class $\mathcal{D}$, $\mathcal{E}$, or $\mathcal{F}$ graph where the corresponding web has one of the obstructions from Theorem \ref{w4cuts}, and then after the reductions given above, we do not end up with an $L'(X')$-minor (Lemma \ref{L'(X')minors}) satisfying the properties in Lemma \ref{terminalseparatingL(X)}. It would be nice to obtain a cleaner structure theorem for when a graph has an $L'(X')$-minor. Additionally, if one desired, they could attempt to give a spanning subgraph characterization of $1$-connected graphs by working through the cut vertex reductions.

\bibliography{ThesisBib}
  \bibliographystyle{plain}

\end{document}